\documentclass[10pt]{article}

\usepackage{latexsym,amsfonts,amsmath,amsthm,amssymb, epsfig}
\usepackage[a4paper,total={170mm,250mm},vcentering,dvips]{geometry}

\usepackage{dsfont}
\usepackage{enumitem}
\usepackage{csquotes}
\usepackage{bm}
\usepackage[pdftex,dvipsnames]{xcolor}
\usepackage{float}
\usepackage[hypcap=false]{caption}
\usepackage{parskip}
\usepackage{algpseudocode}
\usepackage{algorithm}

\usepackage[hidelinks]{hyperref}
\usepackage[noabbrev,capitalize]{cleveref}
\usepackage[normalem]{ulem}

\renewcommand{\autoref}{\cref}

\newtheorem{thm}{Theorem}[section]
\newtheorem{cor}[thm]{Corollary}
\newtheorem{lem}[thm]{Lemma}

\newtheorem{defi}[thm]{Definition}

\newtheorem{rem}[thm]{Remark}

\DeclareMathOperator*{\esssup}{ess\,sup}
\DeclareMathOperator*{\essinf}{ess\,inf}

\newcommand{\R}{\mathbb{R}}
\newcommand{\N}{\mathbb{N}}
\newcommand{\V}{\mathbb{V}}

\newcommand{\diam}{\textup{diam}}
\newcommand{\Lip}{\mathcal{L}}
\newcommand{\LipCov}{\mathcal{N}}
\newcommand{\LipDelta}{\sigma}
\newcommand{\LipProp}{\textup{LipProp}}
\newcommand{\calT}{\mathcal{T}}
\newcommand{\calM}{\mathcal{M}}
\newcommand{\calP}{\mathcal{P}}
\newcommand{\calI}{\mathcal{I}}
\newcommand{\standard}{{\textrm{standard}}}

\newcommand{\thesis}[1]{{}}

\begin{document}

\title{Anisotropic approximation on space-time domains}
\author{Pedro Morin\footnote{Universidad Nacional del Litoral and CONICET, 
Departamento de Matem\'atica, 
Facultad de Ingenier\'{\i}a Qu\'{\i}mica, Santiago del Estero 2829, S3000AOM Santa Fe, Argentina. 
Emails: \href{mailto:pmorin@fiq.unl.edu.ar}{pmorin@fiq.unl.edu.ar}}, Cornelia Schneider\footnote{Friedrich-Alexander-Universit\"at Erlangen-N\"urnberg, Angewandte Mathematik III, Cauerstr. 11, 91058 Erlangen, Germany. Email: \href{mailto:schneider@math.fau.de}{schneider@math.fau.de}}
, and Nick Schneider\footnote{Friedrich-Alexander-Universit\"at Erlangen-N\"urnberg, Angewandte Mathematik III, Cauerstr. 11, 91058 Erlangen, Germany. Email: \href{mailto:nick.schneider@fau.de}{nick.schneider@fau.de}} \footnote{corresponding author}}
\maketitle
\begin{abstract}
We investigate anisotropic (piecewise) polynomial approximation of functions in Lebesgue spaces as well as anisotropic Besov spaces. For this purpose we study temporal and spatial moduli of smoothness and their properties. In particular, we prove Jackson- and Whitney-type inequalities on  Lipschitz cylinders, i.e., space-time domains $I\times D$ with a finite interval $I$ and a bounded Lipschitz domain $D\subset \R^d$, $d\in \N$. As an application, we prove a direct estimate result for adaptive space-time finite element approximation in the discontinuous setting.
{\em MSC2020 Math Subject Classifications: Primary 41A10, 41A17, 41A25, 41A63; Secondary 65M50, 65M60}.
\end{abstract}

\tableofcontents

\section{Introduction and main results}\label{sect:Introduction}

The subject of this paper is to study the approximation by anisotropic piecewise polynomials of functions in Lebesgue spaces $L_p(I\times D)$, $0<p<\infty$, on space-time domains $I \times D$ with a finite interval $I$ and a bounded Lipschitz domain $D \subset  \mathbb{R}^d$, $d \in  \mathbb{N}$,  as well as certain anisotropic Besov spaces. In particular, such functions appear as solutions of various time-dependent problems, demonstrating the importance of tackling such questions that arise in many areas of numerical analysis and PDEs. \\
The case of stationary problems and their approximation with piecewise polynomials is well understood: A good review of a variety of results
for the univariate case is given by DeVore \cite{DeV76}, while spline approximation
of multivariate functions has been developed by de Boor and Fix \cite{BF73},
 Munteanu and Schumaker \cite{MS73}, and Dahmen, DeVore, and Scherer \cite{DDS80}. 
Moreover, adaptive finite element methods (AFEM) are often used to numerically approximate the solution to a PDE and  have been experimentally (and theoretically) shown to outperform standard finite element methods for many problems. Regarding  approximation with (adaptive)
finite elements of arbitrary order we refer to the more recent paper by Gaspoz and Morin \cite{GM14} as well as  its forerunner  \cite{BDDP02} by Binev, Dahmen, DeVore, and Petrushev, who consider linear finite elements in two dimensions.\\
However, not much is known in the time-dependent setting. Here, first results were obtained by the authors in \cite{AMS23} and \cite{AGMSS25} for time-stepping finite element methods. 
We now want to pursue our investigations from a different angle and consider adaptive space-time finite element  approximation instead, thus, 
  generalizing the results from \cite{GM14}  to a time-dependent setting. Therefore, we need estimates of Jackson- and Whitney-type\footnote{The terminology for this type of inequalities is not  consistent throughout the literature. What we (and for example the authors of \cite{OS78}) call Jackson's estimate has sometimes also been called Whitney's estimate; e.g. in \cite{DL04}. Here we  follow the jargon of \cite{AMS23}.}  as well as a corresponding definition of anisotropic Besov spaces with respect to approximation with polynomials of fixed order in time and space, respectively. \\
  We found that in the literature there were no corresponding results that were suited for our FEM analysis, since the results for approximation with anisotropic polynomials (in anisotropic Besov spaces) given in Besov, Il'in, and Nikol'ski\u{i} \cite[Ch.~4]{BIN79}, Leisner \cite[Ch.~2-3]{Lei00}  or Hochmuth \cite[Sect.~2-3]{Hoc02} do not allow for the approximation with a fixed polynomial degree in space that  does not vary among the different spatial coordinate axes. Precisely, we wish to consider anisotropic polynomials of the form 
\begin{align*}
    \sum\limits_{i_0<r_1} \ \sum\limits_{i_1+i_2+\dots+i_d<r_2} c_{i_0, \dots, i_d}\,t^{i_0}\prod\limits_{j=1}^d x_j^{i_j} 
\end{align*} 
in contrast to 
\begin{align*}
    \sum\limits_{i_0<r_1} \ \sum\limits^{\infty}_{\max\{i_1,i_2,\dots,i_d\}<r_2} c_{i_0, \dots, i_d}\, t^{i_0} \prod\limits_{j=1}^d x_j^{i_j},
\end{align*}
where the summation indices $i_j$ are all nonnegative and $r_1, r_2\in \N$ are the temporal and spatial polynomial orders, respectively. In particular, for $r_2=2$ the latter approach would not allow for approximation by linear functions in space, but only for approximation with functions that are linear in every spatial direction like multilinear functions. We refer to \cite[Sect.~3.1]{Lei00} for more details on this topic. Further, in the first source \cite{BIN79} only approximation with respect to the Lebesgue norm $\|\cdot\|_{L_p}$ with $p\ge 1$ is considered, whereas in the latter ones \cite{Lei00, Hoc02} the results are always confined to unit cube domains. \\
As a main tool for our studies we now rely on the so-called spatial and temporal moduli of smoothness and their properties. For a long time moduli of smoothness have been known to be a very useful concept. Already in 1927 a version of the inequality for moduli of smoothness of different orders, the so-called Marchaud’s inequality, was proven by Marchaud \cite{Mar27}. In the literature most results for the moduli are established in the Banach setting when 
$1 \le p < \infty$ and are only partially known outside this range. Since we are also interested in $0<p< 1$, for further references and fundamental characteristics of the moduli we refer to DeVore and Lorentz \cite{DL93} as well as  the recent article by Kolomoitsev and  Tikhonov \cite{KT20}, which also covers the full range of $p$ (in the isotropic case). The importance of the moduli of smoothness stems from the fact that they allow to quantitatively measure smoothness of functions via higher order differences instead of derivatives and thus provide more general estimates than the estimates in terms of derivatives. 
In particular,  the  Jackson and Whitney inequalities, two fundamental results in approximation theory  often used in numerical analysis, can be stated and proven with the help of such moduli of smoothness. \\
Jackson’s estimate allows to measure the error of approximation of an arbitrary $L_p$-function by its modulus of smoothness. Additionally, if the function is smooth enough, e.g. has some  Besov regularity, one can even estimate the (global) local error of approximation with respect to (piecewise) polynomials by the Besov \mbox{(quasi-)}seminorm of the function leading to  Whitney's estimate (the proof makes use of the Jackson’s estimate). In this sense Jackson's inequality  is somehow preferable to Whitney's since it applies to any $f$ regardless of its smoothness. Both inequalities provide, in particular, a convergence characterization for a local polynomial approximation when the degree of the polynomials is fixed and the interval under consideration is small.\\
Inequalities of this kind were  first proven by Jackson \cite{Jac12} and Whitney \cite{Whi57} for $p = \infty$ and extended by Brudny\u{i} \cite{Bru64} to $1 \leq  p < \infty$. Several authors have extended and generalized the results in various directions: A multivariate (isotropic) generalization for functions on the unit cube $Q\subset \mathbb{R}^d$ as well as on convex domains was given by Brudny\u{i} \cite{Bru70a, Bru70b}. Moreover, let us mention the recent contributions of  Dekel and Leviatan  \cite{DL04} and Dekel \cite{Dek22} with focus on convex  domains and the improvement of the constants involved. For the case $0 < p < 1$ we refer to the works of Storo$\check{z}$enko \cite{Sto77} as well as Oswald and Storo$\check{z}$enko \cite{OS78}, where the one-dimensional case has been treated. Finally, a Whitney-type theorem for anisotropic polynomial approximation on a $d$-parallelepiped was established in \cite{DU11} by Dung and Ullrich. \\
It is the main aim of this paper to prove the mentioned inequalities for time-dependent functions.  In our proofs, we are building on techniques from the isotropic case as in \cite[Sect.~1-2]{DL04}, DeVore \cite[Sect.~6.1]{DeV98}, and \cite[Sect.~3.2]{AMS23}. Moreover, as a first application,  we  use our new Jackson's and Whitney's inequalities  in order to obtain direct approximation
theorems for discontinuous finite elements. In a forthcoming paper we then want to extend these also to continuous finite elements. Furthermore, we also plan to examine in what sense the estimates   are best possible by establishing   inverse theorems (which are much harder to obtain compared to direct theorems and therefore are out of our current scope). Together, both kinds of theorems will yield a  characterization (or at least an 'almost' characterization, see \cite{GM14}) of the smoothness of a function (measured with the help of anisotropic Besov spaces) in terms of its degree of approximation by adaptive (dis-)continuous  finite elements.

Throughout the paper, we will denote $A\lesssim B$ if there is a constant $c$ independent of $A$ and $B$ such that $A\le cB$. Moreover, we write  $A\gtrsim B$, if $B\lesssim A$ and  $A\sim B$ if $A\lesssim B$ and $A\gtrsim B$ both hold true. We will further add (possible) dependencies of the corresponding constants on parameters as subscripts to $\lesssim$, $\gtrsim$, \mbox{or $\sim$}.

\subsection{Main results}\label{sect:Main results}

Let $I$ be an interval and $D\subset \R^d$, $d\in \N$, a bounded Lipschitz domain, see Definition \ref{Definition of Lipschitz domain}. We define the space of \textbf{anisotropic polynomials of order $(r_1, r_2)\in \N^2$} on $I\times D$ as
\begin{align*}
    \Pi_{t,\bm{x}}^{r_1,r_2}(I\times D):=\Pi_{t}^{r_1}(I) \otimes \Pi_{\bm{x}}^{r_2}(D):= \left\{P:I\times D\rightarrow \R\:\bigg|\: P(t,\bm{x}):= \sum\limits_{i=0}^{r_1-1} \sum\limits_{|\alpha|<r_2} c_{i, \alpha} t^i \bm{x}^\alpha,\,\, c_{i, \alpha}\in\mathbb{R} \right\},
\end{align*}
where we use the common notations $\bm{x}^\alpha := \prod\limits^{d}_{i=1} x_i^{\alpha_i}$ and $|\alpha|:=\sum\limits_{i=1}^{d} \alpha_i$ for $\bm{x}\in D$ and multi-indices $\alpha \in \N_0^d$. Using the \textbf{temporal and spatial moduli of smoothness}  $\omega_{r_1,t}(\cdot, I\times D, \cdot)_p$ and $\omega_{r_2,\bm{x}}(\cdot, I\times D, \cdot)_p$ from Definition \ref{def:moduli_of_smoothness}, respectively, the following time-dependent version of \cite[Thm.~1.4]{DL04} holds true.

\begin{thm}[Jackson-type inequality]\label{thm:Jackson}
For any $p\in (0,\infty]$ and $f\in L_p(I\times D)$, there exists \mbox{$P\in \Pi_{t,\bm{x}}^{r_1,r_2}(I\times D)$} such that
\begin{align}\label{thm:Jackson - equation}
    \|f-P\|_{L_p(I\times D)}\lesssim_{\,d,p,r_1, r_2, \LipProp(D)\,} \omega_{r_1,t}(f, I\times D, |I|)_p + \omega_{r_2,\bm{x}}(f, I\times D, \diam(D))_p,
\end{align}
where $|I|$ is the length of the interval $I$, $\diam(D):=\sup\limits_{x,y\in D}|x-y|$ is the diameter of $D$, and $\LipProp(D)$ denotes the Lipschitz properties of the domain $D$, which are specified below in \cref{Section:Preliminaries}. Further, the dependency on $\LipProp(D)$ can be omitted if $D$ is additionally convex. 
\end{thm}

Furthermore, we consider space-time domains $I\times D$ with $D$ \textit{polyhedral} which can be covered by a \textbf{space-time partiton} $\calP$, i.e., a non-overlapping covering consisting of prisms of the form $J\times S$ with $J$ an interval and $S$ a \mbox{$d$-dimensional simplex}. For such domains, we can obtain the following result for the approximation of anisotropic Besov functions, see Definition \ref{Definition anisotropic Besov space}.

\begin{thm}[Whitney-type inequality]\label{thm:Whitney}
    For any $s_1, s_2\in (0,\infty)$ and $p,q\in (0,\infty]$ with $\frac{1}{\frac{1}{s_1}+\frac{d}{s_2}}-\frac{1}{q}+\frac{1}{p}>0$, $J\times S\in \calP$, and $f\in B^{s_1, s_2}_{q,q}(I\times D)$, there exists $P_{J\times S}\in \Pi^{r_1, r_2}_{t,\bm{x}}(J\times S)$ such that
    \begin{align}\label{thm:Whitney - equation local}
        \|f-P_{J\times S}\|_{L_p(J\times S)} \lesssim_{\, d,p,q,s_1,s_2, \kappa_S ,a(\calP)} |J\times S|^{\frac{1}{\frac{1}{s_1}+\frac{d}{s_2}}-\frac{1}{q}+\frac{1}{p}}|f|_{B^{s_1,s_2}_{q,q}(J\times S)}.
    \end{align}
    Further, setting $P:=\sum\limits_{J\times S\in \calP} \mathds{1}_{J\times S} P_{J\times S}$, where $\mathds{1}_{J\times S}$ denotes the indicator function of the set $J\times S$, it holds
    \begin{equation}\label{thm:Whitney - equation global}
    \begin{split}
        \|f-P\|_{L_p(I\times D)}&\lesssim_{\, d,p,q,s_1,s_2, \kappa_\calP, a(\calP)} 
        \max_{J\times S\in \calP}|J\times S|^{\frac{1}{\frac{1}{s_1}+\frac{d}{s_2}}-\frac{1}{q}+\frac{1}{p}}|f|_{B^{s_1,s_2}_{q,q}(I\times D)}
        \\
        &\le
        |I\times D|^{\frac{1}{\frac{1}{s_1}+\frac{d}{s_2}}-\frac{1}{q}+\frac{1}{p}}|f|_{B^{s_1,s_2}_{q,q}(I\times D)}.
        \end{split}
    \end{equation}
\end{thm}

Here, $|\cdot|_{B^{s_1,s_2}_{q,q}}$ denotes the corresponding anisotropic Besov \mbox{(quasi-)}seminorm which will be specified in \cref{Section:Approximation_with_Besov_functions},
\begin{align*}
    \kappa_\calP:=\sup\limits_{J\times S\in \calP} \kappa_S :=\sup\limits_{J\times S\in \calP} \frac{\diam(S)}{\rho_S} :=\sup\limits_{J\times S\in \calP}\frac{\diam(S)}{\sup\{\rho\in(0,\infty)\mid B_{\rho}(x_0)\subset S\text{ for some } x_0\in S\}}
\end{align*}
\begin{figure}[H]
	\begin{center}
		\includegraphics[height=4cm]{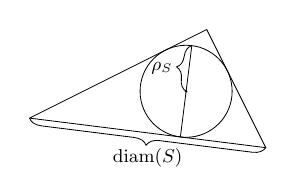}
	\end{center}
    \captionof{figure}{Illustration for $d=2$}
\end{figure}
measures the minimal head angle in the simplices $S$ of the prisms $J\times S\in \calP$, and 
\begin{align}
    a(\calP, s_1, s_2) := a(\calP) := \max\limits_{J\times S\in \calP}\left(\max\left(\frac{|J|}{|S|^{\frac{s_2}{s_1d}}},\frac{|S|^{\frac{s_2}{s_1d}}}{|J|}\right)\right)\label{def:aP}
\end{align}
describes the \textbf{(maximal) anisotropy of the space-time partition} $\calP$.

This directly leads to the following embedding result.

\begin{thm}[Embedding result]\label{thm:Besov_Embedding}
    Under the assumptions of \autoref{thm:Whitney}, the embedding 
    \begin{align*}
        B^{s_1,s_2}_{q,q}(I\times D)\hookrightarrow L_p(I\times D)
    \end{align*} 
    is continuous with an embedding constant dependent on $d,p,q,s_1,s_2, \kappa_\calP, a(\calP), \min\limits_{J\times S\in \calP} |J\times S|$, and $\max\limits_{J\times S\in \calP} |J\times S|$.
\end{thm}

\begin{rem}\label{Remark:Besov_Embedding}
A corresponding embedding for anisotropic Besov spaces which are defined in a slightly different way, can already be found in \cite[Thm.~18.10]{BIN79}. There, in general  the underlying  domain has to fulfill a so-called {\textbf{ horn}}-condition with respect to the anisotropy parameters corresponding to $s_1$ and $s_2$ for the embedding to be true. However, since our domain $I\times D$ has a product structure this condition is fulfilled if the smoothness with respect to the spatial variable is the same in all directions (here $s_2$). We  also refer to \cite[Rem.~3.6]{DS23} for further information in this context. 
\end{rem}

As an application, we will show interesting results for approximation with \textbf{anisotropic space-time finite elements} of order $(r_1, r_2)\in \N^2$ on a space-time partition $\calP$, i.e., we consider the approximation space
\begin{align*}
    \mathbb{V}^{r_1, r_2}_{\calP, \textup{DC}}(I\times D):=\left\{F:I\times D\rightarrow \R\mid F_{|J\times S}\in \Pi^{r_1, r_2}_{t,\bm{x}}(J\times S), \quad J\times S \in \calP\right\}.
\end{align*}

\begin{thm}[Direct theorem]\label{Theorem_dir_est_discont}
	Let $p,q,s_1, s_2$, and $f$ as in \cref{thm:Whitney}, $r_i\in \N$ with $r_i>s_i$, $i=1,2$, and $\varepsilon\in (0,\infty)$. 
    Further, let $\calP_0$ be a tensor-product structure space-time partition covering $I\times D$. Then there exists a refined space-time partition $\calP$ with
    \begin{align}\label{Theorem_dir_est_discont_1}
		\#\calP - \#\calP_0 \le C_1\, \varepsilon^{-\left(\frac{1}{s_1}+\frac{d}{s_2}\right)}
	\end{align}
    with $C_1 = C_1(d, p, q, s_1, s_2, \kappa_{\calP_0}, a(\calP_0), |I\times D|, \min\limits_{J_0\times S_0 \in \calP_0} |J_0\times S_0|, \max\limits_{J_0\times S_0 \in \calP_0} |J_0\times S_0|)$, 
    such that for some $F\in \mathbb{V}^{r_1, r_2}_{\calP, \textup{DC}}(I\times D)$ it holds 
        \begin{align}\label{Theorem_dir_est_discont_2}
        \|f-F\|_{L_p(I\times D)} \le C_2 \, \varepsilon |f|_{B^{s_1, s_2}_{q,q}(I\times D)},
\end{align}
with $C_2 = C_2(\,d, p, q, s_1, s_2, \kappa_{\calP_0}, a(\calP_0), |I\times D|, \min\limits_{J_0\times S_0 \in \calP_0} |J_0\times S_0|, \max\limits_{J_0\times S_0 \in \calP_0} |J_0\times S_0|,\# \calP_0 )$
\end{thm}

Here, \textbf{refinement} means finitely many applications of the method $\textup{ATOMIC}\underline{~}\textup{SPLIT}(\cdot, d, s_1, s_2)$ that will be discussed in \cref{Sect:Dir_est}.

\subsection{Preliminaries}\label{Section:Preliminaries}

Since we will often consider bounded \textbf{Lipschitz domains}, we will give a definition according to \cite[Def.~2.2]{DL04}, to be self-contained.

\begin{defi}\label{Definition of Lipschitz domain}
Let $D\subset \R^d$ for $d\in \N_{\ge 2}$. 
	\begin{enumerate}[label=(\roman*)]
		\item Let $\psi:\R^{d-1}\rightarrow \R$ be a Lipschitz function with Lipschitz constant $\Lip(D)\in (0,\infty)$. Then if  
		\begin{align*}
			D = \{(u,v)\in \R^{d-1}\times \R \mid \psi(u)>v \} ,
		\end{align*}	 
	$D$ is called a \textbf{Lipschitz graph domain}. 
		\item If $D$ is a rotated and/or translated Lipschitz graph domain, it is called a \textbf{special Lipschitz domain}.
		\item $D$ is called a \textbf{Lipschitz domain}, if it is bounded and there are open sets $U_j$ and special Lipschitz domains $D_j$ for $j=1,\dots, \LipCov(D)$, \mbox{$\LipCov(D)\in \N$}, with respect to the Lipschitz constant \mbox{$\Lip(D)\in (0,\infty)$} such that the following conditions hold true:
        \begin{enumerate}
			\item $\displaystyle{\partial D\subset \bigcup\limits_{j=1}^{\LipCov(D)} U_j}$;
			\item $D\cap U_j = D_j\cap U_j$,$\ \ j=1,\dots, \LipCov(D)$; \label{Definition of Lipschitz domain - special Lipschitz domain part}
			\item There is $\LipDelta(D)\in(0,\infty)$ such that for every pair of points \mbox{$\bm{x}_1, \bm{x}_2\in D$} with $|\bm{x}_1 - \bm{x}_2|<\LipDelta(D)$ and $d(\bm{x}_i, \partial D)<\LipDelta(D)$, $i=1,2$, there exists an index \mbox{$j\in\{1,\dots,\LipCov(D)\}$} with $\bm{x}_i\in U_j$ and $d(\bm{x}_i, \partial U_j)>\sigma(D)$ for $i\in\{1,2\}$.\label{Definition of Lipschitz domain - Delta part}\footnote{In general, for $\bm{p}\in \R^d$ a and set $M\subset \R^d$, we denote the euclidean distance between $d$ and $M$ by $d(\bm{p},M)$. Similarly, the distance between $M$ and another set $N\subset\R^d$ is denoted by $d(M,N)$.}
		\end{enumerate}
		The set $\LipProp(D):=\left\{\Lip(D),\LipCov(D) , \frac{\LipDelta(D)}{\diam(D)} \right\}$ is called (the set of) \textbf{Lipschitz properties} of $D$. 
    \end{enumerate}
\end{defi}

\begin{rem}\label{Remark Lipschitz domain}
	\begin{enumerate}[label=(\roman*)]
		\item Let us point out here that $\ref{Definition of Lipschitz domain - Delta part})$ is not necessary for the bounded Lipschitz domains that we will consider, according to \cite[Def.~4.9]{AF03} and the corresponding comment below, respectively. Nevertheless, we choose this definition since we will often need the parameter $\LipDelta(D)$ in the course of this article.
        
		\item Every Lipschitz domain fulfills an \textbf{interior cone condition}, i.e., there is a bounded cone $\mathcal{C}$ with cone tip at $0$ and $\diam(\mathcal{C})\le\frac{\LipDelta(D)}{4}$ depending only on $\Lip(D)$ and $\LipDelta(D)$ such that there are rotations $R_j$, \mbox{$j=1,\dots, \LipCov(D)$}, with the property that for every $\bm{x}\in D$ there exists some $j$ with $\bm{x} + R_j(\mathcal{C})\subset D$. We put $\mathcal{C}_j:=R_j(\mathcal{C})$.\label{Remark Lipschitz domain - interior cone condition}\footnote{Qualitatively, this result can be found in \cite[Rem.~4.11]{AF03}, but there the bound on the diameter is not given. To be self-contained, we give a short proof in \cref{Appendix:Proof of Lemma equivalence averaged spacial modulus of smoothness}.}
		
		\item For any Lipschitz domain, it holds true that $\LipDelta(D) \le \diam(D)$, $j=1,\dots,\LipCov(D)$.\footnote{We assume that this result is well known, but nevertheless give a short proof in \cref{Appendix:Proof of Lemma equivalence averaged spacial modulus of smoothness} as convenience for the reader.} But a lower bound for $\LipDelta(D)$ in terms of $\diam(D)$ cannot be given for arbitrary Lipschitz domains. \label{Definition of Lipschitz domain - delta < diam}
		
	\end{enumerate}
\end{rem}

In the proof of \cref{thm:Jackson} further below, we will need to find a converging sub-sequence of a sequence of bounded Lipschitz domains which share the same Lipschitz properties. To do so, we will make good use of the following lemma. 

\begin{lem}\label{lem:convergence_of_Lipschitz_domains}
    Let $(D_m)_{m\in \N}\subset \R^d$, $d\in \N$, be a sequence of Lipschitz domains with $\LipProp(D_{m_1})=\LipProp(D_{m_2})$ for $m_1, m_2\in \N$. Additionally, assume that $\diam(D_m)=1$ and $D_m\subset [0,1]^d$ for any $m\in \N$. Then, there is a subsequence $(D_{m_k})_{k\in\N}$ converging against another bounded Lipschitz domain $D\subset [0,1]^d$ with $\diam(D)=1$ in the sense that 
    \begin{align}\label{lem:convergence_of_Lipschitz_domains - equation}
        |D\triangle D_{m_k}|:=|(D\setminus D_{m_k})\cup (D_{m_k}\setminus D)|\xrightarrow{k\rightarrow \infty} 0 \quad\text{and}\quad  |D_{r_2, \bm{h}}\triangle (D_{m_k})_{r_2, \bm{h}}| \xrightarrow{k\rightarrow \infty} 0
    \end{align}
    for any $\bm{h}\in \R^d$ with $|\bm{h}|\le 1$.
\end{lem}
\begin{proof}
    \cite[Rem.~2.4.8 and Thm.~2.4.10]{HP18} implies the existence of a bounded Lipschitz domain $D$ to which a subsequence converges with respect to three different notions of convergence, i.e., with respect to the Hausdorff distance, their characteristic functions, and compactly. Furthermore, also the closures $(\overline{D_{m_k}})_{k\in \N}$ and boundaries $(\partial D_{m_k})_{k\in \N}$ converge to $\overline{D}$ and $\partial D$ with respect to the Hausdorff distance, respectively. Since the precise definitions of these types of convergence for sets are not of much relevance in this article, apart from this proof, we do not explicitly provide them here, but instead refer the reader to \cite[Ch.~2.2]{HP18}, where they are very well illustrated. 
    
    In particular, the convergence with respect to the Hausdorff distance implies that $D\subset [0,1]^d$, since $D_m\subset [0,1]^d$ for all $m\in \N$. Further, the diameter of two compact sets is continuous with respect to the Hausdorff distance, therefore $\diam(D)=\diam(\overline{D})=\lim\limits_{k\rightarrow \infty} \diam(\overline{D_{m_k}})=\lim\limits_{k\rightarrow \infty} \diam(D_{m_k})=1$.\footnote{According to \cite[Ch.~2.2.3.2~(10)]{HP18} this might be false for open sets, but it is easy to show that it indeed holds true for compact ones. In particular, this is in agreement with \cite[Ch.~2.2.3~12.]{HP05}, which is the forerunner to \cite{HP18}.} Moreover, the convergence with respect to the characteristic functions yields the first assertion in \eqref{lem:convergence_of_Lipschitz_domains - equation}, according to \cite[Rem.~2.2.5]{HP18}. The second one now follows from it, additionally using the convergence of $(\overline{D}_{m_k})_{k\in \N}$ and $(\partial D_{m_k})_{k\in \N} $ to $\overline{D}$ and $\partial D$, respectively, with respect to the Hausdorff distance.
\end{proof}

Further, an important kind of Lipschitz domains are \textbf{convex sets} which have special Lipschitz properties.
\begin{thm}\label{thm:Lipschitz domain - convex sets}
    Let $D\subset \R^d$, $d\in \N$, be a convex domain. Then $D$ is a Lipschitz domain and there is a non-singular affine linear mapping $A:\R^d\rightarrow \R^d$ such that  $\tilde{D}:=A(D)$ has the properties $B_1(0)\subset \tilde{D}\subset B_d(0)$ and $\LipProp(\tilde{D})\sim_{\,d}1$.
\end{thm}
\begin{proof}
    This is a direct consequence of \textit{John's Theorem}, see \cite[Prop.~2.5]{DL04} (originally proven in \cite[Thm.~III]{Joh48})  together with \cite[Thm.~2.3]{DL04}.
\end{proof}

The rest of this paper is organized as follows: In Section \ref{Sect:Approx of Lebesgue functions} we define the temporal and spatial moduli of smoothness, provide Marchaud's inequality for them, and present a proof for  Jackson-type estimates. In Section \ref{Section:Approximation_with_Besov_functions} we introduce a certain class of anisotropic Besov spaces, study some of their properties, and use them to establish  Whitney-type inequalities. Finally, in Section \ref{Sect:Dir_est} we apply our results in order to obtain direct estimates for approximation with discontinuous adaptive space-time finite elements. Some proofs were moved to  Appendix \ref{Appendix:Proof of Lemma equivalence averaged spacial modulus of smoothness} in order to improve the readability of the paper.

\section{Approximation of Lebesgue functions}\label{Sect:Approx of Lebesgue functions}

From now on, let $I$ be a bounded interval and $D\subset \R^d$, $d\in \N$, a bounded Lipschitz domain. Further, we will use the notation $\mu(p):=\min(1,p)$ for any $p\in (0,\infty]$ for the rest of the article. Then, in particular, the \mbox{(quasi-)}norm $\|\cdot\|_{L_p}$ of the corresponding $L_p$-Lebesgue space is $\mu(p)$-subadditive on arbitrary domains.

\subsection{Temporal and spatial moduli of smoothness}\label{Subsect: Temporal and spacial moduli of smoothness}

For our analysis and definition of anisotropic Besov spaces later on, in particular, for measuring higher smoothness of functions with respect to time and space,  we need to introduce the temporal and spatial moduli of smoothness which in turn are based on the following finite difference operators.

\begin{defi}\label{def:diff_op} We  define the \textbf{temporal and spatial difference operators} for $f:I\times D\rightarrow \R$ inductively via
    \begin{align*}
		&\Delta^1_{h,t}f(\tau,\bm{z}):=\Delta_{h,t}f(\tau,\bm{z}):=f(\tau+h, \bm{z}) - f(\tau,\bm{z}),& \quad &\Delta_{h,t}^{r_1}=\Delta_{h,t}^{r_1-1}\circ \Delta_{h,t}^1\text{ for }r_1\ge 2,& \quad &h\in \R,
		\\&\Delta^1_{\bm{h},\bm{x}}f(\tau,\bm{z}):=\Delta_{\bm{h},\bm{x}}f(\tau,\bm{z}):=f(\tau,\bm{z}+\bm{h}) - f(\tau,\bm{z}),& \quad &\Delta_{\bm{h},\bm{x}}^{r_2}=\Delta_{\bm{h},\bm{x}}^{r_2-1}\circ \Delta_{\bm{h},\bm{x}}^1\text{ for }r_2\ge 2,& \quad &\bm{h}\in \R^d,
	\end{align*}
	with $(\tau,\bm{z})\in I_{r_1,h}\times D$ and $(\tau,\bm{z})\in I\times D_{r_2, \bm{h}}$, respectively. Here $I_{r_1,h}:=\{\tau\in I\mid \tau+ih\in I\text{ for all }i=1,\dots,r_1\}$ and $D_{r_2, \bm{h}}:=\{\bm{z}\in D\mid \bm{z}+i\bm{h}\in D\text{ for all }i=1,\dots,r_2\}$. Moreover, we set $\Delta^0_{h,t}f:=\Delta^0_{\bm{h},\bm{x}}f:=f$.
\end{defi}

\begin{rem}\label{Rem:Difference operators}
    Similar to \cite[Ch.~2.7]{DL93}  one can inductively show the algebraic identities
    \begin{align*}
        \Delta_{h,t}^{r_1}f(\tau,\bm{z}) = \sum\limits^{r_1}_{i=0}(-1)^{r_1-i}\binom{r_1}{i}f(\tau+ih,\bm{z})\quad \text{and} \quad \Delta_{\bm{h}, \bm{x}}^{r_2}f(\tau,\bm{z}) = \sum\limits^{r_2}_{i=0}(-1)^{r_2-i}\binom{r_2}{i}f(\tau,\bm{z}+i\bm{h})
    \end{align*}
    for $r_1, r_2 \in \N_0$. In particular, let $r_1, r_2, k_1, k_2 \in \N_0$, $k_i\le r_i$, $i=1,2$, $h\in \R$, and $\bm{h}\in \R^d$. Then, the algebraic identities together with the $\mu(p)$-subadditivity of the $\|\cdot\|_{L_p}$-\mbox{(quasi-)}norm imply the estimates $\|\Delta_{h,t}^{r_1}f\|_{L_p(I_{r_1,h}\times D)}^{\mu(p)}\le 2^{r_1 - k_1}\|\Delta_{h,t}^{k_1}f\|_{L_p(I_{k_1,h}\times D)}^{\mu(p)}$ and $\|\Delta_{\bm{h},\bm{x}}^{r_2}f\|_{L_p(I\times D_{r_2,\bm{h}})}^{\mu(p)}\le 2^{r_2 - k_2}\|\Delta_{\bm{h},\bm{x}}^{k_2}f\|_{L_p(I\times D_{r_2,\bm{h})}}^{\mu(p)}$ for any \mbox{$f\in L_p(I\times D)$}.
\end{rem}

\begin{defi}\label{def:moduli_of_smoothness}
    Let  $p\in(0,\infty]$, $f\in L_p(I\times D)$, $r_1, r_2\in \N_0$, and $\delta\in(0,\infty)$. The  \textbf{temporal and spatial moduli of smoothness} in the supremum form are defined as 
    \begin{align*}
		\omega_{r_1, t}(f,I\times D,\delta)_p:= \sup\limits_{|h|\le \delta}\|\Delta_{h,t}^{r_1}f\|_{L_p\left(I_{r_1, h}\times D\right)}, \quad \omega_{r_2, \bm{x}}(f,I\times D,\delta)_p:= \sup\limits_{|\bm{h}|\le \delta}\|\Delta_{\bm{h},\bm{x}}^{r_2}f\|_{L_p\left(I\times D_{r_2, \bm{h}}\right)}
	\end{align*}
    as well as in the averaged form, for $0<p<\infty$:
    \begin{align*}
        &w_{r_1, t}(f,I\times D,\delta)_p:= \left(\frac{1}{2\delta}\int\limits_{|h|\le \delta}\|\Delta_{h,t}^{r_1}f\|^p_{L_p\left(I_{r_1, h}\times D\right)}\,dh\right)^\frac{1}{p}, 
        \\&w_{r_2, \bm{x}}(f,I\times D,\delta)_p:= \left(\frac{1}{(2\delta)^d}\int\limits_{|h|\le \delta}\|\Delta_{\bm{h},\bm{x}}^{r_2}f\|^p_{L_p\left(I\times D_{r_2, \bm{h}}\right)}\,d\bm{h}\right)^\frac{1}{p}.
    \end{align*}
    Further, $w_{r_1, t}(f,I\times D,\delta)_\infty:=\omega_{r_1, t}(f,I\times D,\delta)_\infty$, $w_{r_2, \bm{x}}(f,I\times D,\delta)_\infty:=\omega_{r_2, \bm{x}}(f,I\times D,\delta)_\infty$, and $\omega_{r_1, t}(f,I\times D,0)_p:=w_{r_1,t}(f,I\times D,0)_p:=\omega_{r_2, \bm{x}}(f,I\times D,0)_p:=w_{r_2,\bm{x}}(f,I\times D,0)_p:=0$, if $r_1, r_2\ge 1$.
\end{defi}

\begin{rem}\label{Rem:Moduli of smoothness}
Let $f\in L_p(I\times D)$, $(r_i, dir)\in \{(r_1, t),(r_2, \bm{x})\}$, $k_i\in \N_0$ with $k_i\le r_i$, $\delta\in[0,\infty)$, $m\in \N$, and $l\in [0,\infty)$. Then, as in the isotropic case, 
\begin{enumerate}[label=(\roman*)]
\item $\omega_{r_i, dir}(f, I\times D, \cdot)_p$ is monotonically increasing, 
\item $\omega_{r_i, dir}(\cdot, I\times D, \delta)^{\mu(p)}_p$ is subadditive, 
\item $\omega_{r_i, dir}(f, I\times D, \delta)_p^{\mu(p)}\le 2^{r_i-k_i}\omega_{k_i, dir}(f, I\times D, \delta)_p^{\mu(p)}$, and \label{Rem:Moduli of smoothness - scaling item}
\item admits the two similar scaling properties \begin{align*}
\omega_{r_i, dir}(f, I\times D, m\delta)_p^{\mu(p)} &\le m^{r_i}\omega_{r_i, dir}(f, I\times D, \delta)_p^{\mu(p)}, \\ \omega_{r_i, dir}(f, I\times D, l\delta)_p^{\mu(p)} &\le (l+1)^{r_i}\omega_{r_i, dir}(f, I\times D, \delta)_p^{\mu(p)}.
\end{align*}
\end{enumerate}
Furthermore, property \ref{Rem:Moduli of smoothness - scaling item} still holds true in the case $r_i=0$ when \mbox{$\omega_{0,dir}(f,I\times D, \cdot)_p:=\|f\|_{L_p(I\times D)}$} is used. Moreover, the third, and fourth property hold true for the averaged moduli. These results can be proven as in \cite[Ch.~2.7~and~Ch.~12.5]{DL93}. 
\end{rem}

We will also make use of the following lemma that characterizes the behaviour of the moduli of smoothness under affine linear transformations.

\begin{lem}\label{lem:Behaviour under affine linear transformations}
    Let $p\in(0,\infty]$, $f\in L_p(I\times D)$, $\delta\in [0,\infty)$, and  $r_1, r_2\in\N_0$. Further, let $I'$ be a bounded interval, $D'\subset \R^d $, $d\in \N$, a bounded Lipschitz domain, and $\varphi:I'\times D'\rightarrow I\times D$ an affine linear and bijective mapping with $\varphi(\tau,\bm{z})=(a\tau+b, c\bm{z}+d) $ for $a,b,c,d\in \R$. Then 
 \begin{equation}\label{moduli-1}
    \begin{split}
        \omega_{r_1, t}(f, I\times D, \delta)_p & = \left(\frac{|I\times D|}{|I'\times D'|}\right)^\frac{1}{p} \omega_{r_1, t}\left(f\circ \varphi, I'\times D', \delta\,\frac{|I'|}{|I|}\right)_p 
        \\\text{and}\qquad    \omega_{r_2, \bm{x}}(f, I\times D, \delta)_p & = \left(\frac{|I\times D|}{|I'\times D'|}\right)^\frac{1}{p} \omega_{r_2, \bm{x}}\left(f\circ \varphi, I'\times D', \delta\,\frac{|D'|^\frac{1}{d}}{|D|^\frac{1}{d}}\right)_p.  
    \end{split}
    \end{equation}
    Such mappings will be called \textbf{scalings} in the sequel. More generally, for $\varphi=(\varphi_1, \varphi_2)$ with $\varphi_1:I'\rightarrow I$ and $\varphi_2:D'\rightarrow D$ affine linear and bijective, i.e., $\varphi_1(t)= at+b$, $a,b\in \R$, $a\neq 0$, and $\varphi_2(\bm{x})=M\bm{x}+\bm{v}$, $M\in \R^{d\times d}$ regular, $\bm{v}\in \R^d$, it holds,  
    \begin{equation}\label{moduli-2}
    \begin{split}
        \omega_{r_1, t}(f, I\times D, |I|)_p 
        & = |\textup{det } \nabla \varphi|^\frac{1}{p} \, \omega_{r_1, t}\left(f\circ \varphi, I'\times D', |I'|\right)_p
        \\\text{and}\qquad     \omega_{r_2, \bm{x}}(f, I\times D, \diam(D))_p & = |\textup{det } \nabla \varphi|^\frac{1}{p} \omega_{r_2, \bm{x}}\left(f\circ \varphi, I'\times D', \diam(D')\right)_p,
    \end{split}
    \end{equation}
    where $\nabla \varphi$ corresponds to the Jacobian of $\varphi$. The corresponding results hold for the averaged moduli as well.
\end{lem}
\begin{proof}
    The first two results \eqref{moduli-1} are an immediate consequence of the Jacobi transformation theorem. The second two in \eqref{moduli-2} follow by the latter together with the fact that 
    \begin{align*}
        \varphi_1\big(\big\{h'\in \R^d \mid h'=t'-\tau', \text{ for } t', \tau'\in \overline{I'}\big\}\big)=\big\{h\in \R^d \mid h=t-\tau, \text{ for } t,\tau\in \overline{I}\big\}.
    \end{align*}
    and
    \begin{align*}
        \varphi_2\big(\big\{\bm{h'}\in \R^d \mid \bm{h'}=\bm{y'}-\bm{z'}, \text{ for } \bm{y'},\bm{z'}\in \overline{D'}\big\}\big)=\big\{\bm{h}\in \R^d \mid \bm{h}=\bm{y}-\bm{z}, \text{ for } \bm{y},\bm{z}\in \overline{D}\big\}.
    \end{align*}
\end{proof}
\begin{rem}
    The corresponding result to \eqref{moduli-2} for the isotropic case has (for example) been used in \cite[proof~of~Thm.~1.4]{DL04}.
\end{rem}

Furthermore, both versions of the moduli are equivalent for small enough $\delta$:

\begin{lem}\label{Lemma equivalence averaged temporal modulus of smoothness}
	Let $p\in(0,\infty]$, $f\in L_p(I\times D)$, and $r\in\N$. Then
	\begin{equation*}
		 w_{r,t}(f, I\times D, \delta)_p\le \omega_{r,t}(f, I\times D, \delta)_p \lesssim_{\,p,r}w_{r,t}(f, I\times D, \delta)_p
	\end{equation*}
	for all $\delta\in[0,\delta_0]$ and $0<\delta_0=\delta_0(r, |I|))=\frac{|I|}{4r}$.
\end{lem}
\begin{proof}
    Analogous to Lemma 5.1 of \cite[Ch.~6.5]{DL93}.
\end{proof}

\begin{lem}\label{Lemma equivalence averaged spacial modulus of smoothness}
	Let $p\in(0,\infty]$, $f\in L_p(I\times D)$, and $r\in\N$. Then
	\begin{equation*}
		 w_{r,\bm{x}}(f,I\times D, \delta)_p\le \omega_{r,\bm{x}}(f,I\times D, \delta)_p \lesssim_{\,d,p,r,\LipProp(D)} w_{r,\bm{x}}(f,I\times D, \delta)_p
	\end{equation*}
	for all $\delta\in[0,\delta_0]$ and $0<\delta_0=\delta_0(r, \LipDelta(D))=\frac{\LipDelta(D)}{4r}$.
\end{lem}
\begin{proof}
    Qualitatively, this can be shown as in \cite[Cor.~4.2]{GM14}. Since we state an explicit range for admissible $\delta$, we give a proof in \cref{Appendix:Proof of Lemma equivalence averaged spacial modulus of smoothness} for the readers convenience. 
\end{proof}

\subsection{Marchaud's inequality}

In order to be able to prove the Jackson-type inequality in our setting, we first need a Marchaud-type inequality for our newly defined moduli of smoothness. For the temporal one, this works as in the isotropic case.

\begin{thm}\label{Theorem Marchaud temporal modulus of smoothness}
	Let $r\in\N_{\ge 2}$, $k\in\N$, $k<r$, $f\in L_p(I\times D)$, and $\delta\in [0,\infty)$. Then it holds
	\begin{equation*}
		\omega_{k,t}(f, I\times D, \delta )_p^{\mu(p)}\lesssim_{\,p,k,r}\delta^{k\mu(p)} \left( \|f\|_{L_p(I\times D)}^{\mu(p)} + \int\limits_\delta^\infty \frac{\omega_{r,t}(f, I\times D, s )_p^{\mu(p)}}{s^{k\mu(p)}}\frac{ds}{s}\right).
	\end{equation*}	
\end{thm}

\begin{proof}
	Analogous to \cite[Ch.~2.8]{DL93}. In particular, scaling from $I\times D$ to $[0,1]\times D$ and again rescaling shows that the constant in the estimate can be chosen independently of $I$, due to the first part of \cref{lem:Behaviour under affine linear transformations}.
\end{proof}

In the spatial direction, it is not possible  to proceed as in the proof of the corresponding isotropic result for higher-dimensional domains as in \cite[Thm.~2]{Dit88}, where $p\ge 1$ was chosen, since this proof relies on extension operators that we do not have available in our setting. Also, we will not use the technical approach from \cite[Thm.~1.21~and~Cor.~1.22]{Dek22} for the case $p\in (0,1)$. Instead, we now state the isotropic Marchaud's inequality in the setting/formulation from \cite[Rem.~4.8]{GM14} and then derive a Marchaud's inequality for our spatial modulus of smoothness using the results from \cref{Subsect: Temporal and spacial moduli of smoothness}. 

\begin{lem}[{\cite[Rem.~4.8]{GM14}}]\label{Theorem Marchaud spacial modulus of smoothness - isotropic}
	Let $p\in(0,\infty]$, $r\in\N_{\ge 2}$, $k\in\N$, $k<r$, $f\in L_p(D)$, and $\delta\in [0,\infty)$. Then
	\begin{align*}
		\omega_{k}(f, D, \delta )_p^{\mu(p)}\lesssim_{\,p,k,r, \LipProp(D)} \delta^{k\mu(p)} \left(\|f\|_{L_p(D)}^{\mu(p)} + \int\limits_\delta^\infty \frac{\omega_{r}(f, D, s )_p^{\mu(p)}}{s^{k\mu(p)}}\frac{ds}{s}\right).
	\end{align*}
\end{lem}
\begin{rem}
    The modulus $\omega_r(\cdot, D, \cdot)_{p}$ above is defined like our spatial modulus of smoothness from \cref{Subsect: Temporal and spacial moduli of smoothness} except for the missing time dependency. In the same manner, we will also use the averaged version $w_r(\cdot, D, \cdot)_{p}$, which is defined accordingly. Also, for these moduli, \cref{Lemma equivalence averaged spacial modulus of smoothness} holds correspondingly.
\end{rem}

Now the Marchaud inequality for the spatial modulus of smoothness can be stated as follows.

\begin{thm}\label{Theorem Marchaud spacial modulus of smoothness - anisotropic}
	Let $p\in(0,\infty]$, $r\in\N_{\ge 2}$, $k\in\N$, $k<r$, $f\in L_p(I\times D)$, and $\delta\in [0,\infty)$. Then
	\begin{align*}
		\omega_{k,\bm{x}}(f, & I\times D, \delta )_p^{\mu(p)} \lesssim_{\,d,p,k,r, \LipProp(D)} \delta^{k\mu(p)} \left( \|f\|_{L_p(I\times D)}^{\mu(p)} + \int\limits_\delta^\infty \frac{\omega_{r,\bm{x}}(f, I\times D, s )_p^{\mu(p)}}{s^{k\mu(p)}}\frac{ds}{s}\right).
	\end{align*}	
\end{thm}
\begin{proof}
    First, we consider the case $p<\infty$. Since the estimate which we want to prove remains invariant under scalings, due to the first part of \cref{lem:Behaviour under affine linear transformations}, which, in particular, does not change the Lipschitz properties of $D$, we can assume $\LipDelta(D)=1$ without loss of generality. For $\tau\in I$, we define $f_\tau:=f(\tau,\cdot)$. Then $f_\tau\in L_p(D)$ almost everywhere on $I$ such that we can calculate
    \begin{align*}
		\omega_{k,\bm{x}}(f, I\times D, \delta)_p^{\mu(p)}&= \left(\sup\limits_{|\bm{h}|\le \delta}\int\limits_I \|\Delta^k_{\bm{h}}f_\tau\|^p_{L_p(D_{k,\bm{h}})}\,d\tau\right)^\frac{\mu(p)}{p}
		\\&\le \left(\int\limits_I \sup\limits_{|\bm{h}|\le \delta} \|\Delta^k_{\bm{h}}f_\tau\|^p_{L_p(D_{k,\bm{h}})}\,d\tau\right)^\frac{\mu(p)}{p}
		\\&= \left(\int\limits_I \left(\omega_{k}(f_\tau, D, \delta)_p^{\mu(p)}\right)^\frac{p}{\mu(p)} \,d\tau\right)^\frac{\mu(p)}{p}.
	\end{align*}
    Next, we can use this together with \cref{Theorem Marchaud spacial modulus of smoothness - isotropic} to derive
    \begin{equation}\label{Theorem Marchaud spacial modulus of smoothness - anisotropic proof step 2}
    \begin{split}
	\omega_{k,\bm{x}}(f, I\times D, \delta)_p^{\mu(p)}&\lesssim_{\,p,k,r,\LipProp(D)} \delta^{k\mu(p)} \left(\int\limits_I \left(\|f_\tau\|_{L_p(D)}^{\mu(p)} + \int\limits_\delta^\infty \frac{\omega_{r}(f_\tau, D, s )_p^{\mu(p)}}{s^{k\mu(p)}}\frac{ds}{s}\right)^\frac{p}{\mu(p)} \,d\tau\right)^\frac{\mu(p)}{p}
		\\&\lesssim_{\, p} \delta^{k\mu(p)} \left(\int\limits_I \|f_{\tau}\|_{L_p(D)}^{p} + \left( \int\limits_\delta^\infty \frac{\omega_{r}(f_\tau, D, s )_p^{\mu(p)}}{s^{k\mu(p)}}\frac{ds}{s}\right)^\frac{p}{\mu(p)} \,d\tau\right)^\frac{\mu(p)}{p}
		\\  &= \delta^{k\mu(p)}  \left(\|f\|_{L_p(I\times D)}^{p} + \int\limits_I\left( \int\limits_\delta^\infty \frac{\omega_{r}(f_\tau, D, s )_p^{\mu(p)}}{s^{k\mu(p)}}\frac{ds}{s}\right)^\frac{p}{\mu(p)} \,d\tau\right)^\frac{\mu(p)}{p}
    \\
    &\lesssim_{\, p}  \delta^{k\mu(p)} \left(\|f\|_{L_p(I\times D)}^{\mu(p)} + \left(\int\limits_I\left( \int\limits_\delta^\infty \frac{\omega_{r}(f_\tau, D, s )_p^{\mu(p)}}{s^{k\mu(p)}}\frac{ds}{s}\right)^\frac{p}{\mu(p)} \,d\tau\right)^\frac{\mu(p)}{p} \right).
    \end{split}
\end{equation}
Since $\frac{p}{\mu(p)}\ge 1$, we can use Minkowski's inequality for integrals to get
\begin{equation}\label{Theorem Marchaud spacial modulus of smoothness - anisotropic proof step 3}
\begin{split}
	\Bigg(\int\limits_I \bigg( \int\limits_\delta^\infty &\frac{\omega_{r}(f_\tau, D, s )_p^{\mu(p)}}{s^{k\mu(p)}}\frac{ds}{s}\bigg)^\frac{p}{\mu(p)} \,d\tau\Bigg)^\frac{\mu(p)}{p} 
	\le \int\limits_\delta^\infty \Bigg(\int\limits_I \frac{\omega_{r}(f_\tau, D, s )_p^{p}}{s^{kp}}\,d\tau \Bigg)^\frac{\mu(p)}{p}\frac{ds}{s}
	\\
    &\le \int\limits_{\min(\delta, \delta_0)}^{\delta_0}  \Bigg(\int\limits_I  \omega_{r}(f_\tau, D, s )_p^{p}\,d\tau \Bigg)^\frac{\mu(p)}{p}\frac{ds}{s^{k\mu(p)+1}} + \int\limits^{\infty}_{\delta_0} \left(\int\limits_I \omega_{r}(f_\tau, D, s )_p^{p}\,d\tau \right)^\frac{\mu(p)}{p}\frac{ds}{s^{k\mu(p)+1}}.
\end{split}
\end{equation}
    For $s\le \delta_0(r):=\frac{1}{4r}$, \cref{Lemma equivalence averaged spacial modulus of smoothness} gives
\begin{align}\label{Theorem Marchaud spacial modulus of smoothness - anisotropic proof step 4}
	\left(\int\limits_I \omega_{r}(f_\tau, D, s )_p^{p}\,d\tau \right)^\frac{\mu(p)}{p} &\lesssim_{\,d,p,r,\LipProp(D)}  \left(\int\limits_I w_{r}(f_\tau, D, s )_p^{p}\,d\tau \right)^\frac{\mu(p)}{p} = w_{r,\bm{x}}(f, I\times D, s)_p^{\mu(p)}. 
\end{align}

In the last step, we have used the following identity 
\begin{align*}
	\int\limits_I w_{r}(f_\tau, D, s )_p^{p}\,dt =  &\frac{1}{(2\delta)^d} \int\limits_{\{|\bm{h}|\le s\}}\int\limits_I \|\Delta^r_{\bm{h}}f_\tau\|_{L_p(D_{r, \bm{h}})}^p \,d\tau \,d\bm{h}
	\\=  &\frac{1}{(2\delta)^d} \int\limits_{\{|\bm{h}|\le s\}}\int\limits_I \int\limits_{D_{r,\bm{h}}} |\Delta^r_{\bm{h}}f_\tau(\bm{z})|^p \,d\bm{z}\,d\tau \,d\bm{h}
	\\ =& \frac{1}{(2\delta)^d} \int\limits_{\{|\bm{h}|\le s\}}\int\limits_{I\times D_{r, \bm{h}}}|\Delta^r_{\bm{h},\bm{x}} f(\tau,\bm{z})|^p\,d(\tau,\bm{z})\,d\bm{h}
	\\ = & \frac{1}{(2\delta)^d} \int\limits_{\{|\bm{h}|\le s\}} \|\Delta^r_{\bm{h},\bm{x}}f\|_{L_p(I\times D_{r, \bm{h}})}^p \, d\bm{h} = w_{r,\bm{x}}(f, I\times D, s)_p^p, 
\end{align*}
where we have applied Fubini's theorem in the first and third step, and the identity \mbox{$\Delta^r_{\bm{h}}f_{\tau}(\bm{z})=\Delta^r_{\bm{h},\bm{x}} f(\tau,\bm{z})$} in the third step as well. Furthermore, with the estimate $\omega_{r}(f_\tau, D, s )_p\lesssim_{\,p,r} \|f_\tau\|_{L_p(D)}$ we derive
\begin{align}\label{Theorem Marchaud spacial modulus of smoothness - anisotropic proof step 5}
	\int\limits^{\infty}_{\delta_0} \left(\int\limits_I \omega_{r}(f_\tau, D, s )_p^{p}\,d\tau \right)^\frac{\mu(p)}{p}\frac{ds}{s^{k\mu(p)+1}}&\lesssim_{\, p,r}\|f\|_{L_p(I\times D)}^{\mu(p)} \int\limits_{\delta_0}^{\infty}\frac{ds}{s^{k\mu(p)+1}}\lesssim_{\, p,k,r}\|f\|_{L_p(I\times D)}^{\mu(p)}.
\end{align}
Combining \eqref{Theorem Marchaud spacial modulus of smoothness - anisotropic proof step 2}, \eqref{Theorem Marchaud spacial modulus of smoothness - anisotropic proof step 3}, \eqref{Theorem Marchaud spacial modulus of smoothness - anisotropic proof step 4}, and \eqref{Theorem Marchaud spacial modulus of smoothness - anisotropic proof step 5} we arrive at the assertion for $p<\infty$. The case $p=\infty$ works similarly with the usual changes; mainly exchanging some integrals with essential suprema and setting $\frac{p}{\mu(p)}$ to $1$. 
\end{proof}

\subsection{Jackson-type estimate}\label{subsect:Jackson-type estimate}

In this subsection, we proceed similar to \cite[Ch.~2]{DL04}. The first crucial part is to extend the construction of the (spatial) step function from \cite[Lem.~2.8]{DL04} on convex domains to arbitrary Lipschitz domains using the cone property from \cref{Remark Lipschitz domain}\ref{Remark Lipschitz domain - interior cone condition} while simultaneously adding a time-dependency. The idea of this application of the cone property (in space) has been roughly sketched in \cite[Lem.~4.4]{GM14} with respect to a certain subclass of Lipschitz domains, there called \enquote{reference stars}. The second part is to extend the domain convergence technique from \cite[Def.~2.11~and~Lem.~2.12]{DL04} from convex domains to bounded Lipschitz domains using \cref{lem:convergence_of_Lipschitz_domains}, which was applied in \cite[Thm.~1.4]{DL04} for proving the case $0<p<1$. Together, these will allow us  to establish \cref{thm:Jackson}, i.e., provide a uniform time-dependent Jackson's estimate for arbitrary Lipschitz domains sharing the same Lipschitz properties, not only convex ones. \\

We start  studying approximation with constants, i.e., polynomials of order $(1,1)$.

\begin{lem}\label{Jackson for r_1=r_2=1 lemma}
	Let $p\in (0,\infty]$ and $f\in L_p(I\times D)$. Then there is a constant $c\in \R$ such that
	\begin{align*}
		\|f - c|&|_{L_p(I\times D)}\lesssim_{\,p} \displaystyle\left(\frac{1}{|I|}\int\limits_{\{|h|\le|I|\}}\|\Delta_{h,t}^1\|_{L_p(I_{1,h}\times D)}^p\, dh + \frac{1}{|D|}\displaystyle \int\limits_{\{|\bm{h}|\le \diam(D)\}}\|\Delta_{\bm{h},\bm{x}}^1\|_{L_p(I\times D_{1,\bm{h}})}^p\, d\bm{h}\right)^\frac{1}{p}
	\end{align*}
	for $p\in (0,\infty)$ and 
	\begin{align*}
		\|f - c\|_{L_\infty(I\times D)} \le \omega_{1,t}(f,I\times D,|I|)_\infty  + \omega_{1,\bm{x}}(f,I\times D,\diam(D))_\infty
	\end{align*}		
	for $p=\infty$.
\end{lem}
\begin{proof}
	We first consider the case $p<\infty$ and define
	\begin{equation*}
		\varphi:I\times D\rightarrow \R, \,\, \varphi(s,\bm{y}):= \int\limits_{I\times D} |f(\tau,\bm{z})-f(s,\bm{y})|^p\,d(\tau,\bm{z}).
	\end{equation*}
	Then it exists $(s, \bm{y})\in I\times D$, such that
	\begin{equation*}
		\varphi(s,\bm{y})\le \frac{1}{|I\times D|}\int\limits_{I\times D} \varphi(\tau,\bm{z})\, d(\tau,\bm{z}).
	\end{equation*}
	We set $c:= f(s,\bm{y})$ and derive 
	\begin{align*}
		\|f-c\|^p_{L_p(I\times D)} = \varphi(s,\bm{y}) &\le \frac{1}{|I\times D|}\int\limits_{I\times D} \varphi(\tau,\bm{z})\, d(\tau,\bm{z})
		\\&= \frac{1}{|I\times D|}\int\limits_{I\times D}\int\limits_{I\times D} |f(\tilde{\tau}, \tilde{\bm{z}})-f(\tau,\bm{z})|^p\, d(\tau,\bm{z})\, d(\tilde{\tau}, \tilde{\bm{z}})
		\\&\lesssim_{\,p} \frac{1}{|I\times D|}\int\limits_{I\times D}\int\limits_{I\times D} |f(\tilde{\tau}, \tilde{\bm{z}})-f(\tau,\tilde{\bm{z}})|^p\, d(\tau,\bm{z})\, d(\tilde{\tau}, \tilde{\bm{z}})
		\\&\quad\quad +\: \frac{1}{|I\times D|}\int\limits_{I\times D}\int\limits_{I\times D} |f(\tau, \tilde{\bm{z}})-f(\tau,\bm{z})|^p\, \, d(\tau,\bm{z})\, d(\tilde{\tau}, \tilde{\bm{z}})
		\\&:= (I) + (II).
	\end{align*}
	
	The integrand of $(I)$ is independent of $\bm{z}$, thus with Fubini's theorem we can conclude
	\begin{align*}
		(I) = \frac{|D|}{|I\times D|} \int\limits_{I\times D} \int\limits_I|\Delta^1_{\tilde{\tau}-\tau, t}f(\tau,\tilde{\bm{z}})|^p\, d\tau \, d(\tilde{\tau},\tilde{\bm{z}})\, = \frac{1}{|I|}\int\limits_{\{|h|\le |I|\}}\|\Delta_{h,t}^1 f\|^p_{L_p(I_{1,h}\times D)} \,d\bm{h},
	\end{align*}		
	since $|I\times D|= |I\|D|$. Analogously, the integrand of $(II)$ is independent from $\tilde{\tau}$, which shows the assertion for $p<\infty $. Moreover, for $p=\infty$ and $c\in \R$ with $\essinf\limits_{(s,\bm{y})\in I\times D} f(s,\bm{y})\le c \le \esssup\limits_{(s,\bm{y})\in I\times D} f(s,\bm{y})$ , we can estimate similarly when we replace the integrals by (essential) suprema.
\end{proof}

\begin{lem}\label{Jackson for r_1=r_2=1}
 \cref{thm:Jackson} holds true for $(r_1, r_2)=(1,1)$. In particular, for this choice of parameters there is no dependency on $\LipCov(D)$.
\end{lem}
\begin{proof}
    For $p=\infty $, \cref{Jackson for r_1=r_2=1 lemma} directly states this result. For $p<\infty$, first consider 
    \begin{equation*}
    	\frac{1}{|I|}\int\limits_{\{|h|\le |I|\}} \|\Delta_{h,t}^1 f\|^p_{L_p(I_{1,h}\times D)} \,dh\le \frac{2\,|I|}{|I|}\omega_{1, t}(f, I\times D, |I|)_p^p= 2\,\omega_{1, t}(f, I\times D, |I|)_p^p.
	\end{equation*}		
    Further, be aware of the fact  that we can again scale to $\diam(D)=1$ without changing $\LipProp(D)$ and that the interior cone condition from \cref{Remark Lipschitz domain}\ref{Remark Lipschitz domain - interior cone condition}  implies especially that there is some $\bm{x}\in D$ and $r=r(d,\Lip(D),\LipDelta(D))$, such that $B_r(\bm{x})\subset D$. Then, 
    \begin{align*}
		\frac{1}{|D|}\int\limits_{\{|\bm{h}|\le \diam(D)\}} \|\Delta_{\bm{h},\bm{x}}^1 f\|^p_{L_p(I\times D_{1,\bm{h}})} \,d\bm{h} &\le \omega_{1, \bm{x}}(f, I\times D, \diam(D))_p^p \frac{2^d}{|B_r(\bm{x})|} \\&\lesssim_{\,d,r,\Lip(D),\LipDelta(D)} \omega_{1, \bm{x}}(f, I\times D, \diam(D))_p^p.
	\end{align*}
    Now the assertion follows by combining both estimates above with the result of \cref{Jackson for r_1=r_2=1 lemma} for $p<\infty$.
\end{proof}

For the proof of the Jackson's estimate for arbitrary $r_1, r_2\in\N$, we will need the following lemma concerning approximation with step functions.
    
\begin{lem}\label{Lemma_Approximation_step_functions}
	Let $p\in(0,\infty]$, $n \in \mathbb{N}$, and $f\in L_p(I\times D)$. Additionally assume $I=[0,1]$ and $\diam(D)=1$. Then there is a space-time step function
		\begin{equation*}
		\varphi = \sum\limits^{n}_{i=1}\sum\limits^{\mathcal{K}}_{k=1} \mathds{1}_{\big[\frac{i-1}{n},\frac{i}{n}\big)\times Q_k}c_{i,k}, 
	\end{equation*}
    with the following properties:
	\begin{enumerate}[label=(\roman*)]
		\item $Q_k$, $k=1,\dots,\mathcal{K}$, are cubes of side length $n^{-1}$ from the uniform grid on $\R^d$ such that $\mathcal{K}\lesssim_{\,d} n^d$,\label{Lemma_Approximation_step_functions_enum_2}
		\item $D\subset \bigcup\limits_{k=1}^{\mathcal{K}}Q_k$.
		\label{Lemma_Approximation_step_functions_enum_3}
		\item $\|f-\varphi\|_{L_p(I\times D)} \lesssim_{\,d,p, \Lip(D), \LipDelta(D)}  \omega_{1,t}(f, I\times D, n^{-1})_p + \omega_{1,\bm{x}}(f, I\times D, n^{-1})_p$ and
		\label{Lemma_Approximation_step_functions_enum_4}
		\item $\|\varphi\|_{L_p(\R^{d+1})}  \lesssim_{\,d,p, \Lip(D), \LipDelta(D)}  \|f\|_{L_p(I\times D)}$. \label{Lemma_Approximation_step_functions_enum_5}
	\end{enumerate}
\end{lem}

\begin{proof}
    We first consider the case   $f\ge 0$ and set $I_i:=\big[\frac{i-1}{n},\frac{i}{n}\big)$ for $i=1,\dots,n$. Further, one can surely translate $D$ without loss of generality such that it can be covered with cubes on the uniform grid of $[-1, 1]^d$ with sidelength $n^{-1}$. This means that we get $Q_k$, $k=1,\dots,\tilde{\mathcal{K}}\le (2n)^d$ of such cubes, whereby we require $|Q_k\cap D|>0$ for $k\in \tilde{\mathcal{K}}$. In particular, this implies that $\tilde{\mathcal{K}}< (2n)^d$ is possible. \\

\begin{minipage}{0.5\textwidth}
    By symmetric scaling at the center of $Q_k$ for every $k=1,\dots,\tilde{\mathcal{K}}$, we get more cubes $\tilde{Q}_k$, $k=1,\dots,\tilde{\mathcal{K}}$ with side length $3n^{-1}$.
For those, it holds
    \begin{equation}\label{Lemma_Approximation_step_functions_proof_1}
    	|\tilde{Q}_k\cap D|\gtrsim_{\, d,\Lip(D), \LipDelta(D)} n^{-d},
    \end{equation}
    which we can see as follows: Choose $x_k\in Q_k\cap D$ with corresponding cone $V_k$ (with cone tip $x_k$) depending on $\Lip(D)$ and $\LipDelta(D)$ according to \cref{Remark Lipschitz domain}\ref{Remark Lipschitz domain - interior cone condition} such that $V_k\subset D$. Then, as the cone is not too thin, we see that
    \begin{align*}
	|\tilde{Q}_k\cap D|\ge |\tilde{Q}_k\cap V_k| \gtrsim_{\, d, \Lip(D), \LipDelta(D)} n^{-d}
    \end{align*}
    holds. 
    \end{minipage}\hfill \begin{minipage}{0.45\textwidth}
    \begin{center}
    \includegraphics[height=7cm]{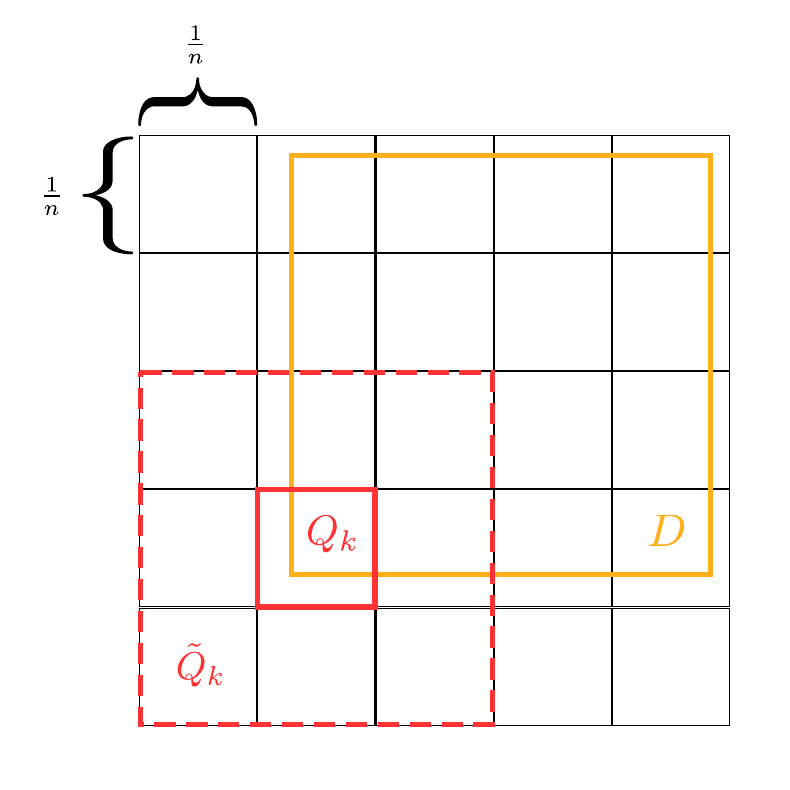}
    \end{center}
    \vspace{-0.8cm}
    \captionof{figure}{$Q_k$ and $\tilde{Q}_k$}
\end{minipage}\\[0.3cm]
    Now we complete the finite sequence $(Q_k)_{k=1,\dots,\tilde{\mathcal{K}}}$ with further cubes $(Q_k)_{k=\tilde{\mathcal{K}}+1,\dots,\mathcal{K}}$, $\mathcal{K}\lesssim_{\,d} n^{d} $, of side length $n^{-1}$ from the uniform grid such that
\begin{equation*}
	\bigcup\limits_{k=1}^\mathcal{K} Q_k = \bigcup\limits_{k=1}^{\tilde{\mathcal{K}}} \tilde{Q}_k.
\end{equation*} 

The set $(Q_k)_{k=1,\dots,\mathcal{K}}$ now fulfills \ref{Lemma_Approximation_step_functions_enum_2} and \ref{Lemma_Approximation_step_functions_enum_3}. Still, \ref{Lemma_Approximation_step_functions_enum_4} and \ref{Lemma_Approximation_step_functions_enum_5} remain to be shown. First, we consider $p<\infty$. For $i\in\{1,\dots,n\}$ and $k\in\{1,\dots,\tilde{\mathcal{K}}\}$ there is a constant $\tilde{c}_{i,k}$ according to \cref{Jackson for r_1=r_2=1 lemma} which allows us to estimate
\begin{align}\label{Lemma_Approximation_step_functions_proof_2}
	\int\limits_{I_{i}\times (\tilde{Q}_k\cap D)} | f(t,\bm{x})- \tilde{c}_{i,k}|^p\,d(t,\bm{x}) &\lesssim_{\,p} \frac{1}{|I_i|}\int\limits_{\{|h|\le |I_{i}|\}} \|\Delta^1_{h,t}\|^p_{L_p((I_i)_{1,h}\times (\tilde{Q}_k\cap D))}\,dh 
	\\ & \ \ \ + \frac{1}{|\tilde{Q}_k\cap D|}\int\limits_{\{|\bm{h}|\le \diam(\tilde{Q}_k\cap D)\}} \|\Delta^1_{\bm{h},\bm{x}}\|^p_{L_p(I_i\times (\tilde{Q}_k\cap D)_{1,\bm{h}})}\,d\bm{h}	\notag	
	\\&\le \frac{1}{|I_i|}\int\limits_{\{|h|\le |I_i|\}} \|\Delta^1_{h,t}\|^p_{L_p((I_i)_{1,h}\times (\tilde{Q}_k\cap D))}\,dh \notag
	\\ & \ \ \ + \frac{1}{|\tilde{Q}_k\cap D|}\int\limits_{\{|\bm{h}|\le 3n^{-1}\sqrt{d}\}} \|\Delta^1_{\bm{h},\bm{x}}\|^p_{L_p(I_i\times (\tilde{Q}_k\cap D)_{1,\bm{h}})}\,d\bm{h}.\notag
\end{align}

\begin{minipage}{0.45\textwidth}
For $Q_k$, $k=1,\dots,\mathcal{K}$, we define $\left(\tilde{Q}_k^{(j)}\right)_{j=1,\dots,J(k)}$, $J(k)\le 3^d$, as a subsequence of $(\tilde{Q}_k)_{k=1,\dots,\tilde{\mathcal{K}}}$ for which it holds true that
\begin{equation*}
	Q_k \subset \tilde{Q}_k^{(j)}, \quad j=1,\dots,J(k).
\end{equation*}
We denote the constants from \eqref{Lemma_Approximation_step_functions_proof_2}  that correspond to $I_i\times \tilde{Q}_k^{(j)}$ as $\tilde{c}_{i,k}^{(j)}$ and set
\begin{equation*}
	c_{i,k} = \frac{1}{J(k)}\sum\limits_{j=1}^{J(k)} \tilde{c}_{i,k}^{(j)},\quad i=1,\dots,n ,\quad k=1,\dots,\mathcal{K}.
\end{equation*}
 \end{minipage}\hfill \begin{minipage}{0.45\textwidth}
\begin{center}
\includegraphics[height=7cm]{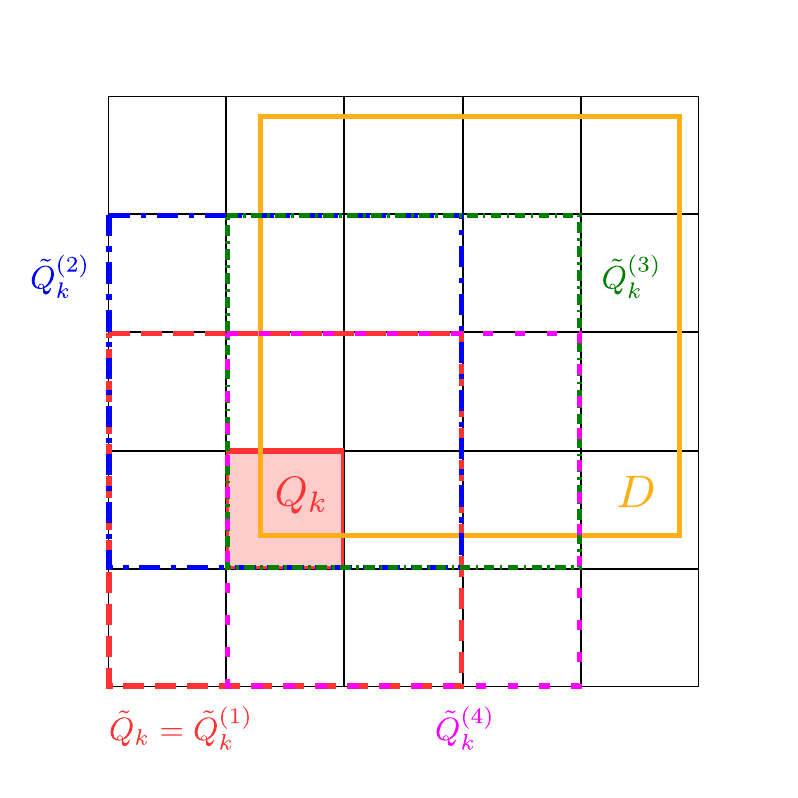}
\end{center}
\vspace{-0.5cm}
\captionof{figure}{Illustration of subsequence $\tilde{Q}^{(j)}_k$ with  $Q_k\subset \tilde{Q}^{(j)}_k$, $j=1,\dots, 4$.}
\end{minipage}\\[0.3cm]

The definitions of $\tilde{Q}_k$ for $k=1,\dots,\tilde{\mathcal{K}}$ as well as $\tilde{Q}_k^{(j)}$, $\tilde{c}_{i,k}$, and $ \tilde{c}_{i,k}^{(j)}$, for $i=1,\dots,n$, $k=1,\dots,\mathcal{K}$, and \mbox{$j=1,\dots,J(k)$}, in particular imply

\begin{align}\label{Lemma_Approximation_step_functions_proof_3}
\bigcup\limits_{k=1}^{\tilde{\mathcal{K}}} \tilde{Q}_k\times\left\{\tilde{c}_{i,k}\right\} = \bigcup\limits_{k=1}^{\mathcal{K}} \bigcup\limits_{j=1}^{J(k)} Q_k\times\left\{\tilde{c}_{i,k}^{(j)}\right\} ,\quad i=1,\dots,n. 
\end{align}

Further, we define $\varphi:=\sum\limits_{i=1}^n \sum\limits_{k=1}^{\mathcal{K}}\mathds{1}_{I_i\times Q_k}c_{i,k}$. Now we  derive

\begin{align}\label{Lemma_Approximation_step_functions_proof_4}
	\|f-\varphi\|_{L_p(I\times D)}^p &= \sum\limits_{i=1}^n \sum\limits_{k=1}^{\tilde{\mathcal{K}}} \,\int\limits_{I_i\times (Q_k \cap D)} |f(t,\bm{x})-c_{i,k}|^p\,d(t,\bm{x})\notag 
	\\&= \sum\limits_{i=1}^n \sum\limits_{k=1}^{\tilde{\mathcal{K}}} \,\int\limits_{I_i\times (Q_k \cap D)} \left|\frac{1}{J(k)}\sum\limits_{j=1}^{J(k)} \left(f(t,\bm{x})- \tilde{c}_{i,k}^{(j)}\right)\right|^p\,d(t,\bm{x})\notag
	\\  &\lesssim_{\,p,d} \sum\limits_{i=1}^n \sum\limits_{k=1}^{\tilde{\mathcal{K}}}\sum\limits_{j=1}^{J(k)}\,\int\limits_{I_i\times (Q_k \cap D)}\left|f(t,\bm{x}) - \tilde{c}_{i,k}^{(j)}\right|^p\,d(t,\bm{x})\notag
	\\&=\sum\limits_{i=1}^n \sum\limits_{k=1}^{\tilde{\mathcal{K}}} \,\int\limits_{I_i\times (\tilde{Q}_k \cap D)}|f(t,\bm{x}) - \tilde{c}_{i,k}|^p\, d(t,\bm{x}), 
\end{align}
where we have used that $|Q_k\cap D|=0$ for $k>\tilde{\mathcal{K}}$ in the first step, $J(k)\ge 1$ as well as the \mbox{(quasi-)}norm equivalence in $\R^{J(k)}\subset \R^{3^d}$ in the third, and \eqref{Lemma_Approximation_step_functions_proof_3} in the last step. We continue by combining \eqref{Lemma_Approximation_step_functions_proof_2} with \eqref{Lemma_Approximation_step_functions_proof_4} to obtain
\begin{align*}
	\|f-\varphi \|_{L_p(I\times D)}^p \lesssim_{\,p,d} &\sum\limits_{i=1}^n \sum\limits_{k=1}^{\tilde{\mathcal{K}}}\frac{1}{|I_i|}\int\limits_{\{|h|\le |I_i|\}} \|\Delta^1_{h,t}\|^p_{L_p((I_i)_{1,h}\times (\tilde{Q}_k\cap D))}\,dh 
	\\&+ \sum\limits_{i=1}^n \sum\limits_{k=1}^{\tilde{\mathcal{K}}}\frac{1}{|\tilde{Q}_k\cap D|}\int\limits_{\{|\bm{h}|\le 3n^{-1}\sqrt{d}\}} \|\Delta^1_{\bm{h},\bm{x}}\|^p_{L_p(I_i\times (\tilde{Q}_k\cap D)_{1,\bm{h}})}\,d\bm{h}
\\:=& \ \ \ (I)+ (II)	
\end{align*}
Using $|I_i|=n^{-1}$ and $\bigcup\limits_{i=1}^{n} \bigcup\limits_{k=1}^{\tilde{\mathcal{K}}} (I_i)_{1,h} \times (\tilde{Q}_k\cap D)\subset \left(\bigcup\limits_{i=1}^{n} I_i\right)_{1,h}\times \left(\bigcup\limits_{k=1}^{\tilde{\mathcal{K}}} \tilde{Q}_k\cap D\right) = I_{1,h} \times D$, we get 
\begin{align*}
	(I) =& \sum\limits_{i=1}^n \sum\limits_{k=1}^{\tilde{\mathcal{K}}}\frac{1}{|I_i|}\int\limits_{\{|h|\le |I_i|\}} \|\Delta^1_{h,t}\|^p_{L_p((I_i)_{1,h}\times (\tilde{Q}_k\cap D))}\,dh
	\\\le&\  n \int\limits_{\{|h|\le n^{-1}\}} \|\Delta^1_{h,t}\|^p_{L_p(I_{1,h}\times D)}\,dh \le 2\,\omega_{1,t}\left(f, I\times D, n^{-1}\right)_p^p.
\end{align*}

Similarly, for the spatial component it holds
\begin{align*}
	(II) &=\sum\limits_{i=1}^n \sum\limits_{k=1}^{\tilde{\mathcal{K}}}\frac{1}{|\tilde{Q}_k\cap D|}\int\limits_{\{|\bm{h}|\le 3n^{-1}\sqrt{d}\}} \|\Delta^1_{\bm{h},\bm{x}}\|^p_{L_p(I_i\times (\tilde{Q}_k\cap D)_{1,\bm{h}})}\,d\bm{h}
	\\ &\lesssim_{\,d,\Lip(D), \LipDelta(D)} \,n^d\int\limits_{\{|\bm{h}|\le 3n^{-1}\sqrt{d}\}} \|\Delta^1_{\bm{h},\bm{x}}\|^p_{L_p(I\times D_{1,\bm{h}})}\,d\bm{h}
	\\&\lesssim_{\, d} \omega_{1,\bm{x}}\left(f,I\times D,3n^{-1}\sqrt{d}\right)_p^p
    \lesssim_{\,d,p} \omega_{1,\bm{x}}\left(f,I\times D,n^{-1}\right)_p^p,
\end{align*}
due to \eqref{Lemma_Approximation_step_functions_proof_1}. This implies \ref{Lemma_Approximation_step_functions_enum_4}. Now we can turn to \ref{Lemma_Approximation_step_functions_enum_5}. Since $f\ge 0$, we especially have $\tilde{c}_{i,k}^{(j)},\, c_{i,k}\ge 0$ for any $i\in \{1,\dots, n\}$, $ k\in \{1,\dots, \mathcal{K}\}$, and $j\in \{1,\dots, J(k)\}$.

Thus, we obtain
\begin{align}\label{Lemma_Approximation_step_functions_proof_5}
	\|\varphi\|_{L_p(\R^{d+1})}^p &= \sum\limits_{i=1}^n \sum\limits_{k=1}^{\mathcal{K}} |I_i\times Q_k|\, c_{i,k}^p \notag 
	\\&= n^{-(d+1)} \sum\limits_{i=1}^{n} \sum\limits_{k=1}^{\mathcal{K}}\, \left(\frac{1}{J(k)}\sum\limits_{j=1}^{J(k)} \tilde{c}_{i,k}^{(j)}\right)^p\notag
	\\&\lesssim_{\,p,d}n^{-(d+1)} \sum\limits_{i=1}^{n} \sum\limits_{k=1}^{\mathcal{K}}\, \sum\limits_{j=1}^{J(k)} \left(\tilde{c}_{i,k}^{(j)}\right)^p\notag
	\\&\lesssim_{\,d,\Lip(D), \LipDelta(D)} \sum\limits_{i=1}^{n} \sum\limits_{k=1}^{\mathcal{K}} \sum\limits_{j=1}^{J(k)} \left(\tilde{c}_{i,k}^{(j)}\right)^p \left|I_i\times \left(\tilde{Q}_{k}^{(j)} \cap D\right)\right|\notag
	\\&\lesssim_{\,d} \sum\limits_{i=1}^{n} \sum\limits_{k=1}^{\mathcal{\tilde{K}}}  \left(\tilde{c}_{i,k}\right)^p \left|I_i\times \left(\tilde{Q}_{k} \cap D\right)\right|
	= \sum\limits_{i=1}^{n}\sum\limits_{k=1}^{\tilde{\mathcal{K}}}\sum\limits_{j=1}^{J(k)} \left(\tilde{c}_{i,k}^{(j)}\right)^p\left|I_i \times (Q_k\cap D)\right|,
\end{align}
where we again applied $J(k)\ge 1$ as well as the \mbox{(quasi-)}norm equivalence in $\R^{J(k)}\subset \R^{3^d} $ in the third step and \eqref{Lemma_Approximation_step_functions_proof_1} together with $|I_i| = n^{-1}$ in the fourth one. Further, we exploited that the summand $\left(\tilde{c}_{i,k}^{(j)}\right)^p\left|I_i \times (\tilde{Q}_k^{(j)}\cap D)\right|$ appears multiple but not more than $3^d$ times in the sum 
\begin{align*}
\sum\limits_{i=1}^{n} \sum\limits_{k=1}^{\mathcal{K}} \sum\limits_{j=1}^{J(k)} \left(\tilde{c}_{i,k}^{(j)}\right)^p \left|I_i\times \left(\tilde{Q}_{k}^{(j)} \cap D\right)\right|
\end{align*}
in the penultimate step and \eqref{Lemma_Approximation_step_functions_proof_3} in the last one. Another application of the \mbox{(quasi-)}norm equivalence in $\R^{J(k)}\subset \R^{3^d}$ together with \eqref{Lemma_Approximation_step_functions_proof_5} now leads to
\begin{align*}
	\|\varphi\|_{L_p(\R^{d+1})}^p  &\lesssim_{\,d,p,\Lip(D),\LipDelta(D)} \sum\limits_{i=1}^{n}\sum\limits_{k=1}^{\tilde{\mathcal{K}}}\left(\sum\limits_{j=1}^{J(k)} \tilde{c}_{i,k}^{(j)}\right)^p\left|I_i \times (Q_k\cap D)\right| 
	\\&=\sum\limits^n_{i=1} \sum\limits_{k=1}^{\tilde{\mathcal{K}}} J(k)^p\int\limits_{I_i\times (Q_k\cap D)}c_{i,k}^p\, d(t,\bm{x})
	\\&\lesssim_{\,d,p} \sum\limits^n_{i=1} \sum\limits_{k=1}^{\tilde{\mathcal{K}}} \int\limits_{I_i\times (Q_k\cap D)}c_{i,k}^p\, d(t,\bm{x}) 
	\\&= \|\varphi\|_{L_p(I\times D)}^p. 
\end{align*}
With this estimate, the \mbox{(quasi-)}subadditivity of $\|\cdot\|_{L_p(I\times D)}^p$, \ref{Lemma_Approximation_step_functions_enum_4}, and the possibility to estimate the modulus of smoothness of a function by its $\|\cdot\|_{L_p(I\times D)}$-\mbox{(quasi-)}norm, we can conclude
\begin{align*}
\|\varphi\|_{L_p(\R^{d+1})}^p &\lesssim_{\, d,p,\Lip(D), \LipDelta(D)}\|\varphi\|_{L_p(I\times D)}^p
    \lesssim_{\,p} \|f\|_{L_p(I\times D)}^p  + \|f-\varphi\|_{L_p(I\times D)}^p 
	\\& \lesssim_{\,d,p,\Lip(D), \LipDelta(D)} \|f\|_{L_p(I\times D)}^p  + \omega_{1,t}(f, I\times D, n^{-1})_p^p + \omega_{1,\bm{x}}(f, I\times D, n^{-1})_p^p \lesssim_{\,p} \|f\|_{L_p(I\times D)}^p.
\end{align*}
This shows \ref{Lemma_Approximation_step_functions_enum_5}. Therefore, we have shown the initial assertion for $p<\infty$ and $f\ge 0$. For arbitrary $f\in L_p(I\times D)$, we decompose $f$ classically in $f_+, f_-\ge 0$, such that $f=f_+-f_-$. This leads to step functions $\varphi_+, \varphi_-$, which fulfill the above assertions \ref{Lemma_Approximation_step_functions_enum_2}-\ref{Lemma_Approximation_step_functions_enum_5}. We now make an overlay of the space grid of $\varphi_+$ and $\varphi_-$ to get a common space grid. This fulfills the requirements \ref{Lemma_Approximation_step_functions_enum_2}-\ref{Lemma_Approximation_step_functions_enum_3} and $\varphi:=\varphi_+ - \varphi_-$ is itself a step function on this refined grid. Further, $\varphi$ fulfills \ref{Lemma_Approximation_step_functions_enum_4}-\ref{Lemma_Approximation_step_functions_enum_5}, which can be shown by simple calculations. Lastly, the case $p=\infty$ works very similar to the case $p<\infty$ and is therefore omitted.
\end{proof}

Further, we will need the following result which we will also prove in \cref{Appendix:Proof of Lemma equivalence averaged spacial modulus of smoothness}.

\begin{lem}\label{Lemma moduli of smoothness = 0 => anisotropic polynomials}
	Let $p\in(0,\infty]$ and $f\in L_p(I\times D)$. Then 
    \begin{enumerate}[label=(\roman*)]
        \item $\omega_{r_1,t}(f, I\times D, \delta_1)_p=0=\omega_{r_2,\bm{x}}(f, I\times D, \delta_2)_p$ for some $\delta_1, \delta_2\in(0,\infty)$ implies $f\in \Pi_{t,\bm{x}}^{r_1,r_2}(I\times D)$.\footnote{Here and in what follows we will write $f\in \Pi_{t,\bm{x}}^{r_1,r_2}(I\times D)$ to indicate that $f=\bar f$ almost everywhere with $\bar f\in \Pi_{t,\bm{x}}^{r_1,r_2}(I\times D)$.}
        \label{Lemma moduli of smoothness = 0 => anisotropic polynomials - 1}
        \item $f\in \Pi_{t,\bm{x}}^{r_1,r_2}(I\times D)$ implies $\omega_{r_1,t}(f, I\times D, \delta_1)_p=\omega_{r_2,\bm{x}}(f, I\times D, \delta_2)_p=0$ for any $\delta \in [0,\infty)$.\label{Lemma moduli of smoothness = 0 => anisotropic polynomials - 2}
    \end{enumerate}
\end{lem}

Now we have all the necessary ingredients to prove the preliminary step to \cref{thm:Jackson}.

\begin{lem}\label{lem:last_step_to_thm:Jackson}
    Let $p\in (0,\infty]$ and $f\in L_p(I\times D)$. Then, there exists $P\in \Pi^{r_1, r_2}_{t,\bm{x}}(I\times D)$ with
    \begin{align*}
        \|f-P\|_{L_p(I\times D)}\lesssim_{\,d,p,r_1, r_2, D\,} \,\omega_{r_1,t}(f, I\times D, |I|)_p + \omega_{r_2,\bm{x}}(f, I\times D, \diam(D))_p.
    \end{align*}
\end{lem}

\begin{rem}\label{rem:to_lem:last_step_to_thm:Jackson}
    Notice, that by proving \cref{lem:last_step_to_thm:Jackson}, we do not yet show the assertion of \cref{thm:Jackson}, where it is claimed that the involved constant depends on $D$ only through $\LipProp(D)$. 
\end{rem}

\begin{proof}
    For $(r_1, r_2)=(1,1)$, the assertion has already been shown in \cref{Jackson for r_1=r_2=1}, so we can assume $(r_1, r_2)\in \N^2\setminus \{(1,1)\}$. Additionally, we can assume $I=[0,1]$, $\diam(D)=1$, and $D\subset [0,1]^d$ without loss of generality, since the Jackson-type estimate and the Lipschitz properties of $D$ are scaling and translation invariant, due to the Jacobi transformation formula and the second part of \cref{lem:Behaviour under affine linear transformations}.

    Assume that the statement is false, i.e., there is no constant independent of $f$ such that~\eqref{thm:Jackson - equation} holds. Then there is a sequence of functions $(f_m)_{m\in\N}\subset L_p(I\times D)$ such that 
    \begin{equation}\label{Proof Jackson Inequaltiy Assumption 1}
	   \min\limits_{P\in \Pi^{r_1, r_2}_{t,\bm{x}}(I\times D)} \|f_m-P\|_{L_p(I\times D)} > m \left(\left(\omega_{r_1,t}(f_m, I\times D, |I|)_p + \omega_{r_2, \bm{x}}(f_m, I\times D, \diam(D)\right)_p\right).
    \end{equation}
    For any $m\in \N$, set $P_m:= \underset{P\in \Pi^{r_1, r_2}_{t,\bm{x}}(I\times D)}{\textup{argmin}} \|f_m-P\|_{L_p(I\times D)}$ and $g_m:=\|f_m - P_m \|_{L_p(I\times D)}^{-1} (f_m - P_m)$. Then, obviously, $\|g_m\|_{L_p(I\times D)} = 1$ and also
    \begin{align}\label{Proof Jackson Inequaltiy properties 1.5}
         \min\limits_{P\in \Pi^{r_1, r_2}_{t,\bm{x}}(I\times D)} \|g_m-P\|_{L_p(I\times D)} = 1.
    \end{align}
\thesis{
    Indeed, we can estimate
    $\min\limits_{P\in \Pi^{r_1, r_2}_{t,\bm{x}}(I\times D)} \|g_m-P\|_{L_p(I\times D)} \le \|g_m\|_{L_p(I\times D)} = 1$ by choosing $P=0$. On the other hand, if \mbox{$\min\limits_{P\in \Pi^{r_1, r_2}_{t,\bm{x}}(I\times D)} \|g_m-P\|_{L_p(I\times D)}<1$}, then there is $\tilde{P}_m$ with $\|g_m-\tilde{P}_m\|_{L_p(I\times D)}<1$. Multiplication with $\|f_m-P_m\|_{L_p(I\times D)}$ on both sides of this inequality gives 
    $$\|f_m - \underbrace{(P_m + \|f_m - P_m\|_{L_p(I\times D)}\tilde{P}_m)}_{\in \Pi^{r_1, r_2}_{t, \bm{x}(I\times D)}}\|_{L_p(I\times D)}<\|f_m - P_m\|_{L_p(I\times D)}= \min\limits_{P\in \Pi^{r_1, r_2}_{t,\bm{x}}(I\times D)} \|f_m-P\|_{L_p(I\times D)}$$ 
    (recall the definition of $g_m$ from above).
    But this is a contradiction to the choice of $P_m$. Therefore, it must hold that $\min\limits_{P\in \Pi^{r_1, r_2}_{t,\bm{x}}(I\times D)} \|g_m-P\|_{L_p(I\times D)} = 1$. 
}
    Further, for $dir\in\{t,\bm{x}\}$ and the corresponding $i\in\{1,2\}$, $\delta_i\in\{|I|, \diam(D)\}$, we can estimate
    \begin{align}\label{Proof Jackson Inequaltiy properties g_m 2}
        \omega_{r_i, dir}(g_m, I\times D, \delta_i)_p &= \|f_m - P_m \|_{L_p(I\times D)}^{-1}\omega_{r_i, dir}(f_m, I\times D,\delta_i)_p 
        \\&< \|f_m - P_m \|_{L_p(I\times D)}^{-1} \, \frac{\min\limits_{P\in \Pi^{r_1, r_2}_{t,\bm{x}}(I\times D)} \|f_m-P\|_{L_p(I\times D)}}{m}=\frac{1}{m},\notag
    \end{align}	
    where we have applied \cref{Lemma moduli of smoothness = 0 => anisotropic polynomials}\ref{Lemma moduli of smoothness = 0 => anisotropic polynomials - 2} in the first step, \eqref{Proof Jackson Inequaltiy Assumption 1} in the second, and the choice of $P_m$ in the third one. Now we can use this result together with the Marchaud inequalities from \cref{Theorem Marchaud temporal modulus of smoothness} and \cref{Theorem Marchaud spacial modulus of smoothness - anisotropic}, respectively, as well as the monotonicity of the moduli of smoothness in order to obtain
    \begin{align}\label{Proof Jackson Inequaltiy result with Marchaud}
	\omega_{1,dir}(g_m, I\times D,\delta)_p^{\mu(p)}&\lesssim_{\,d,p,r_1, r_2,\LipProp(D)} \delta^{\mu(p)}\left(\|g_m\|_{L_p(I\times D)}^{\mu(p)}+\int\limits_\delta^\infty \frac{\omega_{r_i,dir}(g_m,I\times D,s)_p^{\mu(p)}}{s^{\mu(p)+1}}\,ds\right)\notag
	\\ &< \delta^{\mu(p)}\left(1 + \int\limits_\delta^\infty \frac{s^{-\mu(p)-1}}{m^{\mu(p)}}\,ds\right)\lesssim_{\,p}  \delta^{\mu(p)} + m^{-\mu(p)},
\end{align}
    for $dir\in\{t,\bm{x}\}$, corresponding $i\in\{1,2\}$, and $\delta \in [0,\infty)$, if $r_i>1$. If $r_i=1$, the bound in \eqref{Proof Jackson Inequaltiy result with Marchaud} is already given as a consequence of \eqref{Proof Jackson Inequaltiy properties g_m 2} and the monotonicitiy of the moduli of smoothness. So for every $\varepsilon>0$ there exist 
    \begin{align*}
        \delta_0(\varepsilon)=\delta_0(\varepsilon, d,p,r_1, r_2, \LipProp(D))>0 \quad \text{and} \quad m_0(\varepsilon)=m_0(\varepsilon,p)\in\N
    \end{align*}    
    such that
    \begin{equation}\label{Proof Jackson Inequaltiy Konstruktion Treppenfunktion 1}
	   \omega_{1,t}(g_m, I\times D,\delta)_p+ \omega_{1,\bm{x}}(g_m, I\times D,\delta)_p\lesssim_{\,d,p,r_1, r_2,\LipProp(D)}  \varepsilon
    \end{equation}
    for $\delta \in (0, \delta_0(\varepsilon))$ and $m\in \N$ with $m\ge m_0(\varepsilon)$. Now we  use \cref{Lemma_Approximation_step_functions} and obtain step functions $\varphi_{m,n}$ for any $m,n\in \N$ satisfying 
    \begin{align}\label{Proof Jackson Inequaltiy Konstruktion Treppenfunktion 2}
    	\|g_m - \varphi_{m, n}\|_{L_p(I\times D)}\lesssim_{\,d,p,\Lip(D), \LipDelta(D)} \omega_{1,t}\left(g_m, I\times D,n^{-1}\right)_p+ \omega_{1,\bm{x}}\left(g_m, I\times D,n^{-1}\right)_p.
    \end{align}
    Setting $n_0(\varepsilon):= \left\lceil \delta_0(\varepsilon)^{-1}\right\rceil + 1$ then yields $n^{-1}<\delta_0(\varepsilon)$ for $n\ge n_0(\varepsilon)$. For such $n$ and $m\ge m_0(\varepsilon)$ we   thus derive  
    \begin{equation}\label{Proof Jackson Inequaltiy properties step function 1}
    	\|g_m - \varphi_{m, n}\|_{L_p(I\times D)}\lesssim_{\,d,p,r_1, r_2, \LipProp(D)} \varepsilon.
    \end{equation}
    by combining \eqref{Proof Jackson Inequaltiy Konstruktion Treppenfunktion 1} and \eqref{Proof Jackson Inequaltiy Konstruktion Treppenfunktion 2}. Further, the property \ref{Lemma_Approximation_step_functions_enum_5} of \cref{Lemma_Approximation_step_functions}  implies for the step functions 
    \begin{equation}\label{Proof Jackson Inequaltiy properties step function 2}
    	\|\varphi_{m,n}\|_{L_p(\R^{d+1})}\lesssim_{\,d,p, \Lip(D), \LipDelta(D)} \|g_m\|_{L_p(I\times D)}=1 .
    \end{equation}
    Now we can exploit the representation formula for the step functions. For any $m,n\in \N$, there is a set of cubes of side length $n^{-1}$ which we denote as $\left(Q_k^{(m,n)}\right)_{k\in \mathcal{K}_{m,n}}$, $\mathcal{K}_{m,n}\lesssim_{\,d}n^d$, such that 
    \begin{equation*}
	   \varphi_{m,n} =  \sum\limits^{n}_{i=1}\sum\limits^{\mathcal{K}_{m,n}}_{k=1} \mathds{1}_{\big[\frac{i-1}{n},\frac{i}{n}\big)\times Q_k^{(m,n)}}c_{i,k}^{(m,n)}.
    \end{equation*}
    This allows us to estimate
\begin{align*}
	\|\varphi_{m,n}\|_{L_\infty(I\times D)}&= \sup\limits_{\substack{i=1,\dots,n \\ k=1,\dots, \mathcal{K}_{m,n}}} |c_{i,k}^{(m,n)}| \le \left(\sum\limits^{n}_{i=1}\sum\limits^{\mathcal{K}_{m,n}}_{k=1} \left|c_{i,k}^{(m,n)}\right|^p\right)^\frac{1}{p}
	\\&=   \left(n^{d+1}\sum\limits^{n}_{i=1}\sum\limits^{\mathcal{K}_{m,n}}_{k=1} \left|c_{i,k}^{(m,n)}\right|^p |I_i^{(m,n)}\times Q_k^{(m,n)}|\right)^\frac{1}{p}\\
    &=   n^{\frac{d+1}{p}} \|\varphi_{m,n}\|_{L_p(\R^{d+1})} \le C(d,p, \Lip(D), \LipDelta(D))n^{\frac{d+1}{p}},
\end{align*}
for $p<\infty$, where \eqref{Proof Jackson Inequaltiy properties step function 2} has been applied in the last step. With $\frac{d+1}{\infty}:=0$, this result obviously also holds true for $p=\infty$. Now set $M:=n^{\frac{d+1}{p}}C(d,p, \Lip(D), \LipDelta(D))$ and consider the finite set $\Phi(\varepsilon)$ of all space-time step functions on the regular grid of $(d+1)$-dimensional cubes on $[0,1]^{d+1}$ of side-length $n^{-1}$, where on each of the space-time cubes values of the form $z\varepsilon|D|^{-\frac{1}{p}}$ with $z\in \mathbb{Z}$, \mbox{$0\le|z|\le \left\lceil\varepsilon^{-1}|D|^{\frac{1}{p}} M \right\rceil$}, are attained. Then, for any $n,m\in \N$, it follows
\begin{align}\label{Proof Jackson Inequaltiy properties step function 3}
\min\limits_{\varphi\in \Phi(\varepsilon)}\|\varphi - \varphi_{m,n}\|_{L_p(I\times D)}\le&
 \left(\:\int\limits_{I\times D}\left(\varepsilon|D|^{-\frac{1}{p}}\right)^p\,d(t,\bm{x})\right)^\frac{1}{p}=\varepsilon,
\end{align}
since the value on each step of $\varphi$ can be chosen with maximal difference of $\varepsilon|D|^{-\frac{1}{p}}$ to $\varphi_{m,n}$ on $I\times D$ as well as $0$ outside of $I\times D$, and $I\times D\subset [0,1]^{d+1}$ holds true. For $m\in\N$ and $\varepsilon>0$ we define $\varphi^\varepsilon_m:=\varphi_{m+m_0(\varepsilon), n_0(\varepsilon)}$ and $g_m^\varepsilon:= g_{m+m_0(\varepsilon)}$. According to the above calculation, there is always one $\varphi_\varepsilon\in\Phi(\varepsilon$) such that for every $m\in \N $,
\begin{align*}
	\|g_m^\varepsilon-\varphi_\varepsilon\|_{L_p(I\times D)}\lesssim_{\,p} \|g_m^\varepsilon-\varphi_m^\varepsilon\|_{L_p(I\times D)} + \|\varphi_m^\varepsilon-\varphi_\varepsilon\|_{L_p(I\times D)}\lesssim_{\,d,p,r_1, r_2, \LipProp(D)} \varepsilon,
\end{align*}
holds true, due to \eqref{Proof Jackson Inequaltiy properties step function 1} and \eqref{Proof Jackson Inequaltiy properties step function 3}. Now choose $\varepsilon_k:=k^{-1}$ and set $\varphi_{k}:=\varphi_{\varepsilon_k}$ as well as $g_{m,k}:=g^{\varepsilon_k}_{m}$ for $k\in\N$. Since we can assume the mapping $(\varepsilon\rightarrow m_0(\varepsilon))$ to be monotonically decreasing without loss of generality, $(g_{m,k+1})_{m\in\N}$ is a subsequence of $(g_{m,k})_{m\in\N}$ for any $k\in \N$. Therefore, in particular, 
    \begin{align}\label{Proof Jackson Inequaltiy properties step function 4.5}
        \|g_{m,l}-\varphi_k\|_{L_p(I\times D)}\lesssim_{\,d,p,r_1, r_2, \LipProp(D)} \varepsilon_k=k^{-1}
    \end{align}    
    if $l\ge k$ for any $m\in \N$. For $k, k'\in \N$ with $k'\ge k$ this means
    \begin{equation*}
        \|\varphi_{k}-\varphi_{k'}\|_{L_p(I\times D)} \lesssim_{\,p} \|\varphi_{k}-g_{k,k'}\|_{L_p(I\times D)} + \|\varphi_{k'}-g_{k,k'}\|_{L_p(I\times D)} \le k^{-1}+(k')^{-1}\lesssim k^{-1} \xrightarrow{k\rightarrow \infty} 0.
    \end{equation*}
    So we can conclude that $(\varphi_k)_{k\in \N}$ is a Cauchy sequence in $L_p(I\times D)$, which is a complete space, and therefore the sequence converges to some $\psi\in L_p(I\times D)$. Next, we will show some properties of $\psi$. Let $P\in \Pi_{t,\bm{x}}^{r_1, r_2}(I\times D)$ be arbitrary, then
    \begin{align*}
        \|P-\psi\|^{\mu(p)}_{L_p(I\times D)}&\ge \|P- g_{k,k}\|_{L_p(I\times D)}^{\mu(p)} - \|g_{k,k}-\varphi_k\|_{L_p(I\times D)}^{\mu(p)} - \|\varphi_k-\psi\|_{L_p(I\times D)}^{\mu(p)}
        \\&\gtrsim_{\, d, p, r_1, r_2, \LipProp(D)} 1 - k^{-\mu(p)} - \|\varphi_k-\psi\|_{L_p(I\times D)}^{\mu(p)}\xrightarrow{k\rightarrow\infty}1.
    \end{align*}
    by the reversed $\mu(p)$-triangle inequality, \eqref{Proof Jackson Inequaltiy properties 1.5}, \eqref{Proof Jackson Inequaltiy properties step function 4.5}, and $\varphi_k\xrightarrow{k\rightarrow\infty}\psi$ in $L_p(I\times D)$. This especially shows that the $L_p(I\times D)$-distance from $\psi$ to $\Pi^{r_1,r_2}_{t,\bm{x}}(I\times D)$ is positive. Next, we use the result from \cref{Rem:Difference operators} and \cref{Rem:Moduli of smoothness} for the estimation of the \mbox{(quasi-)}norm of differences and moduli of smoothness, respectively, by the \mbox{(quasi-)}norm of the functions itself, $\varphi_k\xrightarrow{k\rightarrow\infty}\psi$ in $L_p(I\times D)$, as well as the $\mu(p)$-subadditivity the moduli of smoothness, with $dir\in\{t,\bm{x}\}$ and the corresponding $i\in\{1,2\}$, $\delta_i\in\{|I|, \diam(D)\}$, to obtain
    \begin{align*}
        \omega_{r_i, dir}(\psi,I\times D, \delta_i)_p^{\mu(p)}&\le \lim\limits_{k\rightarrow\infty} \omega_{r_i, dir}(\varphi_k,I\times D, \delta_i)_p^{\mu(p)}
        \\&\le \lim\limits_{k\rightarrow\infty} \left(\omega_{r_i, dir}(g_{k,k},I\times D, \delta_i)_p^{\mu(p)} + \omega_{r_i, dir}(\varphi_k - g_{k,k},I\times D, \delta_i)_p^{\mu(p)}\right) \\
        &\le \Big(1+ 2^{\max\{r_1,r_2\}}\Big) \lim\limits_{k\rightarrow\infty} k^{-1}=0.
    \end{align*}
    In the last step, \eqref{Proof Jackson Inequaltiy properties g_m 2} has additionally been used. Therefore, \mbox{$\psi\in \Pi^{r_1,r_2}_{t,\bm{x}}(I\times D)$} according to \cref{Lemma moduli of smoothness = 0 => anisotropic polynomials}. This is a contradiction, so the assertion is shown.
\end{proof}

In order to show that there is a uniform Jackson estimate constant for all bounded Lipschitz domains with the same Lipschitz properties, we combine the convergence result from \cref{lem:convergence_of_Lipschitz_domains} with the technique from the proof \cite[Thm.~1.4]{DL04}, where the existence of a uniform Jackson estimate constant for convex domains has been proven in the stationary case. 

\begin{proof}[Proof of \cref{thm:Jackson}.]
    As in the proof of \cref{lem:last_step_to_thm:Jackson}, we can assume that $(r_1, r_2)\in \N^2\setminus \{(1,1)\}$. If the assertion does not hold true, there is a sequence of bounded Lipschitz domains $(D_m)_{m\in \N}$, which share the same Lipschitz properties, i.e., $\LipProp(D_{m_1})=\LipProp(D_{m_2})$ for $m_1, m_2\in \N$, and $g_m\in L_p(I\times D_m)$, $m\in \N$, with $I=[0,1]$, $\|g_m\|_{L_p(I\times D_m)}=1$, and 
    \begin{align}\label{thm:Jackson - proof equation 1}
        	1=\min\limits_{P\in \Pi^{r_1, r_2}_{t,\bm{x}}(I\times D_m)} \|g_m-P\|_{L_p(I\times D_m)} > m \left(\left(\omega_{r_1,t}(g_m, I\times D_m, |I|)_p + \omega_{r_2, \bm{x}}(g_m, I\times D_m, \diam(D_m)\right)_p\right).
    \end{align}
    In contrast to \eqref{Proof Jackson Inequaltiy Assumption 1}, the $g_m$-functions are now defined on generally distinct domains. By scaling, $\diam(D_m)=1$ and $D_m\subset [0,1]^d$ can be assumed (temporarily) for any $m\in \N$. Now making use of \cref{lem:convergence_of_Lipschitz_domains} allows us to assume, without loss of generality, that there exists a bounded Lipschitz domain $D$ with $\diam(D)=1$, $D\subset[0,1]^d$, and $|D\triangle D_m|\xrightarrow{m\rightarrow\infty}0$. 
    
    Choosing $\varepsilon>0$ and applying the same constructions as in the proof of \cref{lem:last_step_to_thm:Jackson}, now yields uniformly the same $m_0(\varepsilon)$ and $n_0(\varepsilon)$ for every $D_m$, since all the bounded Lipschitz domains $(D_m)_{m\in \N}$ share the same Lipschitz properties. Furthermore, it again yields a sequence of step functions $(\varphi_k)_{k\in \N}$, such that 
    \begin{align}\label{thm:Jackson - proof equation 2}
        \|g_{m,l}-\varphi_k\|_{L_p(I\times D_{m,l})}\lesssim_{\,d,p,r_1, r_2, \LipProp(D_1)}k^{-1}
    \end{align}
    for any $m,l,k\in \N$ with $l\ge k$, where $g_{m,l}:=g_{m+m_0(\varepsilon_l)}$ is well-defined on $D_{m,l}:=D_{m+m_0(\varepsilon_l)}$ with $\varepsilon_l:=\frac{1}{2l}$. 
    Exploiting  the convergence $|D\triangle D_{k,l}|\xrightarrow{l\rightarrow\infty}0$, we estimate
    \begin{align}\label{thm:Jackson - proof equation 3}
        \min\limits_{P\in \Pi^{r_1, r_2}_{t,\bm{x}}}\|\varphi_k-P\|_{L_p(I\times D)}&= \nonumber \lim\limits_{l\rightarrow \infty} \min\limits_{P\in \Pi^{r_1, r_2}_{t,\bm{x}}}\|\varphi_k-P\|_{L_p(I\times D_{k,l})}
        \\&\gtrsim_{\,p}  1 - \lim\limits_{l\rightarrow \infty} \|g_{k,l}-\varphi_k\|_{L_p(I\times D_{k,l})} \gtrsim_{\,d,p,r_1, r_2, \LipProp(D_1)} 1-k^{-1}\xrightarrow{k\rightarrow \infty}1,
    \end{align}
    by using the reversed $\mu(p)$-triangle inequality together with the first part of \eqref{thm:Jackson - proof equation 1} as well as \eqref{thm:Jackson - proof equation 2}. Similarly, for $h\in \R$ with $|h|\le 1 = |I|$ we obtain
    \begin{align*}
        \|\Delta_{h,t}^{r_1}\varphi_k\|_{L_p(I_{r_1, h}\times D)}&= \lim\limits_{l\rightarrow \infty} \|\Delta_{h,t}^{r_1}\varphi_k\|_{L_p(I_{r_1, h}\times D_{k,l})}
        \\&\lesssim_{\,p}\lim\limits_{l\rightarrow \infty} \|\Delta_{h,t}^{r_1}g_{k,l}\|_{L_p(I_{r_1, h}\times D_{k,l})} + \|\Delta_{h,t}^{r_1}(g_{k,l}-\varphi_k)\|_{L_p(I_{r_1, h}\times D_{k,l})}
        \\&\lesssim_{\,p, r_1}\lim\limits_{l\rightarrow \infty} \omega_{r_1,t}(g_{k,l},I\times D_{k,l},1)_p + \|g_{k,l}-\varphi_k\|_{L_p(I\times D_{k,l})}
        \\&\lesssim_{\,d,p,r_1, r_2, \LipProp(D_1)}\lim\limits_{l\rightarrow \infty} \frac{1}{k+m_0(\varepsilon_l)} + k^{-1}= k^{-1},
    \end{align*}
    where the estimate between the $L_p$-norms of difference operators of different order from \cref{Rem:Difference operators} has also been used in the third step as well as the right side of \eqref{thm:Jackson - proof equation 1} in the last one. The case of the spatial differences is very similar. There, for every $\bm{h}\in \R^d$ with $|\bm{h}|\le 1 = \diam(D) = \diam(D_{k,l})$, $k,l\in \N$, $|D_{r_2, \bm{h}}\triangle (D_{k,l})_{r_2, \bm{h}}|\xrightarrow{l\rightarrow\infty}0$ holds true, due to \cref{lem:convergence_of_Lipschitz_domains}. Now, almost the same calculation as before yields 
    \begin{align*}
        \|\Delta_{\bm{h},\bm{x}}^{r_2}\varphi_k\|_{L_p(I\times D_{r_2, \bm{h}})}\lesssim_{\,d,p,r_1, r_2, \LipProp(D_1)}k^{-1}.
    \end{align*}
    Taking the supremum over $h$ and $\bm{h}$ in the two above estimates, we derive
    \begin{align}\label{thm:Jackson - proof equation 4}
        \omega_{r_1, t}(\varphi_k, I\times D, |I|)_p+ \omega_{r_2, \bm{x}}(\varphi_k, I\times D, \diam(D))_p \lesssim_{\,d,p,r_1, r_2, \LipProp(D_1)} k^{-1}\xrightarrow{k\rightarrow \infty}0.
    \end{align}
    Lastly, \eqref{thm:Jackson - proof equation 3} and \eqref{thm:Jackson - proof equation 4} together are a contradiction to \cref{lem:last_step_to_thm:Jackson} applied to the Lipschitz cylinder $I\times D$, so the assertion holds true. If $D$ is additionally convex, then the dependency of the inequality constant on $\LipProp(D)$ can additionally be removed by combining the results of \cref{lem:Behaviour under affine linear transformations} and \cref{thm:Lipschitz domain - convex sets}.
\end{proof}

\section{Approximation of anisotropic Besov functions}\label{Section:Approximation_with_Besov_functions}

Now we will define function spaces as subspaces of $L_p(I\times D)$-functions with some (local) regularity constraint measured by the (decay of) the temporal and spatial moduli of smoothness. Since this approach is very similar to the one for (isotropic) Besov spaces but allows for the possibility of different smoothness in time and space, we call these spaces \textbf{anisotropic Besov spaces}. As described in \cref{sect:Introduction}, there are similar approaches in the literature but they were not suited for our purposes.

\subsection{Anisotropic Besov spaces}

\begin{defi}\label{Definition anisotropic Besov space}
	Let $p,q\in (0,\infty]$, $s_1, s_2\in (0,\infty) $, $r_i:=\lfloor s_i \rfloor +1$, $i=1,2$, and put
\begin{align*}
	|f|_{B_{p,q}^{s_1, s_2}(I\times D)}:= 
		\begin{cases} 
			\left(\displaystyle{\int\limits_0^\infty} \delta^{-s_1 q} \omega_{r_1,t}(f, I\times D, \delta)_p^q\frac{d\delta}{\delta} + \displaystyle{\int\limits_0^{\infty}} \delta^{-s_2 q} \omega_{r_2,\bm{x}}(f, I\times D, \delta)_p^q\frac{d\delta}{\delta}\right)^\frac{1}{q},& q < \infty,
			\\\sup\limits_{\delta \in [0,\infty)} \delta^{-s_1} \omega_{r_1,t}(f, I\times D, \delta)_p + \sup\limits_{\delta \in [0,\infty)} \delta^{-s_2} \omega_{r_2,\bm{x}}(f, I\times D, \delta)_p,& q = \infty.
		\end{cases}
\end{align*}
Now we can define the corresponding anisotropic Besov space as 
\begin{align*}
	B^{s_1, s_2}_{p,q}(I\times D):=\left\{f\in L_p(I\times D)\: \big| \: \|f\|_{B_{p,q}^{s_1, s_2}}(I\times D):= \|f\|_{L_p(I\times D)} +|f|_{B_{p,q}^{s_1, s_2}(I\times D)}<\infty \right\}.
\end{align*}
\end{defi}

\begin{rem}\label{Rem - equivalence of Besov(quasi-)seminorms}
As in the isotropic case, $|f|_{B_{p,q}^{s_1, s_2}(I\times D)}$ is a \mbox{(quasi-)}seminorm and $\|f\|_{B_{p,q}^{s_1, s_2}(I\times D)}$ is a \mbox{(quasi-)}norm. Further, as in the isotropic case, one can show that the equivalences
    \begin{align*}
    &\left(\displaystyle{\int\limits_0^\infty}  \delta^{-s_1 q} \omega_{r_1,t}(f, I\times D, \delta)_p^q\frac{d\delta}{\delta} + \displaystyle{\int\limits_0^{\infty}} \delta^{-s_2 q} \omega_{r_2,\bm{x}}(f, I\times D, \delta)_p^q\frac{d\delta}{\delta}\right)^\frac{1}{q}
    \\&\qquad \sim_{\,p,q,s_1,s_2,|I|,\diam(D)}\left(\displaystyle{\int\limits_0^1} \delta^{-s_1 q} \omega_{r_1,t}(f, I\times D, \delta)_p^q\frac{d\delta}{\delta} + \displaystyle{\int\limits_0^{1}} \delta^{-s_2 q} \omega_{r_2,\bm{x}}(f, I\times D, \delta)_p^q\frac{d\delta}{\delta}\right)^\frac{1}{q}
    \\ &\qquad \sim_{\, p,q,s_1,s_2,a,b}\left(\sum\limits_{n=0}^{\infty} a^{ns_1q} \omega_{r_1,t}\left(f, I\times D, a^{-n}\right)_p^q+ \sum\limits_{n=0}^{\infty} b^{ns_2q} \omega_{r_2,\bm{x}}\left(f, I\times D, b^{-n}\right)_p^q\right)^\frac{1}{q}
    \end{align*}
    hold true for $q\in (0,\infty)$ and arbitrary $a,b\in(1,\infty)$. The case $q=\infty$ can be treated similarly with the usual modifications. Also the following embeddings hold true:
	\begin{enumerate}[label=(\roman*)]
		\item $B^{s_1, s_2}_{p,q_1}(I\times D)\hookrightarrow B^{s_1, s_2}_{p,q_2}(I\times D),\text{ if }q_1\le q_2$,
		\item $B^{s_1, s_2}_{p_1,q}(I\times D)\hookrightarrow B^{s_1, s_2}_{p_2,q}(I\times D),\text{ if }p_1\ge p_2\ \text{ and }\ |I\times D|<\infty$,
		\item $B^{s_1, s_2}_{p,q_1}(I\times D)\hookrightarrow B^{\tilde{s_1}, \tilde{s_2}}_{p,q_2}(I\times D),\text{ if }s_i> \tilde{s_i}, i=1,2$.
	\end{enumerate}
    All of these results can be proven following the procedure in \cite[Ch.~2.10]{DL93}. 
\end{rem}

If we choose $p=q$, we also obtain an equivalent \mbox{(quasi-)}norm if we use the averaged moduli of smoothness instead of the supremum versions.

\begin{lem}\label{Lemma equivalence of (quasi-)seminorms for polyhedral domains}
Let $q\in (0,\infty]$, $s_1, s_2 \in (0, \infty)$, and $r_i:=\lfloor s_i \rfloor +1$, $i=1,2$. Then
	\begin{align*}
	 |f|^{\bullet}_{B_{q,q}^{s_1,s_2}(I\times D)}:= \begin{cases} 
			\left(\displaystyle{\int\limits_0^\infty} \delta^{-s_1 q} \mathrm{w}_{r_1,t}(f, I\times D, \delta)_q^q\frac{d\delta}{\delta} + \displaystyle{\int\limits_0^{\infty}} \delta^{-s_2 q} \mathrm{w}_{r_2,\bm{x}}(f, I\times D, \delta)_q^q\frac{d\delta}{\delta}\right)^\frac{1}{q},& q < \infty,
			\\\sup\limits_{\delta \in [0,\infty)} \delta^{-s_1} \mathrm{w}_{r_1,t}(f, I\times D, \delta)_q + \sup\limits_{\delta \in [0,\infty)} \delta^{-s_2} \mathrm{w}_{r_2,\bm{x}}(f, I\times D, \delta)_q,& q = \infty
		\end{cases}
	\end{align*}
	is a \mbox{(quasi-)}seminorm equivalent to $|f|_{B_{q,q}^{s_1,s_2}(I\times D)}$ with equivalence constants only depending on $d,q,s_1, s_2$, and $\LipProp(D)$. 
\end{lem}
\begin{rem}
    The corresponding results for isotropic moduli are mentioned in \cite[p.~403]{BDDP02} and \cite[p.~2144]{GM14}, but no explicit proofs have been given there.
\end{rem}
\begin{proof}
    For $q=\infty$, there is nothing to show, since then $\omega_{r_i, dir}(\cdot, I\times D, \delta)_\infty=\mathrm{w}_{r_i, dir}(\cdot, I\times D, \delta)_\infty$ for $dir \in \{t, \bm{x}\}$, the corresponding $i\in \{1,2\}$, and $\delta\in [0,\infty)$. Therefore, let $q<\infty $ from now on. The \mbox{(quasi-)}seminorm properties of $|\cdot|^{\bullet}_{B_{q,q}^{s_1,s_2}(I\times D)}$ can be derived analogously to those of $|\cdot|_{B_{q,q}^{s_1,s_2}(I\times D)}$. Since the averaged moduli are always less or equal then the supremum versions, $ |\cdot|^{\bullet}_{B_{q,q}^{s_1,s_2}(I\times D)}\le |\cdot|_{B_{q,q}^{s_1,s_2}(I\times D)}$ must hold. For the other direction, we will first consider the spatial integral for $q<\infty$. Due to the first part of \cref{lem:Behaviour under affine linear transformations} and the invariance of $\LipProp(D)$ under scaling, we may additionally assume $\diam(D)=1$ without loss of generality. We split the spatial integral as follows:
	\begin{align}\label{Lemma equivalence of (quasi-)seminorms for polyhedral domains - equation 1}
		\int\limits_0^\infty \delta^{-s_2 q}\omega_{r_2,\bm{x}}(f, I\times D, \delta)_q^q  \frac{d\delta}{\delta} 
		=&\int\limits_0^{\frac{\LipDelta(D)}{4r_2}} \delta^{-s_2 q}\omega_{r_2,\bm{x}}(f, I\times D, \delta)_q^q \frac{d\delta}{\delta} + \int\limits_{\frac{\LipDelta(D)}{4r_2}}^1 \delta^{-s_2 q}\omega_{r_2,\bm{x}}(f, I\times D, \delta)_q^q \frac{d\delta}{\delta}
		\\ &\ + \int\limits_{1}^\infty \delta^{-s_2 q}\omega_{r_2,\bm{x}}(f, I\times D, \delta)_q^q \frac{d\delta}{\delta} := A_\omega + B_\omega + C_\omega. \notag
	\end{align}
	For $A_\omega$, we can simply apply \cref{Lemma equivalence averaged spacial modulus of smoothness} to obtain
	\begin{align}\label{Lemma equivalence of (quasi-)seminorms for polyhedral domains - equation 2}
		A_\omega\lesssim_{\, d,q,s_2,\LipProp(D)} \int\limits_0^{\frac{\LipDelta(D)}{4r_2}} \delta^{-s_2 q}\mathrm{w}_{r_2,\bm{x}}(f, I\times D, \delta)_q^q \frac{d\delta}{\delta}\le\int\limits_0^\infty \delta^{-s_2 q}\mathrm{w}_{r_2,\bm{x}}(f, I\times D, \delta)_q^q\frac{d\delta}{\delta}.
	\end{align}
	Further, we can use the monotonicity of the supremum modulus, see \cref{Rem:Moduli of smoothness}, to estimate $B_\omega$ from below
	\begin{align}\label{Lemma equivalence of (quasi-)seminorms for polyhedral domains - equation 3}
		B_\omega &\ge \omega_{r_2,\bm{x}}\left(f, I\times D, \frac{\LipDelta(D)}{4r_2}\right)_q^q \int\limits_{\frac{\LipDelta(D)}{4r_2}}^{1} \delta^{-s_2 q} \frac{d\delta}{\delta}\gtrsim_{\,q,s_2,\LipDelta(D)} \omega_{r_2,\bm{x}}\left(f, I\times D, \frac{\LipDelta(D)}{4r_2}\right)_q^q.
	\end{align}
	Similarly, one can give an upper bound for $C_\omega$, since $D_{r_2, \bm{h}}=\emptyset$ for $|h|\ge 1=\diam(D)$, i.e.,
	\begin{align}\label{Lemma equivalence of (quasi-)seminorms for polyhedral domains - equation 4}
	C_\omega = \omega_{r_2, \bm{x}}\left(f, I\times D,1\right)_q^q \int\limits_1^\infty \delta^{-s_2q} \frac{d\delta}{\delta} \lesssim_{\,q,s_2} \omega_{r_2, \bm{x}}\left(f,I\times D, \frac{\LipDelta(D)}{4r_2}\right)_q^q \lesssim_{\,q,s_2,\LipDelta(D)}  B_\omega,
	\end{align}
	where we have used the scaling properties of the modulus of smoothness from \cref{Rem:Moduli of smoothness} in the penultimate step as well as \eqref{Lemma equivalence of (quasi-)seminorms for polyhedral domains - equation 3} in the last one. Furthermore, we can estimate $B_\omega$ from above using the scaling properties of the modulus of smoothness and linear substitution $\tilde{\delta}=\frac{\delta \LipDelta(D)}{4r_2}$:
	\begin{align}\label{Lemma equivalence of (quasi-)seminorms for polyhedral domains - equation 6}
		B_\omega &= \int\limits_{\frac{\LipDelta(D)}{4r_2}}^{1} \delta^{-s_2 q}\omega_{r_2,\bm{x}}\left(f, I\times D, \frac{4r_2}{\LipDelta(D)}\frac{\delta \LipDelta(D)}{4r_2}\right)_q^q \frac{d\delta}{\delta}\lesssim_{\, q, s_2, \LipDelta(D)} \int\limits_{\frac{\LipDelta(D)}{4r_2}}^{1} \delta^{-s_2 q}\omega_{r_2,\bm{x}}\left(f, I\times D, \frac{\delta \LipDelta(D)}{4r_2}\right)_q^q \frac{d\delta}{\delta}
		\\&\lesssim_{\, q, s_2, \LipDelta(D)} \int\limits_{\left(\frac{\LipDelta(D)}{4r_2}\right)^2}^{\frac{\LipDelta(D)}{4r_2}} \tilde{\delta}^{-s_2 q}\omega_{r_2,\bm{x}}\left(f, I\times D, \tilde{\delta}\right)_q^q \frac{d\tilde{\delta}}{\tilde{\delta}} \le A_\omega.\notag
	\end{align}
		Now, \eqref{Lemma equivalence of (quasi-)seminorms for polyhedral domains - equation 1}, \eqref{Lemma equivalence of (quasi-)seminorms for polyhedral domains - equation 2}, \eqref{Lemma equivalence of (quasi-)seminorms for polyhedral domains - equation 4}, and \eqref{Lemma equivalence of (quasi-)seminorms for polyhedral domains - equation 6} imply $\int\limits_0^\infty \delta^{-s_2 q}\omega_{r_2,\bm{x}}(f, I\times D, \delta)_q^q \frac{d\delta}{\delta} \lesssim_{\, d,q,s_2, \LipProp(D)} \int\limits_0^\infty \delta^{-s_2 q}\mathrm{w}_{r_2,\bm{x}}(f, I\times D, \delta)_q^q \frac{d\delta}{\delta}$.
	In a similar manner, one can treat the temporal part, which lastly shows the assertion.
\end{proof}

\subsection{Whitney-type estimate}

First, we will need a result for polynomial scaling whose isotropic form is well-known and can be found, for example in \cite[Lem.~3.3]{GM14}.

\begin{lem}\label{lem:Polynomial scaling}
	Let $J$ be a finite interval and $S\subset \R^d$, $d\in \N$, a $d$-dimensional simplex. Additionally, assume $p,q\in(0,\infty]$ and $r_1, r_2\in\N_0$. Then it holds
	\begin{align*}
		\|P\|_{L_p(J\times S)}\sim_{\,d,p,q,r_1, r_2} |J\times S|^{\frac1p - \frac1q} \|P\|_{L_q(J\times S)}\quad \text{for all}\quad  P\in \Pi^{r_1, r_2}_{t,\bm{x}}(J\times S).
	\end{align*}
\end{lem}
\begin{proof}
	For $J=[0,1]$ and $S=S_{\standard}$ the $d$-dimensional reference simplex, i.e., 
	\begin{equation*}
		S_{\standard}=\left\{x\in\R^d \mid x_i\ge 0, \quad i=1,\dots,d, \quad \sum\limits_{i=1}^{d} x_i \le 1\right\},
	\end{equation*}
	this follows directly by $|[0,1]\times S_{\standard}|=|S_{\standard}|\sim_{\,d}1$ and the \mbox{(quasi-)}norm-equivalence in the finite dimensional vector space $\Pi^{r_1, r_2}_{t,\bm{x}}([0,1]\times S_{\standard})$. For general $J\times S$ as in the requirement, this is a consequence of the Jacobi transformation theorem with respect to a bijective, affine mapping $\varphi: J\times S\rightarrow [0,1]\times S_{\standard}$. 
\end{proof}

Now we can show the following precursor to \cref{thm:Whitney}, whose proof transfers the ideas of the proofs of \cite[Sect.~6.1,~p.~97-99]{DeV98} and \cite[Prop.~3.12]{AMS23} to our situation.

\begin{lem}\label{Lemma_Whitney_standard_simplex}
	Let $p,q\in(0,\infty]$, $s_1, s_2 \in(0,\infty)$, \mbox{$\frac{1}{\frac{1}{s_1}+\frac{d}{s_2}}-\frac{1}{q}+\frac{1}{p}>0$}, $r_1, r_2 \in\mathbb{N}$, $r_i>s_i$, $i=1,2$, and let $S$ be a $d$-dimensional simplex with $|S|=1$. Then, there exists some \mbox{$P=P(f)\in \Pi^{r_1, r_2}_{t,\bm{x}}([0,1]\times S)$} for every \mbox{$f\in B^{s_1, s_2}_{q,q}([0,1]\times S)$}, such that
	\begin{equation}\label{Lemma_Whitney_standard_simplexEquation}
		\|f-P\|_{L_p([0,1]\times S)}\lesssim_{\,d,p,q,s_1, s_2,\kappa_S} |f|_{B^{s_1, s_2}_{q,q}([0,1]\times S)}.
	\end{equation}
\end{lem}

\begin{proof}
	Since $\Pi^{r_1, r_2}_{t,\bm{x}}([0,1]\times S)\subset \Pi^{r_1', r_2'}_{t,\bm{x}}([0,1]\times S)$ if $r_i\le r_i'$, $i=1,2$, it is sufficient to show the assertion for $r_i=\lfloor s_i \rfloor +1$, $i=1,2 $. Let $f\in B^{s_1, s_2}_{q,q}([0,1]\times S)$ be arbitrary. Since $|[0,1]\times S|=|S|= 1$, Hölder's inequality yields
	\begin{equation}\label{help}
		\|f-P\|_{L_p([0,1]\times S)}\le |[0,1]\times S|^{\frac{1}{p}-\frac{1}{q}}\|f-P\|_{L_q([0,1]\times S)}=\|f-P\|_{L_q([0,1]\times S)},
	\end{equation}	
	for every $P\in \Pi^{r_1, r_2}_{t,\bm{x}}([0,1]\times S)$, if $p<q$. In particular, taking $P=0$ in \eqref{help} implies $f\in L_p([0,1]\times S) $.
	
	So we can assume $p\ge q$ without loss of generality and define an initial triangulation $\calT_0:=\{S\}$ of $S$. For $n\in \N$, let $\calT_n$ be the triangulation generated from $\calT_{n-1}$ through uniform bisection of all the simplices contained in $\calT_{n-1}$, starting with a bisection of an arbitrary but fixed edge of $S$. By bisection, we mean the more-dimensional bisection routine for well-labeled $d$-dimensional simplices from \cite{Mau95} and \cite{Tra97} as it has been summarized in \cite[Sect.~2 and the beginning of Sect.~4]{Ste08}. This procedure yields
	\begin{align}\label{Lemma_Whitney_standard_simplexproof_-1}
		|R|\sim_{\,d} 2^{-n}|S|=2^{-n}\quad \text{and}\quad \LipDelta(R)\sim_{\,\LipProp(R)}\diam(R)\sim_{\,d}2^{-\frac{n}{d}}\diam(S)\sim_{\, d, \kappa_S} 2^{-\frac{n}{d}}|S|=2^{-\frac{n}{d}}, 
	\end{align}
	for $ R\in \calT_n, n\in \N_0$, according to equation $(4.1)$ of \cite[Sect.~4]{Ste08}. Additionally, we define the temporal partition 
	\begin{align*}
		\mathcal{Z}_n:=\left\{\left[\frac{k}{2^n}, \frac{k+1}{2^n}\right]\ \bigg|\ k=0,\dots, 2^n-1\right\}, \quad n\in \N_0,
	\end{align*}
	which is the result of $n$ uniform bisections of the time interval $\mathcal{Z}_0:=\{[0,1]\}$. Further, let $a(n):=\left\lceil \frac{ns_2}{s_1 d} \right\rceil\in \N_0$ and  $\calP_n:= \{J\times R\mid J\in \mathcal{Z}_{a(n)}, \ R\in \calT_n \}$ for $ n\in \N_0$. Then $(\calP_n)_{n\in \N_0}$ is a sequence of space-time partitions of $[0,1]\times S$. For such $n$ and $J\times R\in\calP_n$, Jackson's inequality (beware of the convexity of $R$) implies the existence of a piecewise anisotropic polynomial $P_n$ with $P_n|_{J\times R}\in \Pi^{r_1,r_2}_{t,\bm{x}}(J\times R)$ and
	\begin{align}\label{Lemma_Whitney_standard_simplexproof_0}
		\|f-P_n\|_{L_q\left(J\times R\right)}\lesssim_{\, d,q,s_1, s_2, \mathrm{LipProp}(R)}\omega_{r_1,t}\left(f, J\times R, |J|\right)_q + \omega_{r_2,\bm{x}}\left(f, J\times R, \diam(R)\right)_q.
	\end{align}
    For the simplices $R\in \calT_n$, $n\in \N_0$, we know that $\LipProp(R)\sim_{\,d} \kappa_R$. Together, with $\kappa_R\sim_{\,d}\kappa_S$ due to properties of the used bisection routine, in particular, \cite[Thm.~2.1]{Ste08},
	\cref{Lemma equivalence averaged spacial modulus of smoothness}, and the right part of \eqref{Lemma_Whitney_standard_simplexproof_-1} this shows that, for every $n\in \N_0$ there is some $\delta_n \sim_{\,d, s_2, \kappa_S} 2^{-\frac{n}{d}}\sim_{\,d, s_2, \kappa_S}\diam(R)$, with $\delta_n\le \delta_0(r_2, \LipDelta(R))$ for any $J\times R\in \calP_n$, where we borrow the notation $\delta_0(r_2, \LipDelta(R))$ from \cref{Lemma equivalence averaged spacial modulus of smoothness}. In particular, $\delta_n$ is independent of $J\times R\in \calP_n$ for every $n\in \N_0$. Therefore, applying the scaling property of the moduli of smoothness from \cref{Rem:Moduli of smoothness} gives
	\begin{align}\label{Lemma_Whitney_standard_simplexproof_Number_needed_later_1}
		\|f-P_n\|_{L_q\left(J\times R\right)}&\lesssim_{\, d,q,s_1, s_2,\kappa_S}\omega_{r_1,t}\left(f, J\times R, \frac{2^{-\frac{ns_2}{s_1d}-1}}{4r_1}\right)_q + \omega_{r_2,\bm{x}}\left(f, J\times R, \delta_n\right)_q
		\\&\lesssim_{\, d,q,s_1,s_2,\kappa_S} \mathrm{w}_{r_1,t}\left(f, J\times R, \frac{2^{-\frac{ns_2}{s_1d}-1}}{4r_1}\right)_q + \mathrm{w}_{r_2,\bm{x}}\left(f, J\times R, \delta_n\right)_q,\notag
	\end{align}
	where we have used additionally that $|J|=2^{-a(n)}\in\left\{ 2^{-\frac{ns_2}{s_1d}-1}, 2^{-\frac{ns_2}{s_1d}}\right\}$.
	For $q<\infty$ and $n\in\mathbb{N}_0$, this means 
	\begin{align} \label{Lemma_Whitney_standard_simplexproof_Number_needed_later_2}
		\|f-P_n\|_{L_q([0,1]\times S)}^q &=\sum\limits_{J\times R\in \calP_n} \|f-P_n\|_{L_q\left(J\times R\right)}^q 
		\\ & \lesssim_{\, d,q,s_1, s_2,\kappa_S} \sum\limits_{J\times R\in \calP_n} \mathrm{w}_{r_1,t}\left(f, J\times R, \frac{2^{-\frac{ns_2}{s_1d}-1}}{4r_1}\right)_q^q + \mathrm{w}_{r_2,\bm{x}}\left(f, J\times R, \delta_n\right)_q^q\notag
		\\ & \le  \mathrm{w}_{r_1,t}\left(f, [0,1]\times S, \frac{2^{-\frac{ns_2}{s_1d}-1}}{4r_1}\right)_q^q + \mathrm{w}_{r_2,\bm{x}}\left(f, [0,1]\times S,  \delta_n\right)_q^q \notag 
		\\ & \le  \omega_{r_1,t}\left(f, [0,1]\times S, \frac{2^{-\frac{ns_2}{s_1d}-1}}{4r_1}\right)_q^q + \omega_{r_2,\bm{x}}\left(f, [0,1]\times S,  \delta_n\right)_q^q \notag 
		\\ &\lesssim_{\,d,q,s_1,s_2,\kappa_{S}}\omega_{r_1,t}\left(f, [0,1]\times S, 2^{-  \frac{ns_2}{s_1d} }\right)_q^q + \omega_{r_2,\bm{x}}\left(f, [0,1]\times S,  2^{-\frac{n}{d}}\right)_q^q,\notag
	\end{align}	
	where we applied the subadditivity of the averaged moduli of smoothness in the third step and the monotonicity as well as the scaling property of the moduli of smoothness from \cref{Rem:Moduli of smoothness} in the fifth one.
	For $q=\infty$, we similarly get a corresponding result. Together, this yields 
	\begin{equation}\label{Lemma_Whitney_standard_simplexproof_1}
		\|f-P_n\|_{L_q([0,1]\times S)}\lesssim_{\, d,q,s_1,s_2,\kappa_S} \, \omega_{r_1,t}\left(f, [0,1]\times S, 2^{-\frac{ns_2}{s_1d}}\right)_q + \omega_{r_2,\bm{x}}\left(f, [0,1]\times S, 2^{-\frac{n}{d}}\right)_q,
	\end{equation}
	for any $q\in(0,\infty]$ and $n\in\mathbb{N}_0$. It holds \mbox{$\omega_{r_1,t}\left(f, [0,1]\times S,2^{-\frac{ns_2}{s_1d}}\right)_q, \omega_{r_2,\bm{x}}\left(f, [0,1]\times S,2^{-\frac{n}{d}}\right)_q\rightarrow 0$} for \mbox{$n\rightarrow \infty$}, since $f\in B^{s_1,s_2}_{q,q}([0,1]\times S)$. Together with \eqref{Lemma_Whitney_standard_simplexproof_1}, this implies $P_n\rightarrow f$ in $L_q([0,1]\times S)$ for $n\rightarrow \infty$. Now we estimate 
	\begin{align*}
		 \|P_n - P_{n+1}\|_{L_p([0,1]\times S)}^p & = \sum\limits_{J\times R \in \calP_{n+1}}  \|P_n - P_{n+1}\|_{L_p\left(J\times R\right)}^p
		 \\&\lesssim_{\, d,p, q, s_1, s_2} \sum\limits_{J\times R \in \calP_{n+1}}  \left|J\times R\right|^{\left(1-\frac{p}{q}\right)} \|P_n - P_{n+1}\|_{L_q\left(J\times R\right)}^p
		 \\&\lesssim_{\,p,q} \left(2^{-n\frac{s_2}{s_1d}}\, 2^{-n}\right)^{\left(1-\frac{p}{q}\right)} \sum\limits_{J\times R\in \calP_{n+1}}  \|P_n - P_{n+1}\|_{L_q\left(J\times R\right)}^p
		 \\&\le (2^{-n})^{\left(\frac{s_2}{s_1d}+1\right)\left(1-\frac{p}{q}\right)}\left(\sum\limits_{J\times R\in \calP_{n+1}}\|P_n - P_{n+1}\|_{L_q\left(J\times R\right)}^q\right)^\frac{p}{q}
		 \\&
		 \\& = (2^{-n})^{\left(\frac{s_2}{s_1d}+1\right)\left(1-\frac{p}{q}\right)} \|P_n - P_{n+1}\|_{L_q([0,1]\times S)}^p
	\end{align*}
	for $p<\infty$, by \cref{lem:Polynomial scaling} in the second step and the embedding $\ell^{\frac{q}{p}}(\mathbb{N}_0)\hookrightarrow \ell^{1}(\mathbb{N}_0)$\footnote{By $\ell^t(\N_0)$, $t\in (0,\infty]$, we denote the classical sequence spaces of $t$-summable infinite sequences with the corresponding \mbox{(quasi-)}norm $\|\cdot\|_{\ell^t(\N_0)}$. }, due to $q\le p$, in the fourth one. For $p=\infty$, we analogously derive
	\begin{equation*}
		\|P_n - P_{n+1}\|_{L_\infty([0,1]\times S)}\lesssim_{\, d,p,q,s_1,s_2} (2^{-n})^{\left(\frac{s_2}{s_1d}+1\right)\left(-\frac{1}{q}\right)}\|P_n - P_{n+1}\|_{L_q([0,1]\times S)}.
	\end{equation*}
	So, for arbitrary $p,q\in(0,\infty]$ with $q\le p$, it holds
	\begin{equation*}
		\|P_n - P_{n+1}\|_{L_p([0,1]\times S)}\lesssim_{\, d,p,q,s_1,s_2} (2^{-n})^{\left(\frac{s_2}{s_1d}+1\right)\left(\frac{1}{p}-\frac{1}{q}\right)}\|P_n - P_{n+1}\|_{L_q([0,1]\times S)}.
	\end{equation*}
	We use this to obtain
	\begin{align}\label{Lemma_Whitney_standard_simplexproof_2} 
		\sum\limits_{n=0}^\infty& \|P_n - P_{n+1}\|_{L_p([0,1]\times S)}^{\mu(p)} \notag
		\\&\lesssim_{\, d,p,q,s_1,s_2} \sum\limits_{n=0}^\infty (2^{-n})^{\left(\frac{s_2}{s_1d}+1\right)\left(\frac{1}{p}-\frac{1}{q}\right)\mu(p)}\|P_n - P_{n+1}\|_{L_q([0,1]\times S)}^{\mu(p)} \notag
		\\&\lesssim \sum\limits_{n=0}^\infty (2^{-n})^{\left(\frac{s_2}{s_1d}+1\right)\left(\frac{1}{p}-\frac{1}{q}\right)\mu(p)}\|f - P_n\|_{L_q([0,1]\times S)}^{\mu(p)}
		\nonumber
		\\&\lesssim_{\, d,p,q,s_1,s_2,\kappa_S} \sum\limits_{n=0}^\infty (2^{-n})^{\left(\frac{s_2}{s_1d}+1\right)\left(\frac{1}{p}-\frac{1}{q}\right)\mu(p)}\left(\omega_{r_1,t}\left(f, [0,1]\times S, 2^{-\frac{ns_2}{s_1 d}}\right)_q^{\mu(p)} + \omega_{r_2,\bm{x}}\left(f, [0,1]\times S, 2^{-\frac{n}{d}}\right)_q^{\mu(p)}\right)
	\end{align}
	by the $\mu(p)$-triangle inequality and \eqref{Lemma_Whitney_standard_simplexproof_1}. For $q\le \mu(p)$, it holds
	\begin{align}\label{Lemma_Whitney_standard_simplexproof_3} 
		&\sum\limits_{n=0}^\infty (2^{-n})^{\left(\frac{s_2}{s_1 d}+1\right)\left(\frac{1}{p}-\frac{1}{q}\right)\mu(p)}\left(\omega_{r_1,t}\left(f, [0,1]\times S, 2^{-\frac{ns_2}{s_1 d}}\right)_q^{\mu(p)} + \omega_{r_2,\bm{x}}\left(f, [0,1]\times S, 2^{-\frac{n}{d}}\right)_q^{\mu(p)}\right)\notag
		\\ & = \sum\limits_{n=0}^\infty (2^{-n})^{\left(\frac{s_2}{s_1 d}+1\right)\left(\frac{1}{\frac{1}{s_1}+ \frac{d}{s_2}} + \frac{1}{p}-\frac{1}{q}\right)\mu(p)}2^{n\frac{s_2}{d}\mu(p)}\left(\omega_{r_1,t}\left(f, [0,1]\times S, 2^{-\frac{ns_2}{s_1 d}}\right)_q^{\mu(p)} + \omega_{r_2,\bm{x}}\left(f, [0,1]\times S, 2^{-\frac{n}{d}}\right)_q^{\mu(p)}\right)\notag
		\\ & \le  \sum\limits_{n=0}^\infty 2^{n\frac{s_2}{d}\mu(p)}\left(\omega_{r_1,t}\left(f, [0,1]\times S, 2^{-\frac{ns_2}{s_1d}}\right)_q^{\mu(p)} + \omega_{r_2,\bm{x}}\left(f, [0,1]\times S, 2^{-\frac{n}{d}}\right)_q^{\mu(p)}\right)\notag
		\\ &\lesssim_{\, p,q} \left(\sum\limits_{n=0}^\infty 2^{n\frac{s_2}{d}q}\left(\omega_{r_1,t}\left(f, [0,1]\times S, 2^{-\frac{ns_2}{s_1d}}\right)_q^{q} + \omega_{r_2,\bm{x}}\left(f, [0,1]\times S, 2^{-\frac{n}{d}}\right)_q^{q} \right)\right)^{\frac{\mu(p)}{q}}\notag
		\\&=\left(\sum\limits_{n=0}^\infty \left(2^{\frac{s_2}{s_1 d}}\right)^{ns_1q}\omega_{r_1,t}\left(f,[0,1]\times S, \left(2^{\frac{s_2}{s_1 d}}\right)^{-n} \right)_q^q+ \sum\limits_{n=0}^{\infty}\left(2^{\frac{1}{d}}\right)^{ns_2q}\omega_{r_2, \bm{x}}\left(f, [0,1]\times S, \left(2^{\frac{1}{d}}\right)^{-n}\right)_q^q\right)^\frac{\mu(p)}{q}\notag
	 	\\&\lesssim_{\,d,p,q,s_1,s_2}|f|_{B^{s_1,s_2}_{q,q}([0,1]\times S)}^{\mu(p)}, 
	\end{align}
	due to $2^{-n\frac{s_2}{d}} = 2^{-n\left(\frac{s_2}{s_1 d}+1\right)\frac{1}{\frac{1}{s_1}+ \frac{d}{s_2}}}$, $\frac{1}{\frac{1}{s_1}+\frac{d}{s_2}}+\frac{1}{p}-\frac{1}{q}> 0$ in the second step, the embedding \mbox{$\ell^{\frac{q}{\mu(p)}}(\mathbb{N}_0)\hookrightarrow \ell^1(\mathbb{N}_0)$} in the third, and \cref{Rem - equivalence of Besov(quasi-)seminorms} in the last one. For $q>\mu(p)$, let $t\in[1,\infty)$ be the conjugate Hölder exponent to $\frac{q}{\mu(p)}$. Thus, applying Hölder's inequality implies
		\begin{align}\label{Lemma_Whitney_standard_simplexproof_4}
	&\sum\limits_{n=0}^\infty  \left(2^{-n}\right)^{\left(\frac{s_2}{s_1d}+1\right)\left(\frac{1}{\frac{1}{s_1}+\frac{d}{s_2}}+\frac{1}{p}-\frac{1}{q}\right)\mu(p)} 2^{n\frac{s_2}{d}\mu(p)}\left(\omega_{r_1,t}\left(f, [0,1]\times S, 2^{-\frac{ns_2}{s_1 d}}\right)_q^{\mu(p)} + \omega_{r_2,\bm{x}}\left(f, [0,1]\times S, 2^{-\frac{n}{d}}\right)_q^{\mu(p)}\right)\notag
		\\&\le  \left[\sum\limits_{n=0}^\infty \left(2^{-t\left(\frac{s_2}{s_1d}+1\right)\left(\frac{1}{\frac{1}{s_1}+\frac{d}{s_2}}+\frac{1}{p}-\frac{1}{q}\right)\mu(p)}\right)^n ~ \right]^\frac{1}{t}\notag
		\\ &~~~~~~\times\left|\left|\left(2^{n\frac{s_2}{d}\mu(p)}\left(\omega_{r_1,t}\left(f, [0,1]\times S, 2^{-\frac{ns_2}{s_1 d}}\right)_q^{\mu(p)} + \omega_{r_2,\bm{x}}\left(f, [0,1]\times S, 2^{-\frac{n}{d}}\right)_q^{\mu(p)}\right)\right)_{n\in\mathbb{N}_0}\right|\right|_{\ell^\frac{q}{\mu(p)}(\mathbb{N}_0)}\notag
		\\ &\lesssim_{\,d,p,q,s_1,s_2} \left|\left|\left(2^{n\frac{s_2}{d}\mu(p)}\left(\omega_{r_1,t}\left(f, [0,1]\times S, 2^{-\frac{ns_2}{s_1d}}\right)_q^{\mu(p)} + \omega_{r_2,\bm{x}}\left(f, [0,1]\times S, 2^{-\frac{n}{d}}\right)_q^{\mu(p)}\right)\right)_{n\in\mathbb{N}_0}\right|\right|_{\ell^\frac{q}{\mu(p)}(\mathbb{N}_0)}\notag
        \\ & \lesssim_{\, d,q}\left(\sum\limits_{n=0}^\infty \left(2^{\frac{s_2}{s_1 d}}\right)^{ns_1q}\omega_{r_1,t}\left(f,[0,1]\times S, \left(2^{\frac{s_2}{s_1 d}}\right)^{-n} \right)_q^q+ \sum\limits_{n=0}^{\infty}\left(2^{\frac{1}{d}}\right)^{ns_2q}\omega_{r_2, \bm{x}}\left(f, [0,1]\times S, \left(2^{\frac{1}{d}}\right)^{-n}\right)_q^q\right)^\frac{\mu(p)}{q}\notag
		\\ &\lesssim_{\,d,p,q,s_1,s_2} |f|_{B^{s_1,s_2}_{q,q}([0,1]\times S)}^{\mu(p)}. 
	\end{align}
	Here, the geometric sum converges due to $-t\left(\frac{s_2}{s_1}+d\right)\left(\frac{1}{\frac{1}{s_1}+ \frac{d}{s_2}}+\frac{1}{p}-\frac{1}{q}\right)\mu(p)<0$. Now, using the embedding $\ell^{\mu(p)}(\mathbb{N}_0)\hookrightarrow \ell^{1}(\mathbb{N}_0)$ as well as inserting \eqref{Lemma_Whitney_standard_simplexproof_3} and \eqref{Lemma_Whitney_standard_simplexproof_4}, respectively, into \eqref{Lemma_Whitney_standard_simplexproof_2} yields 
	\begin{equation}\label{Lemma_Whitney_standard_simplexproof_5}
		\sum\limits_{n=0}^\infty \|P_n - P_{n+1}\|_{L_p([0,1]\times S)} \le \left(\sum\limits_{n=0}^\infty \|P_n - P_{n+1}\|_{L_p([0,1]\times S)}^{\mu(p)} \right)^\frac{1}{\mu(p)} \lesssim_{\,d,p,q,s_1,s_2, \kappa_S} |f|_{B^{s_1,s_2}_{q,q}([0,1]\times S)}<\infty.
	\end{equation} 
	Therefore, $(P_n)_{n\in\mathbb{N}_0}$ is a Cauchy sequence in $L_p([0,1]\times S)$. This space is complete such that $P_n\rightarrow g$ as $n\rightarrow \infty$ in $L_p([0,1]\times S)$ for some $g\in L_p([0,1]\times S)$. Now $p\ge q$ and $|[0,1]\times S|=1<\infty$ leads to the fact that $P_n\rightarrow g$ for $n\rightarrow \infty$ in $L_q([0,1]\times S)$ also holds true, since then $L_p([0,1]\times S)\hookrightarrow L_q([0,1]\times S) $. The $L_q([0,1]\times S)$-limit is unique and, thus, it holds $f=g$ almost everywher and $P_n\rightarrow f$ in $L_p([0,1]\times S)$ for $n\rightarrow \infty$. Note that this especially shows $f\in L_p([0,1]\times S)$. The subadditivity of $\|\cdot\|_{L_p([0,1]\times S)}^{\mu(p)}$ then implies
	\begin{equation*}
		\|f-P_0\|^{\mu(p)}_{L_p([0,1]\times S)}= \left|\left|\lim\limits_{n\rightarrow \infty}P_n-P_0\right|\right|^{\mu(p)}_{L_p([0,1]\times S)}\le \sum\limits_{n=0}^\infty \|P_n - P_{n+1}\|^{\mu(p)}_{L_p([0,1]\times S)}.
	\end{equation*}
	Combining this with \eqref{Lemma_Whitney_standard_simplexproof_5} yields \eqref{Lemma_Whitney_standard_simplexEquation}. 
\end{proof}

The following lemma generalizes the above one to arbitrary simplices.

\begin{lem}\label{Lemma_Whitney_any_simplex}
	Let $J$ be a bounded interval, $S\subset \R^d$, $d\in\N $,  a $d$-dimensional simplex, $p,q\in(0,\infty]$, $s_1, s_2 \in(0,\infty)$, \mbox{$\frac{1}{\frac{1}{s_1}+\frac{d}{s_2}}-\frac{1}{q}+\frac{1}{p}>0$}, $r_1, r_2 \in\mathbb{N}$, and $r_i>s_i$, $i=1,2$. Then, for every $f\in B^{s_1, s_2}_{q,q}(J\times S)$ there is some $P=P(f)\in \Pi^{r_1, r_2}_{t,\bm{x}}(J\times S)$ such that
	\begin{equation}\label{Lemma_Whitney_any_simplex__equation}
		\|f-P\|_{L_p(J\times S)}\lesssim_{\, d,p,q,s_1,s_2,\kappa_S} |J\times S|^{\frac{1}{p}-\frac{1}{q}} \max\left(|J|^{s_1}, |S|^{\frac{s_2}{d}}\right)  |f|_{B^{s_1, s_2}_{q,q}(J\times S)}.
	\end{equation}
\end{lem}
\begin{proof}
	Let $\varphi:[0,1]\times S_{ref}\rightarrow J\times S$ be a scaling with $S_{ref}$ a $d$-dimensional simplex with $|S_{ref}|=1$. Thus, the identity $|\det (\nabla \varphi^{-1})| =|J\times S|^{-1}|[0,1]\times S_{ref}|=|J\times S|^{-1}$ holds true. Now for a given \mbox{$f \in B^s_{q,q}(J\times S)$}, we define $\tilde{f}:= f\circ \varphi$. For $q<\infty$ we derive 
	\begin{align*}
		\big|\tilde{f}\big|&\phantom{}_{B^{s_1, s_2}_{q,q}\left([0,1]\times S_{ref}\right)}^q 
		\\&= \int\limits_0^\infty \tilde{\delta}^{-s_1q}\omega_{r_1,t}\left(\tilde{f}, [0,1]\times S_{ref}, \tilde{\delta}\right)_q^q \frac{d\tilde{\delta}}{\tilde{\delta}} + \int\limits_0^\infty \tilde{\delta}^{-s_2q}\omega_{r_2,\bm{x}}\left(\tilde{f}, [0,1]\times S_{ref}, \tilde{\delta}\right)_q^q \frac{d\tilde{\delta}}{\tilde{\delta}}
		\\&= |J\times S|^{-1}\left(\int\limits_0^\infty \tilde{\delta}^{-s_1q}\omega_{r_1, t}\left(f, J\times S, \tilde{\delta}\,|J|\right)_q^q \frac{d\tilde{\delta}}{\tilde{\delta}} + \int\limits_0^\infty \tilde{\delta}^{-s_2q}\omega_{r_2, \bm{x}}\left(f, J\times S, \tilde{\delta}\,|S|^{\frac{1}{d}}\right)_q^q \frac{d\tilde{\delta}}{\tilde{\delta}}\right)
        \\&= |J\times S|^{-1}\left(|J|^{s_1q}\int\limits_0^\infty \delta^{-s_1q}\omega_{r_1, t}\left(f, J\times S, \delta\right)_q^q \frac{d\delta}{\delta} + |S|^\frac{s_2q}{d}\int\limits_0^\infty \delta^{-s_2q}\omega_{r_2, \bm{x}}\left(f, J\times S, \delta\right)_q^q \frac{d\delta}{\delta}\right)
        \\&\le |J\times S|^{-1} \max\left(|J|^{s_1}, |S|^{\frac{s_2}{d}}\right)^q |f|_{B^{s_1, s_2}_{q,q}(J\times S)}^q,
	\end{align*}
    where we have used \cref{lem:Behaviour under affine linear transformations} with respect to $\varphi^{-1}$ in the second step and two linear substitutions in the third one. The case $q=\infty$ is very similar, thus we obtain
	\begin{align}\label{Lemma_Whitney_any_simplex__proof_2}
		\big|\tilde{f}\big|_{B^{s_1, s_2}_{q,q}\left([0,1]\times S_{ref}\right)}\le  |J\times S|^{-\frac{1}{q}} \max\left(|I|^{s_1}, |S|^{\frac{s_2}{d}}\right)  |f|_{B^{s_1,s_2}_{q,q}(J\times S)},
	\end{align}
	for arbitrary $q\in(0,\infty]$. Since $|f|_{B^{s_1,s_2}_{q,q}(J\times S)}<\infty$, this implies $\big|\tilde{f}\big|_{B^{s_1,s_2}_{q,q}\left([0,1]\times S_{ref}\right)}<\infty$. Additionally, Jacobi's transformation formula further shows $\big\|\tilde{f}\big\|_{L_{q}\left([0,1]\times S_{ref}\right)} = |J\times S|^{-\frac{1}{q}}\,\|f\|_{L_q(J\times S)}<\infty$.
	Thus, also $\tilde{f}\in L_q\left([0,1]\times S_{ref}\right)$ holds true. Together, we conclude $\tilde{f}\in B^{s_1, s_2}_{q,q}\left([0,1]\times S_{ref}\right)$. Now \cref{Lemma_Whitney_standard_simplex} gives $\tilde{P}\in \Pi_{t,\bm{x}}^{r_1,r_2}([0,1]\times S_{ref})$ such that 
	\begin{equation}\label{Lemma_Whitney_any_simplex__proof_2.5}
		\left|\left|\tilde{f}-\tilde{P}\right|\right|_{L_p\left([0,1]\times S_{ref}\right)}\lesssim_{\, d,p,q,s_1,s_2,\kappa_S} |\tilde{f}|_{B_{q,q}^{s_1, s_2}\left([0,1]\times S_{ref}\right)}.
	\end{equation}
	Beware of the fact here that $\kappa_{S_{ref}}=\kappa_S$. Defining $P:=\tilde{P}\circ \varphi^{-1}\in \Pi^{r_1,r_2}_{t,\bm{x}}(J\times S)$ and another use of Jacobi's formula yields
	\begin{align}\label{Lemma_Whitney_any_simplex__proof_3}
		\left|\left|\tilde{f}-\tilde{P}\right|\right|_{L_{p}\left([0,1]\times S\right)} = |J\times S|^{-\frac{1}{p}}\|f-P\|_{L_p(J\times S)}.
	\end{align}
	Now \eqref{Lemma_Whitney_any_simplex__proof_2}, \eqref{Lemma_Whitney_any_simplex__proof_2.5}, and \eqref{Lemma_Whitney_any_simplex__proof_3} imply the asserted equation \eqref{Lemma_Whitney_any_simplex__equation}.
\end{proof}

In the remainder of this section, in addition to the assumptions made in the beginning of \cref{Sect:Approx of Lebesgue functions}, we also assume that $D$ is polyhedral and covered with a space-time partition $\calP$, i.e., a non-overlapping covering of $I\times D$ by prisms $J\times S$ with $J$ a bounded interval and $S$ a $d$-dimensional simplex. For such domains, we can give the following corollary to \cref{Lemma equivalence of (quasi-)seminorms for polyhedral domains} which corresponds to similar results for the isotropic case in \cite[Lem.~4.4]{BDDP02} and \cite[Lem.~4.10]{GM14}.
\begin{cor}\label{Corollary to Lemma equivalence of (quasi-)seminorms for polyhedral domains}
	Let $q\in (0,\infty]$, $s_1, s_2 \in (0, \infty)$, and $f\in B^{s_1, s_2}_{q,q}(I\times D)$. Then it holds
	\begin{align*}
		\sum\limits_{J\times S\in \calP} |f|^q_{B^{s_1,s_2}_{q,q}(J\times S)}  \lesssim_{\, d,q,s_1,s_2, \kappa_\calP} |f|^q_{B^{s_1,s_2}_{q,q}\left(I\times D\right)},
	\end{align*}
	for $q<\infty$, and
	\begin{align*}
		\sup\limits_{J\times S\in \calP} |f|_{B^{s_1,s_2}_{\infty,\infty}(J\times S)}  \lesssim_{\, d,s_1,s_2, \kappa_\calP} |f|_{B^{s_1,s_2}_{\infty,\infty}\left(I\times D\right)},
	\end{align*}
	if $q=\infty$. 
\end{cor}
\begin{proof} 
    The assertion follows from the inclusion $\bigcup\limits_{J\times S\in \calP} J_{r_1,h}\subset \left(\bigcup\limits_{J\times S\in \calP} J \right)_{r_1,h}=I_{r_1,h}$, $h\in \R$, and its counterpart \mbox{$\bigcup\limits_{J\times S\in \calP} S_{r_2,\bm{h}}\subset \left(\bigcup\limits_{J\times S\in \calP} S\right)_{r_2,\bm{h}}=D_{r_2,\bm{h}}$}, $\bm{h}\in \R^d$, \cref{Lemma equivalence of (quasi-)seminorms for polyhedral domains}, beware that $\LipProp(S)\sim_{\,d}\kappa_S\le \kappa_\calP$, and the subadditivity of the averaged moduli of smoothness. 
\end{proof}
Now we can finally proceed to the proofs of two pending main results, i.e., \cref{thm:Whitney} and \cref{thm:Besov_Embedding}. 

\begin{proof}[Proof of \cref{thm:Whitney}]
    Recall the definition of $a(\calP)$ in \eqref{def:aP} from \cref{sect:Main results}. Then, since for any $J\times S\in \calP$,
    \begin{align*}
    |J\times S|^\frac{1}{\frac{1}{s_1}+\frac{d}{s_2}}\sim_{a(\calP)}\left(|S|^{\frac{s_2}{s_1d}+1}\right)^\frac{1}{\frac{1}{s_1}+\frac{d}{s_2}}= \left(|S|^{\frac{s_2}{s_1d}+1}\right)^\frac{s_1s_2}{s_2+ds_1}=|S|^\frac{s_2}{d}\sim_{a(\calP),s_1}\max\left(|J|^{s_1}, |S|^{\frac{s_2}{d}}\right), 
    \end{align*}
    the local estimate \eqref{thm:Whitney - equation local} follows immediately from \cref{Lemma_Whitney_any_simplex}. For the global estimate \eqref{thm:Whitney - equation global}, we first assume $q\le p$ and $p<\infty$. Then, we can estimate
    \begin{align*}
        \|f-P\|_{L_p(I\times D)}&=\left(\sum\limits_{J\times S\in \calP} \|f-P_{J\times S}\|^p_{L_p(J\times S)}\right)^{\frac{1}{p}}
        \\&\lesssim_{\,d,p,q,s_1, s_2, \kappa_\calP, a(\calP)}\left(\sum\limits_{J\times S\in \calP} |J\times S|^{\left(\frac{1}{\frac{1}{s_1}+\frac{d}{s_2}}-\frac1q +\frac1p\right)p} |f|^p_{B^{s_1, s_2}_{q,q}(J\times S)}\right)^{\frac{1}{p}}
        \\
        &\le \max_{J\times S\in \calP}|J\times S|^{\frac{1}{\frac{1}{s_1}+\frac{d}{s_2}}-\frac{1}{q}+\frac{1}{p}}
        \left(\sum\limits_{J\times S\in \calP} |f|^q_{B^{s_1, s_2}_{q,q}(J\times S)}\right)^{\frac{1}{q}} 
        \\&\lesssim_{\,d,p,q,s_1,s_2, \kappa_\calP} \max_{J\times S\in \calP}|J\times S|^{\frac{1}{\frac{1}{s_1}+\frac{d}{s_2}}-\frac{1}{q}+\frac{1}{p}} |f|_{B^{s_1, s_2}_{q,q}(I\times D)}
        \\&\le |I\times D|^{\frac{1}{\frac{1}{s_1}+\frac{d}{s_2}}-\frac1q +\frac1p} |f|_{B^{s_1, s_2}_{q,q}(I\times D)}.
    \end{align*}
    We have applied the local estimate and $\kappa_S \le \kappa_\calP$, for every $J\times S\in \calP$, in the second step, the embedding $\ell^{q}(\N_0)\hookrightarrow \ell^{p}(\N_0)$ in the third, and \cref{Corollary to Lemma equivalence of (quasi-)seminorms for polyhedral domains} in the last one. For $p=\infty$, this works analogously with suprema instead of sums. For $p<q$, we apply Hölder's inequality to $\|f-P\|_{L_p(I\times D)}$ with respect to $\frac{q}{p}>1$ and then use the above estimate with \enquote*{$p=q$} to prove the assertion.
\end{proof}

\begin{proof}[Proof of \cref{thm:Besov_Embedding}]
    We define $P$ as in \cref{thm:Whitney} and estimate 
    \begin{align}\label{Proof of thm:Besov_Embedding - Equation 1}
        \|f\|_{L_p(I\times D)}&\lesssim_{\,p} \|f-P\|_{L_p(I\times D)} + \|P\|_{L_p(I\times D)}
        \\&\lesssim_{\,d,p,q,s_1,s_2, \kappa_\calP, a(\calP), \max\limits_{J\times S\in \calP}|J\times S|} \ |f|_{B^{s_1, s_2}_{q,q}(I\times D)} + \|P\|_{L_p(I\times D)} ,\notag
    \end{align}
    where we have applied \eqref{thm:Whitney - equation global} from \cref{thm:Whitney}. We will again first consider the case $q\le p$. The choice of $P_{J\times S}$, $J\times S\in \calP$, from \cref{Lemma_Whitney_standard_simplex} and \cref{Lemma_Whitney_any_simplex}, respectively, guarantees that it fulfills   Jackson's estimate from \cref{thm:Jackson} on $J\times S$ with respect to $q$, due to the scaling invariance of   Jackson's estimate. The latter is (as it has been mentioned before) a consequence of the Jacobi transformation theorem and \cref{lem:Behaviour under affine linear transformations}. Therefore, for $p<\infty$, we can estimate
    \begin{align}\label{Proof of thm:Besov_Embedding - Equation 2}
        \|P\|_{L_p(I\times D)}&\lesssim_{\,} \left(\sum\limits_{J\times S\in \calP} \|P_{J\times S}\|_{L_p(J\times S)}^p\right)^{\frac{1}{p}}\lesssim_{\,d,p,q,s_1, s_2} \left(\sum\limits_{J\times S\in \calP} |J\times S|^{\frac{q}{p}-1} \|P_{J\times S}\|_{L_q(J\times S)}^q\right)^{\frac{1}{q}}
        \\ & \le \left(\min\limits_{J\times S\in \calP}|J\times S|\right)^{\frac{1}{p}-\frac 1q} \left(\sum\limits_{J\times S\in \calP} \|P_{J\times S}\|_{L_q(J\times S)}^q\right)^{\frac{1}{q}}\notag
        \\ & \lesssim_{\,p,q,\min\limits_{J\times S\in \calP}|J\times S| }\left(\sum\limits_{J\times S\in \calP} \|f-P_{J\times S}\|_{L_q(J\times S)}^q\right)^{\frac{1}{q}} + \left(\sum\limits_{J\times S\in \calP} \|f\|_{L_q(J\times S)}^q\right)^{\frac{1}{q}}\notag
        \\ & \lesssim_{\,d,q,s_1,s_2} \left(\sum\limits_{J\times S\in \calP} \omega_{r_1, t}(f, J\times S, |J|)_q^q + \omega_{r_2, \bm{x}}(f, J\times S, \diam(S))_q^q\right)^{\frac{1}{q}} + \|f\|_{L_q(I\times D)}\notag
        \\ & \lesssim_{\,q,s_1,s_2}\|f\|_{L_q(I\times D)}.\notag
    \end{align}
    where we have applied $\ell^{q}(\N_0)\hookrightarrow \ell^{p}(\N_0)$ and \cref{lem:Polynomial scaling} in the second step, $\frac{q}{p}-1<0$ in the third, the aforementioned property of Jackson's estimate (beware of the convexity of $S$) in the fifth, and the possibility to estimate the moduli of smoothness by the corresponding Lebesgue norm from \cref{Rem:Moduli of smoothness} in the last step. Finally, \eqref{Proof of thm:Besov_Embedding - Equation 1} and \eqref{Proof of thm:Besov_Embedding - Equation 2} together show the asserted embedding for $q\le p$. For $q>p$, we can simply use the chain of embeddings $B^{s_1,s_2}_{q,q}(I\times D)\hookrightarrow L_q(I\times D)\hookrightarrow L_p(I\times D)$.
\end{proof}

\section{Direct estimates for discontinuous, adaptive space-time finite elements}\label{Sect:Dir_est}

We will now use the  results proven in \cref{Section:Approximation_with_Besov_functions}, in particular, the Whitney-type inequality from \cref{thm:Whitney}, in order to demonstrate that our new approach to anisotropic Besov spaces is an appropriate one for the study of approximation with space-time finite elements. Therefore, we will prove a direct theorem for adaptive approximation with discontinuous, space-time finite elements starting with a initial tensor-product structure space-time partition. The achievement of a corresponding result for continuous finite elements as well as inverse estimates is currently under investigation, but   out of the scope of this article.

As mentioned above, initially, we will cover the domain $I\times D$  by a tensor product space-time partition of the form $\calP_0 :=:\calI_0\otimes \calT_0 :=\{J\times S\mid J\in \calI_0, S\in \calT_0 \}$. Here, by 
\begin{align*}
	\mathcal{I}_0:=\{t_0<\dots< t_N\}:=\{[t_{i-1},t_i)\mid i=1,\dots,N-1\}\cup \{[t_{N-1},t_N]\}, \quad N\in\N, 
\end{align*}
we denote an initial, non-overlapping (temporal) partition of $I$, which we assume to be closed without loss of generality. Further, let $\calT_0$ be an initial non-overlapping (spatial) covering of $\Omega$ consisting of closed $d$-dimensional simplices, i.e., a simplicial triangulation.

Further, we require $\calT_0$ to fulfill the conditions $a)$ and $b)$ from \cite[Sect.~4]{Ste08}, or to be otherwise \textbf{well-labeled}, e.g.  according to \cite{DGS23}.\footnote{Be aware of the fact that this is always fulfilled if $d=1$.} This allows us to apply the $d$-dimensional bisection routine from \cite{Mau95} and \cite{Tra97}, that was already applied in the proof of \cref{Lemma_Whitney_standard_simplex}, to elements of $\calT_0$ (arbitrarily often). In particular, for $t\in \{1,d\}$, we will denote the set of $t$-dimensional simplices created by the bisection of a well-labeled $t$-dimensional simplex $R$ by \mbox{$\textup{BISECT}(t, R)$}.

Furthermore, we will use the notation $\textup{BISECT}(t, R)^{\otimes 0 }:=\{R\}$, $\textup{BISECT}(t, R)^{\otimes 1 }:=\textup{BISECT}(t, R)$, and  
\begin{align*}
    \textup{BISECT}(t, R)^{\otimes k }:=\bigcup\limits_{R' \in \,\textup{BISECT}(t, R)^{\otimes (k-1) }} \textup{BISECT}(t, R') \quad \text{for}\quad k\in \N_{\ge 2}.
\end{align*} 
Additionally, we will denote the \textbf{levels} of simplices by $\ell(R)$, i.e., $\ell(R)=k$, for $k\in \N_0$, if and only if \mbox{$R\in \textup{BISECT}(t, R_0)^{\otimes k}$} for some $R_0\in \mathcal{R}_0$, with $\mathcal{R}_0\in \{\calI_0, \calT_0\}$ chosen according to the selection of $t$.

Now, for given anisotropy parameters $s_1, s_2\in (0,\infty)$, we consider the following algorithm for \textbf{atomic refinement} of individual elements $J\times S$ of a space-time partition $\calP$, which itself has been inductively generated by finitely many applications of the below algorithm, starting with $\calP_0$. 

\begin{algorithm}[h]
    \caption{$\textup{ATOMIC}\underline{~}\textup{SPLIT}(J\times S, d, s_1, s_2)$}
    \label{Algorithm_ANISOTROPIC_BISECT}
    \begin{algorithmic}
    	\State $n\gets \ell(S)+1$ 
    	\State $m\gets \left\lceil \frac{ns_2}{s_1d}\right\rceil - \left\lceil\frac{(n-1)s_2}{s_1d}\right\rceil$
    	\State $S_{children} \gets \textup{BISECT}(d, S)$
    	\State $J_{children}\gets \textup{BISECT}(1, J)\phantom{}^{\otimes m }$
    	\State \Return $J_{children}\times S_{children}$
    \end{algorithmic}
\end{algorithm}

\begin{center}
    \begin{figure}[h]
        \begin{center}
            \includegraphics[height = 6cm]{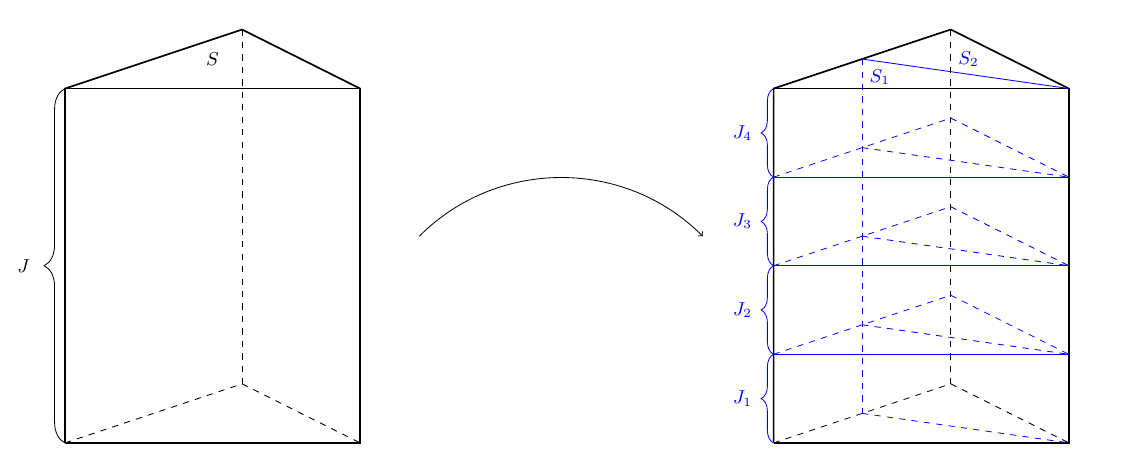}
            $\textup{ATOMIC}\underline{~}\textup{SPLIT}(J\times S, d, s_1, s_2)$ for $d=2$ and $s_2=4s_1=2s_1d$.
        \end{center}
    \end{figure}
\end{center}

We will extend our notation as follows. For $k\in \N_0$, we write \mbox{$\textup{ATOMIC}\underline{~}\textup{SPLIT}( \cdot, d, s_1, s_2)^{\otimes k}$}, correspondingly to our notation of \mbox{$\textup{BISECT}(t, \cdot)^{\otimes k }$}, $t\in \{1,d\} $. Furthermore,  the level $\ell(J\times S)=k$ holds true, if and only if \mbox{$J\times S\in \textup{ATOMIC}\underline{~}\textup{SPLIT}( J_0 \times S_0, d, s_1, s_2)^{\otimes k}$} for some $J_0 \times S_0\in \calP_0$. In particular, the latter is the case if and only if $\ell(J)=\left\lceil \frac{ks_2}{s_1d}\right\rceil$ and $\ell(S)=k$.

\begin{lem}\label{Lemma_Anisotropic_Bisection_maintains_anisotropy}
    The algorithm is well defined. Further, assume that $\calP$ has been derived from $\calP_0$ by $k_0\in\N_0$ iterative applications of $\textup{ATOMIC}\underline{~}\textup{SPLIT}(\cdot, d, s_1, s_2)$. Then there are constants $C_1>0$ and $C_2(d)>0$ independent of the number of iterations, such that
	\begin{align*}
		a(\calP) \le C_1 a(\calP_0), \quad \kappa_{\calP}\le C_2\kappa_{\calP_0} \quad \text{and} \quad \#\calP' -\#\calP \lesssim_{\,d,s_1, s_2} k_0.
	\end{align*}
\end{lem}
\begin{proof}
    Since $\calT_0$ is well-labeled, this also holds true for all of the simplices that have been created (iteratively) by the $d$-dimensional bisection routine. Therefore, the algorithm is indeed well-defined. Now let $J\times S\in \calP$. Then, $J\times S\in \textup{ATOMIC}\underline{~}\textup{SPLIT}(J_0\times S_0, d, s_1, s_2)^{\otimes k}$ for some $k\in\{0,\dots,k_0\}$ and $J_0\times S_0\in \calP_0$. We obtain
    \begin{align*}
        \frac{|J|}{|S|^\frac{s_2}{s_1d}}= \frac{2^{-\left\lceil\frac{ks_2}{s_1d}\right\rceil}|J_0|}{\left(2^{-k} |S_0|\right)^\frac{s_2}{s_1d}}\sim \frac{2^{-\frac{ks_2}{s_1d}}|J_0|}{\left(2^{-k} |S_0|\right)^\frac{s_2}{s_1d}}=\frac{|J_0|}{ |S_0|^\frac{s_2}{s_1d}}\sim a(\calP_0),
    \end{align*}
    which shows the first estimate. The second assertion holds true due to the the properties of the $d$-dimensional bisection routine, see \cite[Thm.~2.1]{Ste08}. The last one is a consequence of the fact that an application of the atomic refinement method creates not more than $2^{\frac{s_2}{s_1d}+1}$ simplices, which one can see as follows. Let \mbox{$l:=\ell(J\times S)\in \N_0$}, then 
    \begin{align*}
	   \# \textup{ATOMIC}\underline{~}\textup{SPLIT}(J\times S, d, s_1, s_2)= 2^{\left\lceil \frac{ls_2}{s_1d} \right\rceil-\left\lceil \frac{(l-1)s_2}{s_1d}\right\rceil}\cdot 2 \le 2^{\frac{ls_2}{s_1d} + 1 - \frac{(l-1)s_2}{s_1d} } = 2^{\frac{s_2}{s_1d}+1}.
\end{align*}
\end{proof}

In order to prove \cref{Theorem_dir_est_discont}, we have been guided by the procedure from \cite[Sect.~4]{AMS23}. We adopt the Greedy algorithm specified there subject to some minor changes as follows. We assume $\delta\in(0,\infty)$.

\begin{algorithm}[h]
\caption{$\textup{GREEDY}(\delta)$} 
\begin{algorithmic}
	\State $k\gets 0$
	\While{True}
		\State $\calM_k:=\left\{J\times S\in \calP_k \ \bigg| \ \inf\limits_{P\in \Pi^{r_1, r_2}_{t,\bm{x}}(J\times S)} \|f-P\|_{L_p(J\times S)}>\delta\right\}$
		\If{$\calM_k=\emptyset$}
			\State \textup{break}
		\EndIf		
		\State $\calP':= \calP_k$
		\For{$J\times S\in \calM_k$}		
			\State $\calP'\gets (\calP' \setminus \{J\times S\} )\cup\textup{ATOMIC}\underline{~}\textup{SPLIT}(J\times S, d, s_1, s_2)$
		\EndFor		
		\State $\calP_{k+1}\gets \calP'$
		\State $k\gets k+1$
	\EndWhile
	\State $\calP\gets \calP_k$
	\State \Return $\calP$
\end{algorithmic}
\end{algorithm}

\begin{rem}\label{Rem_Greedy}
	If the above algorithm terminates in a finite number $k\in\N$ of steps, the complexity estimate $\# \calP_j - \# \calP_0 \lesssim_{\, d, s_1, s_2} \sum\limits_{i=0}^{j-1}\#\calM_i$
	holds for any $j\in \{1,\dots, k\}$ due to the third estimate in Lemma \ref{Lemma_Anisotropic_Bisection_maintains_anisotropy}. 
\end{rem}

\begin{lem}\label{Lemma_dir_est_discont_2}	
	Let $p,q, s_1, s_2$, and $f$ be as in \cref{Theorem_dir_est_discont}. Then $\textup{GREEDY}(\delta)$ terminates for any \mbox{$\delta\in (0,\infty)$}.
\end{lem}
	
\begin{proof}
	Assume that the algorithm does not terminate. Then, for each $k\in\N_0$, there exists some $J_k\times S_k\in \calM_k\subset \calP_k$ with 
    \begin{align*}
        \delta<\inf\limits_{P\in \Pi^{r_1, r_2}_{t,\bm{x}}(J_k\times S_k)} \|f-P\|_{L_p(J_k\times S_k)} .
    \end{align*}
    Let $P_k$ be the anisotropic polynomial in the local Whitney estimate on $J_k\times S_k$ from \cref{thm:Whitney}. Thus, it holds
	\begin{align*}
		\delta<  \|f-P_k\|_{L_p(J_k\times S_k)}&\lesssim_{\,d,p,q,s_1,s_2, \kappa_{\calP_0}, a(\calP_0)} |J_k\times S_k|^{\frac{1}{\frac{1}{s_1}+\frac{d}{s_2}}-\frac{1}{q}+\frac{1}{p}}|f|_{B^{s_1, s_2}_{q,q}(J_k\times S_k)}
	\end{align*}
	due to $\kappa_{\calP_k}\lesssim_{\,d}\kappa_{\calP_0} $ and $a(\calP_k)\lesssim (\calP_0)$ which were shown in \cref{Lemma_Anisotropic_Bisection_maintains_anisotropy}. Since $|J_k\times S_k|\xrightarrow{k\rightarrow\infty} 0$, as well as \mbox{$\frac{1}{\frac{1}{s_1}+\frac{d}{s_2}}-\frac{1}{q}+\frac{1}{p}>0$}, and $|f|_{B^{s_1, s_2}_{q,q}(J_k\times S_k)}\le |f|_{B^{s_1, s_2}_{q,q}(I\times D)}<\infty$, we  conclude $\delta \le 0$ for $k\rightarrow\infty$, which is a contradiction. Thus, the algorithm terminates.
\end{proof}

Now we can proceed to the proof of the last of the three main results stated in \cref{sect:Main results}.

\begin{proof}[Proof of \cref{Theorem_dir_est_discont}]
	If $|f|_{B^{s_1, s_2}_{q,q}(I\times D)}=0$, then $f\in \Pi^{r_1, r_2}_{t,\bm{x}}(I\times D)$ according to \cref{Lemma moduli of smoothness = 0 => anisotropic polynomials}, and therefore the assertion is proven with the choice of $F:=f\in \mathbb{V}^{r_1 , r_2}_{\calP_0, \textup{DC}}(I\times D)$. Now assume $|f|_{B^{s_1, s_2}_{q,q}(I\times D)}>0$, set $\delta:= \varepsilon^{1+\frac{1}{s_1p}+\frac{d}{s_2p}}|f|_{B^{s_1, s_2}_{q,q}(I\times D)}$ and let $\calP:=\textup{GREEDY}(\delta)$, which is well defined because the algorithm terminates in a finite number $k\in\N_0$ of steps, due to Lemma~\ref{Lemma_dir_est_discont_2}. Moreover, due to \cref{Rem_Greedy}, 
	\begin{align}\label{Theorem_dir_est_discont_mesh_size_1}
		\# \calP - \# \calP_0= \# \calP_k - \# \calP_0 \lesssim_{\,d,s_1, s_2} \sum\limits_{j=0}^{k-1}\#\calM_j= \#\overline{\calM}
	\end{align}
	with $\overline{\calM}:=\bigcup\limits_{j=0}^{k-1}\calM_j$. Beware that in fact, $\calM_j=\left\{J\times S \in \overline{\calM}\mid \ell(J\times S)=j\right\}$ for $j\in\{0,\cdots, k-1\}$. Further, we set $\calM_j:=\emptyset$ if $j\ge k$. Then, for any such $j$ we can estimate
	\begin{align}\label{Theorem_dir_est_discont_estimate_min}
	\min\limits_{J\times S\in \calM_j}|J\times S|&\ge \min\limits_{\substack{J\times S\in \textup{ATOMIC}\underline{~}\textup{SPLIT}(J_0\times S_0, d, s_1, s_2)^{\otimes j},\\ J_0\times S_0 \in \calP_0}}|J\times S| \gtrsim_{\, \min\limits_{J_0\times S_0 \in \calP_0}|J_0\times S_0|}2^{-j\left(1+\frac{s_2}{s_1d}\right)}
	\end{align}		
	and therefore $\#\calM_j \lesssim_{\, |I\times D|, \min\limits_{J_0\times S_0 \in \calP_0}|J_0\times S_0|} 2^{j\left(1+\frac{s_2}{s_1d}\right)}$. Also, we can derive
	\begin{align*}
	 	\#\calM_j \,\delta^q &\le \sum\limits_{J\times S\in \calM_j} \left(\inf\limits_{P\in \Pi^{r_1, r_2}_{t,\bm{x}}(J\times S)} \|f-P\|_{L_p(J\times S)}\right)^q
	 	\\&\lesssim_{\, d,p,q,s_1, s_2, \kappa_{\calP_0},a(\calP_0)}\sum\limits_{J\times S\in \calM_j}|J\times S|^{q\left(\frac{1}{\frac{1}{s_1}+\frac{d}{s_2}}-\frac{1}{q}+\frac{1}{p}\right)}|f|^q_{B^{s_1, s_2}_{q,q}(J\times S)}\nonumber
	 	\\&\lesssim_{\,\max\limits_{J_0\times S_0 \in \calP_0}|J_0\times S_0|} 2^{-jq\left(1+\frac{s_2}{s_1d}\right)\left(\frac{1}{\frac{1}{s_1}+\frac{d}{s_2}}-\frac{1}{q}+\frac{1}{p}\right)}\sum\limits_{J\times S\in \calM_j}|f|^q_{B^{s_1, s_2}_{q,q}(J\times S)}\nonumber
	 	\\&\lesssim_{\,d,q,s_1, s_2, \kappa_{\calP_0}}2^{-jq\left(1+\frac{s_2}{s_1d}\right)\left(\frac{1}{\frac{1}{s_1}+\frac{d}{s_2}}-\frac{1}{q}+\frac{1}{p}\right)}|f|^q_{B^{s_1, s_2}_{q,q}(I\times D)},\nonumber
	\end{align*}
	for $q<\infty$, where we have used the local estimate from \cref{thm:Whitney} together with \cref{Lemma_Anisotropic_Bisection_maintains_anisotropy} in the second step, the reverse estimate to \eqref{Theorem_dir_est_discont_estimate_min} in the third, and Corollary \ref{Corollary to Lemma equivalence of (quasi-)seminorms for polyhedral domains} in the last one. Together, these estimates can be summarized as 
	\begin{align}\label{2-series}
	 	\#\calM_j\lesssim_{\,d, p, q, s_1, s_2, \kappa_{\calP_0}, a(\calP_0), |I\times D|, \mu_1, \mu_2}\min\left(2^{j\left(1+\frac{s_2}{s_1d}\right)}, \frac{1}{\delta^q}2^{-jq\left(1+\frac{s_2}{s_1d}\right)\left(\frac{1}{\frac{1}{s_1}+\frac{d}{s_2}}-\frac{1}{q}+\frac{1}{p}\right)}|f|^q_{B^{s_1, s_2}_{q,q}(I\times D)}\right) 
	\end{align}
    with the abbreviations $\mu_1:=\min\limits_{J_0\times S_0 \in \calP_0} |J_0\times S_0|$ and $\mu_2:=\max\limits_{J_0\times S_0 \in \calP_0} |J_0\times S_0|$.
	Since the left series in \eqref{2-series} is monotonically increasing in $j\in\N_0$, whereas the right one is decreasing, there exists 
	\begin{align}\label{Theorem_dir_est_discont_estimate_j_0}
	 	j_0=\min\left\{j\in \N_0 \ \bigg| \ 2^{j\left(1+\frac{s_2}{s_1d}\right)} > \frac{1}{\delta^q}2^{-jq\left(1+\frac{s_2}{s_1d}\right)\left(\frac{1}{\frac{1}{s_1}+\frac{d}{s_2}}-\frac{1}{q}+\frac{1}{p}\right)}|f|^q_{B^{s_1, s_2}_{q,q}(I\times D)}\right\}\in \N_0.
	\end{align}
	Therefore, with the abbreviation $t=t(d,p,q,s_1,s_2):=2^{q\left(1+\frac{s_2}{s_1d}\right)\left(\frac{1}{\frac{1}{s_1}+\frac{d}{s_2}}-\frac{1}{q}+\frac{1}{p}\right)}>1$, we can further estimate
	\begin{align}
	 	\#\overline{\calM}=\sum\limits_{j=0}^{\infty} \#\calM_j&= \sum\limits_{j=0}^{j_0-1} \#\calM_j + \sum\limits_{j=j_0}^{\infty} \#\calM_j \nonumber 
	 	\\&\lesssim_{\,d, p, q, s_1, s_2, \kappa_{\calP_0}, a(\calP_0), |I\times D|, \mu_1, \mu_2} \sum\limits_{j=0}^{j_0-1} 2^{j\left(1+\frac{s_2}{s_1d}\right)} + \frac{|f|^q_{B^{s_1, s_2}_{q,q}(I\times D)}}{\delta^q}\sum\limits_{j=j_0}^{\infty} t^{-j}\nonumber
	 	\\&\lesssim_{\,d,p,q,s_1,s_2} 2^{j_0\left(1+\frac{s_2}{s_1d}\right)} + \frac{|f|^q_{B^{s_1, s_2}_{q,q}(I\times D)}}{\delta^q}t^{-j_0}
        \lesssim  2^{j_0\left(1+\frac{s_2}{s_1d}\right)},\label{Theorem_dir_est_discont_mesh_size_2}
	\end{align}
	 where we have exploited our choice of $j_0$ in the penultimate step. The case $q=\infty$ admits the same result by a very similar calculation. Now we continue by finding a bound on $2^{j_0\left(1+\frac{s_2}{s_1d}\right)}$. If $j_0=0$, then $1$ is obviously a bound. Now assume $j_0\ge 1$. Due to the minimality of $j_0$, the inequality in \eqref{Theorem_dir_est_discont_estimate_j_0} holds true in the reversed form for $j_0-1$, i.e.,
    \begin{align*}
	 	2^{(j_0-1)\left(1+\frac{s_2}{s_1d}\right)}\le \frac{1}{\delta^q}2^{-(j_0-1)q\left(1+\frac{s_2}{s_1d}\right)\left(\frac{1}{\frac{1}{s_1}+\frac{d}{s_2}}-\frac{1}{q}+\frac{1}{p}\right)}|f|^q_{B^{s_1, s_2}_{q,q}(I\times D)}.
	\end{align*}
	This leads to
    \begin{align*}
	 	2^{\frac{1}{q}\left(j_0-1\right)\left(1+\frac{s_2}{s_1d}\right)}\,2^{(j_0-1)\left(1+\frac{s_2}{s_1d}\right)\left(\frac{1}{\frac{1}{s_1}+\frac{d}{s_2}}-\frac{1}{q}+\frac{1}{p}\right)}=2^{(j_0-1)\left(1+\frac{s_2}{s_1d}\right)\left(\frac{1}{\frac{1}{s_1}+\frac{d}{s_2}}+\frac{1}{p}\right)}\le \frac{|f|_{B^{s_1, s_2}_{q,q}(I\times D)}}{\delta}
	\end{align*}
	and therefore 
	\begin{align}
	 	2^{j_0\left(1+\frac{s_2}{s_1d}\right)}&\lesssim_{\,d,s_1,s_2} \left(\frac{|f|_{B^{s_1, s_2}_{q,q}(I\times D)}}{\delta}\right)^{\frac{1}{\left(\frac{1}{\frac{1}{s_1}+\frac{d}{s_2}}+\frac{1}{p}\right)}}\notag 
	 	\\&=\left(\varepsilon^{-\left(1+\frac{1}{s_1p}+\frac{d}{s_2p}\right)}\right)^{\frac{1}{\left(\frac{1}{\frac{1}{s_1}+\frac{d}{s_2}}+\frac{1}{p}\right)}}=\varepsilon^{-\left(\frac{1}{s_1}+\frac{d}{s_2}\right)}. \label{Theorem_dir_est_discont_mesh_size_3}
	\end{align}
	Combining \eqref{Theorem_dir_est_discont_mesh_size_1}, \eqref{Theorem_dir_est_discont_mesh_size_2}, and \eqref{Theorem_dir_est_discont_mesh_size_3} now gives the asserted estimate \eqref{Theorem_dir_est_discont_1}. We will denote the constant involved in \eqref{Theorem_dir_est_discont_1} by $C$ and the one from \eqref{thm:Whitney - equation global}, with $|I\times D|^{\frac{1}{\frac{1}{s_1}+\frac{d}{s_2}}-\frac{1}{q}+\frac{1}{p}}$ incorporated, with respect to $\calP_0$ by $C_1$. Without loss of generality, we increase $C$ such that, in particular $C\ge C_1^{\frac{1}{s_1}+\frac{d}{s_2}}$. Now we have to distinguish two cases. If $\varepsilon^{\frac{1}{s_1}+\frac{d}{s_2}}>C $, then \eqref{Theorem_dir_est_discont_1} implies that $\#\calP - \#\calP_0 < 1$, i.e., $\#\calP = \#\calP_0$, and therefore $\calP =\calP_0$. Now the choice of $C$, $C_1$, and \eqref{thm:Whitney - equation global} yield the existence of some $P\in \V^{r_1, r_2}_{\calP_0, \textup{DC}}=\V^{r_1, r_2}_{\calP, \textup{DC}}$ with
    \begin{align*}
        \|f-P\|_{L_p(I\times D)}\le C_1 |f|_{B^{s_1,s_2}_{q,q}(I\times D)}\le C^{\frac{1}{\frac{1}{s_1}+\frac{d}{s_2}}}|f|_{B^{s_1,s_2}_{q,q}(I\times D)}<\varepsilon |f|_{B^{s_1,s_2}_{q,q}(I\times D)} .
    \end{align*}
    Now consider $\varepsilon^{\frac{1}{s_1}+\frac{d}{s_2}}\le C $ and let $P_{J\times S}\in \Pi^{r_1,r_2}_{t,\bm{x}}(J\times S)$ be a best approximation of $f$ with respect to $\|\cdot\|_{L_p(J\times S)}$ on $J\times S\in \calP$. Such $P_{J\times S}$ always exists (not necessarily unique), since $\Pi^{r_1,r_2}_{t,\bm{x}}(J\times S)$ is finite dimensional. Further, define $P:=\sum\limits_{J\times S\in \calP}\mathds{1}_{J\times S}P_{J\times S}\in \V^{r_1, r_2}_{\calP, \textup{DC}}(I\times D)$ and estimate
	 \begin{align*}
	 	\|f-P\|_{L_p(I\times D)}^p&= \sum\limits_{J\times S\in \calP}\|f-P_{J\times S}\|_{L_p(J\times S)}^p\le \sum\limits_{J\times S\in \calP} \delta^p= \#\calP\, \delta^p
	 	\le \left(C\varepsilon^{-\left(\frac{1}{s_1}+\frac{d}{s_2}\right)}  + \#\calP_0\right)\delta^p
        \\&\le C(1+\#\calP_0) \varepsilon^{-\left(\frac{1}{s_1}+\frac{d}{s_2}\right)} \delta^p
	 	\\&\lesssim_{\, \# \calP_0}C\varepsilon^{-\left(\frac{1}{s_1}+\frac{d}{s_2}\right)} \varepsilon^{\left(1+\frac{1}{s_1p}+\frac{d}{s_2p}\right)p}|f|^p_{B^{s_1, s_2}_{q,q}(I\times D)}=C\varepsilon^p |f|^p_{B^{s_1, s_2}_{q,q}(I\times D)},
	 \end{align*}
	 with the use of \eqref{Theorem_dir_est_discont_1}, for $p<\infty$. The case $p=\infty$ works in the same way by replacing the sums by the corresponding suprema and observing that in this case $\delta=\varepsilon|f|_{B^{s_1, s_2}_{q,q}(I\times D)}$. In any case, this directly implies assertion \eqref{Theorem_dir_est_discont_2}. 
\end{proof}

\appendix
\section{Proofs of auxiliary lemmata}\label{Appendix:Proof of Lemma equivalence averaged spacial modulus of smoothness}

\begin{proof}[Proof of \cref{Remark Lipschitz domain}\ref{Remark Lipschitz domain - interior cone condition}]
    For any $j\in\{1,\dots, \LipCov(D)\}$, $D_j=\Gamma_j(G_j)$ for a rotation $\Gamma_j$ and a Lipschitz graph domain $G_j$ with corresponding Lipschitz constant $\Lip(D)$. Therefore, an infinite cone $\mathcal{C}_{inf}=\mathcal{C}_{inf}(\Lip(D))$ exists, such that $G_j+\mathcal{C}_{inf}\subset G_j$, and thus, $D_j +\mathcal{C}_{inf,j}\subset D_j $ with $\mathcal{C}_{inf,j}:=\Gamma_j(\mathcal{C}_{inf})$, for any $j\in\{1,\dots, \LipCov(D)\}$. Now we cut $\mathcal{C}_{inf,j}$ off appropriately such that we get a bounded cone $\mathcal{C}_j$ with $\diam(\mathcal{C}_j)\le\frac{\LipDelta(D)}{4}$ for any such $j$. In particular, these cones can be chosen such that $(\mathcal{C}_j)_{j=1,\dots, \LipCov(D)}$ is a set of congruent cones. Now let $\bm{x}\in D$. If $d(\bm{x}, \partial D)\ge \frac{\LipDelta(D)}{2}$, then $d(\bm{x}+\mathcal{C}_1, \partial D)\ge \frac{\LipDelta(D)}{4}$, therefore $\bm{x}+\mathcal{C}_1\subset D$. If otherwise $d(\bm{x}, \partial D)< \frac{\LipDelta(D)}{2}$, then, according to \cref{Definition of Lipschitz domain}\ref{Definition of Lipschitz domain - Delta part}), there must be \mbox{$j\in\{1,\dots,\LipCov(D)\}$} such that $\bm{x}\in U_j$ and $d(\bm{x},\partial U_j)>\LipDelta(D)$. Therefore, $\bm{x}+\mathcal{C}_j\subset U_j$, due to $d(\bm{x}+\mathcal{C}_j,\partial U_j)> \frac{3}{4}\LipDelta(D)$. Further, the choice of $\mathcal{C}_j\subset \mathcal{C}_{inf,j}$ implies that $\bm{x}+\mathcal{C}_j\subset D_j$, since $\bm{x}\in D\cap U_j=D_j \cap U_j$ due to \cref{Definition of Lipschitz domain}\ref{Definition of Lipschitz domain - special Lipschitz domain part}). The latter now also shows $\bm{x}+\mathcal{C}_j\subset D$, since $\bm{x}+\mathcal{C}_j\subset D_j\cap U_j=D\cap U_j\subset D$.
\end{proof}

\begin{proof}[Proof of \cref{Remark Lipschitz domain}\ref{Definition of Lipschitz domain - delta < diam}]
    Assume $\LipDelta(D)>\diam(D)$ was true. Then 
    \begin{align*}
        \max\{d(\bm{x},\partial D),d(\bm{y},\partial D),|\bm{x}-\bm{y}|\le \diam(D)<\LipDelta(D)
    \end{align*}    
    would be true for arbitrary $\bm{x},\bm{y}\in D$. Therefore, \cref{Definition of Lipschitz domain}\ref{Definition of Lipschitz domain - Delta part}) implies that there is \mbox{$j\in\{1,\dots,\LipCov(D)\}$} such that $D\subset U_j$. However, $d(D, U_j)>0$, i.e. $D\subsetneq U_j$ since $\LipDelta(D)$ is strictly larger than $\diam(D)$. Hence, since $D_j$ unbounded, $D\cap D_j$ must be a true subset of $U_j \cap D_j$. But also $D\cap D_j = (U_j\cap D_j)\cap D_j= U_j\cap D_j$. This now clearly is a contradiction.
\end{proof}

\begin{proof}[Proof of \cref{Lemma equivalence averaged spacial modulus of smoothness}]

The first inequality is trivial. The second one is, too, if $p=\infty$. Regarding the second estimate for \mbox{$p\in (0,\infty)$}, we will proceed similar to \cite[Lemma~4.1]{GM14}. Again we use that the estimate remains invariant under affine transformations, in particular, scaling by multiples of the unit matrix, which do not change the Lipschitz properties of $D$. Therefore, we can assume $\diam(D)=1$ without loss of generality.

Now we set $\rho:=\frac{\LipDelta(D)}{4}$, $t\in \Big(0,\frac{\rho}{r^2}\Big]$, and $0\le |\bm{h}|\le t$. Further, by $(D_j)_{j=1,\dots,\LipCov(D)}$ and $(\mathcal{C}_j)_{j=1,\dots,\LipCov(D)}$ we denote the set of special Lipschitz domains from \cref{Definition of Lipschitz domain} and corresponding cones from \cref{Remark Lipschitz domain}\ref{Remark Lipschitz domain - interior cone condition}, chosen as in the corresponding proof above, with \mbox{$\mathcal{C}_j^t:=\{\bm{x}\in \R^d\mid \frac{\bm{x}}{t}\in \mathcal{C}_j\}$} for $j=1,\dots, \LipCov(D)$. Next, define \mbox{$D^j:=\left(D\cap U_j\right)_{r,\bm{h}}=\left(D_j\cap U_j\right)_{r,\bm{h}}$}, due to \cref{Definition of Lipschitz domain}\ref{Definition of Lipschitz domain - special Lipschitz domain part}) (see Definition \ref{def:diff_op} regarding the notation \mbox{$(D\cap U_j)_{r,\bm{h}}$} etc.), \mbox{$j=1,\dots, \LipCov(D)$}, and \mbox{$D_{\LipCov(D)+1}:=D\setminus D_{\frac{\LipDelta(D)}{2}}$} with 
\begin{align*}
D_{\frac{\LipDelta(D)}{2}}:=\left\{\bm{z}\in D\mid d(\bm{z},\partial D)<\frac{\LipDelta(D)}{2}\right\},
\end{align*}
$\mathcal{C}_{\LipCov(D)+1}:=\mathcal{C}_1$, and $\mathcal{C}^t_{\LipCov(D)+1}$ consistent to the above cases. \\

\begin{minipage}{0.65\textwidth}
We will first show that $D_{r,\bm{h}}\subset \bigcup\limits_{j=1}^{\LipCov(D)+1}D^j$. Since, $D^{\LipCov(D)+1}=D\setminus D_{\frac{\LipDelta(D)}{2}}\subset D_{r,\bm{h}}$, due to \mbox{$|i \bm{h}|\le rt\le \frac{\LipDelta(D)}{4}<\frac{\LipDelta(D)}{2}$} for any $i=1,\dots, r$, it is sufficient to show $D_{r,\bm{h}}\cap D_{\frac{\LipDelta(D)}{2}}\subset\bigcup\limits_{j=1}^{\LipCov(D)}D^j$. \\
This can be seen as follows. Let $\bm{x}\in D_{r,\bm{h}}\cap D_{\frac{\LipDelta(D)}{2}}$, then $d(\bm{x},\partial D)<\frac{\LipDelta(D)}{2}<\LipDelta(D)$ and $\bm{x}+r\bm{h}\in D$. As above, for any $i=1,\dots, r$, $|i\bm{h}|<\frac{\LipDelta(D)}{2}$ and therefore by the triangle inequality $d(\bm{x}+i\bm{h},\partial D)<\LipDelta(D)$. Now \cref{Definition of Lipschitz domain}\ref{Definition of Lipschitz domain - Delta part}) implies the existence of $j\in \{1,\dots, \LipCov(D)\}$ with $\bm{x}$, $\bm{x}+i\bm{h}\in U_j$. Thus, $\bm{x}\in D^j \subset \bigcup\limits_{j=1}^{\LipCov(D)+1} D^j$. 
\end{minipage}\hfill \begin{minipage}{0.35\textwidth}
	\begin{center}
		\includegraphics[height=4cm]{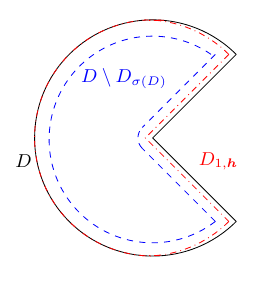}
        \captionof{figure}{Illustration for $r=1$}
	\end{center}
\end{minipage}\\

Further, we proceed to show that $D^j \subset D_{r, \bm{h}+r\bm{s}}$, if $\bm{s} \in \mathcal{C}_j^t$ for any $j=1,\dots, \LipCov(D)+1$. First, consider $j\le \LipCov(D)$. Then $\bm{x}+i\bm{h}\in D^j$ for any $i=1,\dots, r $, and therefore $\bm{x}+i\bm{h}+\mathcal{C}_j\subset D$ due to our choice of the cone $\mathcal{C}_j$. Since, $\bm{s}\in \mathcal{C}_j^t$, $\frac{\bm{s}}{t}\in\mathcal{C}_j$, and $\lambda\,\frac{\bm{s}}{t}\in\mathcal{C}_j$ for any $\lambda\in [0,1]$. Choosing $\lambda:=tm$ for $m\in\{0,\dots, r^2\}$, then gives \mbox{$0\le\lambda=tm\le \frac{\rho}{r^2}\,r^2=\rho<\LipDelta(D)\le \diam(D)=1$}, which means $m\bm{s}\in\mathcal{C}_j$. Together, this shows that $x\in D_{r, \bm{h}+r\bm{s}}$ for such $\bm{s}$. Similarly, for $j=\LipCov(D)+1$: Then $\bm{x}\in D^j$ fulfills $d(\bm{x},\partial D)>\frac{\LipDelta(D)}{2}$ and $x+i\bm{h}+m\bm{s}\in D$ now follows due to $|i\bm{h}+m\bm{s}|\le |i\bm{h}|+|m\bm{s}|\le \frac{\LipDelta(D)}{4} + \frac{\LipDelta(D)}{4} = \frac{\LipDelta(D)}{2}$. Here $|i\bm{h}|\le \frac{\LipDelta(D)}{4}$ can be justified identically to the case $j\le \LipCov(D)$, and $m\bm{s}= tm\,\frac{\bm s}{t}\in \mathcal{C}_{\LipCov(D)+1}$, where $\frac{\bm s}{t}\in \mathcal{C}_{\LipCov(D)+1}$, can be shown as above as well. The latter implies the postulated $|m\bm{s}|\le \diam(\mathcal{C}_j)\le \frac{\LipDelta(D)}{4}$.

Now fix $j\in \{1,\dots, \LipCov(D)+1\}$ and $\bm{s} \in \mathcal{C}_j^t$. Due to the shown result $D^j \subset D_{r, \bm{h}+r\bm{s}}$, the formula
\begin{align*}
    \Delta_{\bm{h},\bm{x}}^r f(\tau,\bm{z})=\sum\limits_{l=1}^{r}(-1)^{l}\binom{r}{l}\left(\Delta^r_{l\bm{s},\bm{x}}f(\tau, \bm{z}+l\bm{h}) - \Delta^r_{\bm{h}+l\bm{s},\bm{x}}f(\tau, \bm{z})\right),
\end{align*}
which one can easily derive algebraically as in the proof of \cite[Lem.~4.1]{GM14}, is well-defined for any $\tau\in I$ and $\bm{z}\in D^j$. Thus, we can estimate
\begin{align*}
    \|\Delta_{\bm{h},\bm{x}}^r f\|^p_{L_p(I\times D^j)}&\lesssim_{\,p,r}\sum\limits^r_{l=1}\binom{r}{l}\left(\|\Delta^r_{l\bm{s},\bm{x}}f(\cdot, \cdot+l\bm{h})\|^p_{L_p(I\times D_{r, \bm{h}+l\bm{s}})}+\|\Delta^r_{\bm{h}+l\bm{s},\bm{x}}f\|^p_{L_p(I\times D_{r, \bm{h}+l\bm{s}})}\right)
    \\&\le \sum\limits^r_{l=1}\binom{r}{l}\left(\|\Delta^r_{l\bm{s},\bm{x}}f\|^p_{L_p(I\times D_{r, l\bm{s}})}+\|\Delta^r_{\bm{h}+l\bm{s},\bm{x}}f\|^p_{L_p(I\times D_{r, \bm{h}+l\bm{s}})}\right).
\end{align*}
Averaging this over all $s\in \mathcal{C}_j^t$, yields
\begin{align*}
     \|\Delta_{\bm{h},\bm{x}}^r f\|^p_{L_p(I\times D^j)}&\lesssim_{\,p,r}\frac{1}{|\mathcal{C}_j^t|}\sum\limits^r_{l=1}\binom{r}{l} \left(\int\limits_{\mathcal{C}_j^t}\|\Delta^r_{l\bm{s},\bm{x}}f\|^p_{L_p(I\times D_{r, l\bm{s}})}\,d\bm{s}+\int\limits_{\mathcal{C}_j^t}\|\Delta^r_{\bm{h}+l\bm{s},\bm{x}}f\|^p_{L_p(I\times D_{r, \bm{h}+l\bm{s}})}\,d\bm{s}\right)
     \\&\lesssim_{\,\Lip(D),\LipDelta(D)} \frac{1}{t^d}\sum\limits^r_{l=1}\binom{r}{l} \left(\int\limits_{\mathcal{C}_j^{lt}}\|\Delta^r_{\bm{u},\bm{x}}f\|^p_{L_p(I\times D_{r,\bm{u}})}\,d\bm{u}+\int\limits_{\mathcal{C}_j^{lt}+\bm{h}}\|\Delta^r_{\bm{v},\bm{x}}f\|^p_{L_p(I\times D_{r, \bm{v}})}\,d\bm{v}\right),
\end{align*}
    where we have used $|\mathcal{C}_j^t|=|\mathcal{C}_j|\,t^d \sim_{\,\Lip(D),\LipDelta(D)}t^d$ as well as the substitutions $\bm{u}:=l\bm{s}$ and $\bm{v}:=\bm{h}+l\bm{s}$. Since $\diam(\mathcal{C})\le\diam(D)=1$, we have $\mathcal{C}_j\subset [-1,1]^d$. This shows $\mathcal{C}_j^{lt}\subset [-lt,lt]^d\subset [-rt,rt]^d $ and, due to $|\bm{h}|\le t$, $\mathcal{C}_j^{lt}+\bm{h}\subset [-(r+1)t,(r+1)t]^d$. This allows us to further estimate
\begin{align*}
     \|\Delta_{\bm{h},\bm{x}}^r f\|^p_{L_p(I\times D^j)}&\lesssim_{\,p,r,\Lip(D),\LipDelta(D)} \frac{1}{t^d} \int\limits_{[-(r+1)t,(r+1)t]^d} \|\Delta^r_{\bm{u},\bm{x}}f\|^p_{L_p(I\times D_{r,\bm{u}})}d\bm{u}\lesssim_{\,d,r}w_r(f, I\times D, (r+1)t)_p^p.
\end{align*}
    Summing this for any $j=1,\dots, \LipCov(D)+1$, going over to the supremum over all $|\bm{h}|\le t$ and taking the $\frac{1}{p}$-th power gives $\omega_r(f, I\times D, t)_p\lesssim_{\,d,p,r,\LipProp(D)} w_r(f, I\times D, (r+1)t)_p$. Together with the scaling result $\omega_r(f, I\times D, (r+1)t)_p\lesssim_{\,p,r}\omega_r(f, I\times D, t)_p$ from \cref{Rem:Moduli of smoothness} and the observation $\delta:=(r+1)t\in \Big(0, \frac{r+1}{r^2}\,\frac{\LipDelta(D)}{4}\Big]\supset\Big(0, \frac{\LipDelta(D)}{4r}\Big]$, the assertion follows (since the case $\delta=0$ is trivial). 
\end{proof}

\begin{proof}[Proof of \cref{Lemma moduli of smoothness = 0 => anisotropic polynomials}.]

First, we will show assertion \ref{Lemma moduli of smoothness = 0 => anisotropic polynomials - 1}. The scaling estimates for the moduli of smoothness from \cref{Rem:Moduli of smoothness} and the monotonicity immediately imply $\omega_{r_1,t}(f, I\times D, |I|)_p=0=\omega_{r_2,\bm{x}}(f, I\times D, \diam(D))_p$. For $\tau\in I$, set $f_\tau:D\rightarrow \R$ via $f_\tau(\bm{z}):=f(\tau,\bm{z})$. Now Fubini leads to
	\begin{align*}
			\infty>\|f\|_{L_p(I\times D)}^p= \int\limits_I \|f_\tau\|_{L_p( D)}^p\, d\tau
	\end{align*}	 
	for $p<\infty$. Therefore, $f_\tau\in L_p(D)$ for almost every $\tau\in I$. Similarly, $f_\tau\in L_{\infty}(D)$ almost surely if $f\in L_\infty(I\times D)$. Further, $\omega_{r_2,\bm{x}}(f, I\times D, \diam(D))_p=0$ yields $\|\Delta^{r_2}_{\bm{h},\bm{x}}f\|_{L_p(I\times D)} = 0$ for any $\bm{h}\in \R^d$, which itself implies 
	\begin{align*}
		0 = \Delta_{\bm{h},\bm{x}}^{r_2}f(\tau,\bm{z}) = f(\tau,\bm{z}+\bm{h}) - f(\tau,\bm{z}) = f_\tau(\bm{z}+\bm{h}) - f_\tau(\bm{z}) = \Delta_{\bm{h}}^{r_2}f_\tau(\bm{z}).
	\end{align*}
	for almost every $(\tau,\bm{z})\in I\times D_{1,\bm{h}}$ and every $\bm{h}\in \R^d$. Thus, $\omega_{r_2}(f_\tau, D, \diam(D))_p = 0$ for almost every $\tau\in I$. According to \cite[Lem.~2.13]{DL04}, this means $f_\tau\in \Pi_{\bm{x}}^{r_2}(D)$ almost surely. This in turn leads to $c_\alpha\in L_p(I)$ for every $|\alpha|<r_2$ such that
\begin{align}\label{Lemma Glattheitsmoduli null => anisotropes Polynom - Beweis Schritt 1}
	f(\tau,\bm{z})=f_\tau(\bm{z})=\sum\limits_{|\alpha|<r_2} c_\alpha(\tau) \bm{z}^\alpha
\end{align}
for almost every $(\tau,\bm{z})\in I\times D$. Since $\omega_{r_1,t}(f, I\times D, |I|)_p=0$, we can similarly show that
\begin{align*}
	0=\Delta_{h,t}^{r_1}f(\tau,\bm{z})=\Delta_{h}^{r_1} f_{\bm{z}}(\tau)=\sum\limits_{|\alpha|<r_2} \left(\Delta_{h}^{r_1}c_\alpha(\tau)\right) \bm{z}^\alpha
\end{align*}
holds true almost surely for every $h\in \R$. Here, we have used the notation $f_{\bm{z}}(\tau):=f(\tau,\bm{z})$, which also yields $f_{\bm{z}}\in L_p(I)$ almost surely. Since the monomials $(\bm{z}^\alpha)_{|\alpha|<r_2}$ are linearly independent in $\Pi_{\bm{x}}^{r_2}(D)$, it follows that $\Delta_{h}^{r_1}c_\alpha(\tau)=0$ for almost every $\tau\in I$, every $|\alpha|<r_2$, and every $h\in \R$, i.e., $\omega_{r_1}(c_\alpha, I, |I|)_p = 0$. Again \cite[Lem.~2.13]{DL04} implies $c_\alpha\in \Pi_t^{r_1}(I)$. Together with \eqref{Lemma Glattheitsmoduli null => anisotropes Polynom - Beweis Schritt 1} this shows the assertion.

Now we can turn to \ref{Lemma moduli of smoothness = 0 => anisotropic polynomials - 2}. With the notation from above, we know that $f_\tau\in \Pi^{r_2}_{\bm{x}}(D)$ for almost every $\tau\in I$ and $f_{\bm{z}}\in \Pi^{r_1}_{t}(I)$ for almost every $\bm{z}\in D$. Therefore, $\Delta_h^{r_1}f_\tau= 0$ and $\Delta_{\bm{h}}^{r_2}f_{\bm{z}}= 0$ almost surely for $h\in \R$ and $\bm{h}\in \R^d$, due to \cite[Thm.~1.26]{Dek22}. This implies $\Delta_{h,t}^{r_1}f=\Delta_{\bm{h},\bm{x}}^{r_2}f=0$ almost surely for such $h$ and $\bm{h}$, respectively. But this directly implies the assertion.
\end{proof}

\paragraph{Acknowledgements.} The first author was partially supported by Agencia Nacional de Promoci\'on Cient\'ifica y Tecnol\'ogica through grant PICT-2020-SERIE A-03820, and by Universidad Nacional del Litoral through grants CAI+D-2020 50620190100136LI and CAI+D-2024 85520240100018LI. Further, part of this work was developed during the third author's visit to Santa Fe in Argentina which was supported by a fellowship for doctoral candidates of the German Academic Exchange Service (DAAD). These grants were gratefully appreciated. Further, we would like to thank an anonymous referee for his valuable feedback, in particular, regarding the fact that our proof of \cref{thm:Jackson} was incomplete in the initial version of this article.

\bibliographystyle{alpha}

\end{document}